\documentclass[reqno]{amsart} 
\usepackage{amsfonts,dsfont}
\usepackage{mathtools}
\usepackage{amssymb}
\usepackage{fullpage}
\usepackage{color}
\usepackage{graphicx,scalerel}
\usepackage{nccmath} 
\usepackage[shortlabels]{enumitem} 
\usepackage[hidelinks]{hyperref}
\usepackage[all]{xy} 
\input xy
\usepackage{framed}

\newtheorem{theorem}{Theorem}
\newtheorem{corollary}[theorem]{Corollary}
\newtheorem{lemma}[theorem]{Lemma}
\newtheorem{proposition}[theorem]{Proposition}

\theoremstyle{definition}
\newtheorem{remark}[theorem]{Remark}
\newtheorem{remarks}[theorem]{Remarks}
\newtheorem{question}[theorem]{Open Problem}
\newtheorem{pgraph}[theorem]{{\!\!}}
\numberwithin{equation}{section}
\numberwithin{theorem}{subsection}

\newcommand{\la}{\langle}
\newcommand{\ra}{\rangle}
\newcommand{\vp}{\varphi}
\newcommand{\vr}{\varrho}
\newcommand{\pma}{\begin{pmatrix}}
\newcommand{\epma}{\end{pmatrix}}
\newcommand{\equi}[1]{\underset{#1}{\equiv}}
\newcommand{\stirlingtwo}[2]{\genfrac{\lbrace}{\rbrace}{0pt}{}{#1}{#2}}
\newcommand{\bs}{\boldsymbol}

\newcommand{\enum}{\begin{enumerate}[{\rm (1)},leftmargin=*,itemsep=1ex]}
\newcommand{\enumi}{\begin{enumerate}[{\rm i)},leftmargin=*,itemsep=1ex]}
\newcommand{\enuma}{\begin{enumerate}[{\rm a)},leftmargin=*,itemsep=1ex]}
\newcommand{\eenum}{\end{enumerate}}

\newcommand{\Co}{\mathds{C}} 

\newcommand\sbullet[1][.5]{\mathbin{\ThisStyle{\vcenter{\hbox{%
  \scalebox{#1}{$\SavedStyle\bullet$}}}}}} 

\DeclareMathOperator{\I}{Id}
\DeclareMathOperator{\V}{var}

\DeclareMathOperator{\FL}{F\langle X|L\rangle}
\DeclareMathOperator{\FU}{F\langle X|U\rangle}
\DeclareMathOperator{\IU}{Id^U\langle X\rangle}
\DeclareMathOperator{\FLU}{F^{L,U}\langle X\rangle}
\DeclareMathOperator{\ILU}{Id^{L,U}\langle X\rangle}

\DeclareMathOperator{\Prim}{Prim}
\DeclareMathOperator{\Max}{Max}
\DeclareMathOperator{\Ann}{Ann}
\DeclareMathOperator{\tr}{tr}
\DeclareMathOperator{\spn}{span}

\DeclareMathOperator{\ad}{ad}
\DeclareMathOperator{\Ad}{Ad}
\DeclareMathOperator{\LM}{LM}
\DeclareMathOperator{\chr}{char}
\DeclareMathOperator{\Der}{Der}
\DeclareMathOperator{\op}{op}

\DeclareMathOperator{\md}{mod}
\DeclareMathOperator{\Mat}{M}
\DeclareMathOperator{\Hom}{Hom}
\DeclareMathOperator{\E}{End}
\DeclareMathOperator{\End}{End}

\newcommand{\Uop}{U(L)^{\op}}
\DeclareMathOperator{\Sl}{\mathrm{sl}}
\newcommand{\M}{M_k(F)}
\DeclareMathOperator{\SL}{\mathrm{sl}_k(F)}
\DeclareMathOperator{\N}{\mathds{N}}

\DeclareMathOperator{\sgn}{sgn}

\usepackage{xcolor}

\begin{document}

\title{Differential identities of matrix algebras}

\author[Brox]{Jose Brox}
\address{Departamento de Álgebra, Análisis Matemático, Geometría y Topología, Universidad de Valladolid, Palacio de Santa Cruz, 47002, Valladolid, Spain}
\email{josebrox@uva.es}

\author[C. Rizzo]{Carla Rizzo}
\address{Dipartimento di Matematica e Informatica, Università degli Studi di Palermo, via Archirafi 34, 90123, Palermo, Italy 
}
\email{carla.rizzo@unipa.it}



\thanks{This work was partially supported by the Centre for Mathematics of the University of Coimbra - UIDB/00324/2020, funded by the Portuguese Government through FCT/MCTES. Jose Brox was first supported by the Portuguese Government through grant SFRH/BPD/118665/2016 (FCT/Centro 2020/Portugal 2020/ESF), later by a postdoctoral fellowship ``Convocatoria 2021'' funded by Universidad de Valladolid and partially supported by grant PID2022-137283NB-C22 funded by MCIN/AEI/10.13039/501100011033 and ERDF ``A way of making Europe''.
\includegraphics[width=75pt]{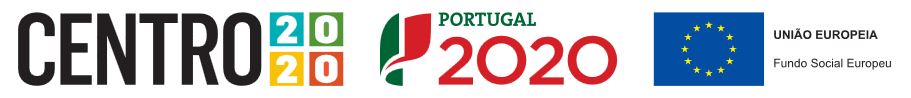}}

\begin{abstract}
We study the differential identities of the algebra $\M$ of $k\times k$ matrices over a field $F$ of characteristic zero when its full Lie algebra of derivations, $L=\Der(\M)$, acts on it. We determine a set of 2 generators of the ideal of differential identities of $\M$ for $k\geq 2$. Moreover, we obtain the exact values of the corresponding differential codimensions and differential cocharacters. Finally  we prove that, unlike the ordinary case, the variety of differential algebras with $L$-action generated by $\M$ has almost polynomial growth for all $k\geq 2$.
\end{abstract}

\maketitle

\section{Introduction}\label{Section:Introduction}

Let $A$ be an associative algebra over a field $F$ of characteristic zero, $F\langle X \rangle$ be the free associative algebra freely generated by an infinite countable set $X$ over $F$, and $\I(A)\subset F\langle X \rangle$ be the $T$-ideal of all polynomial identities of $A$. From a celebrated theorem of Kemer it is known that in characteristic zero every $T$-ideal is finitely generated (see \cite{Kemer1987}). The proof given by Kemer is not constructive, and finding an explicit finite basis of the $T$-ideal of polynomial identities of an algebra is, in general, an extremely hard task. Indeed there is only a handful of nontrivial examples of algebras for which this problem is completely solved. These include the algebra $UT_k(F)$ of upper triangular matrices (\cite{Maltsev1971}), the infinite-dimensional Grassmann algebra $G$ (\cite{KrakowskiRegev1973}), and the tensor product $G\otimes G$ of Grassmann algebras (\cite{Popov1982}). If one adds to the above the full matrix algebra $M_2(F)$ of order $2$ (see \cite{Drensky1981,Razmyslov1973}), one approximately gets the complete list of algebras for which the identities are known. In fact, even the description of the $T$-ideal of $3\times 3$ matrices is still an open problem with no solution in sight.
Since finding the exact form of the polynomial identities satisfied by a given algebra is a goal that seems too hard to achieve for the vast majority of relevant algebras, one is led to the study of identities of algebras with additional structure, such as algebras with a trace, group-graded algebras, algebras with involution, algebras with a Lie algebra action induced by derivations and, more in general, algebras with a Hopf algebra action. Such theories of identities include the theory of ordinary ones as a special case and, overall, their study tends to be less challenging.

In this paper we focus our attention on algebras with derivations, i.e., associative algebras with a Lie algebra action by derivations. If $L$ is such a Lie algebra, then its action can be naturally extended to an action of its universal enveloping algebra $U(L)$, and we say that $A$ is an algebra with derivations from $L$, or an $L$-algebra. In this context the differential identities of $A$ are defined as the polynomials vanishing on $A$ in the variables $x^u:=u(x)$ with $u\in U(L)$, i.e., coming from the free $L$-algebra $\FL$. Notice that the theory of differential identities generalizes the theory of ordinary polynomial identities, as any algebra $A$ can be regarded as an $L$-algebra by letting $L$ act on $A$ trivially,
with $L$ acting on $A$ as the trivial Lie algebra and $U(L)\cong F$. Differential identities were introduced by Kharchenko in \cite{Kharchenko1978} (see also \cite{Kharchenko1979}) and, in later years, relevant work by Gordienko and Kochetov (\cite{GordienkoKochetov2014}) has motivated a growing interest in them. The $T_L$-ideals of differential identities of some important algebras have been determined: in \cite{GiambrunoRizzo2019,Rizzo2021}, Giambruno and Rizzo gave a complete description of the differential identities of the algebra $UT_2(F)$ of $2\times 2$ upper triangular matrices endowed with all possible actions of Lie algebras by derivations; in \cite{DiVincenzoNardozza2021}, Di Vincenzo and Nardozza determined the generators of the $T_L$-ideal of the algebras $UT_k(F)$ under the action of the nonabelian two-dimensional Lie algebra; in \cite{Rizzo2020}, Rizzo studied the differential identities of $G$ with the action of a finite-dimensional abelian Lie algebra of inner derivations. We also refer the interested reader to \cite{MartinoRizzo2022,Nardozza2023} for more results on differential identities of other interesting algebras. 

Since the base field $F$ is of characteristic zero, as it happens in ordinary PI theory, the $T_L$-ideal $\I^L(A)$ is completely determined by its multilinear elements as well. Recall from ordinary PI theory that the codimensions sequence $\{c_n(A)\}_{n\in\N}$ of an algebra $A$ is defined by taking $c_n(A)$ as the dimension of the space $P_n$ of multilinear polynomials of degree $n$ modulo $\I(A)$. The codimensions sequence is also hard to compute, in the sense that, quoting Regev (\cite[p.2]{Regev2012}), in general there is no hope to find a closed formula for $c_n(A)$. Therefore one resorts to studying the growth of the sequence as $n$ tends to infinity. In the late nineties Giambruno and Zaicev (\cite{GiambrunoZaicev1998,GiambrunoZaicev1999}) proved that if $A$ is an algebra satisfying a nontrivial polynomial identity, then the limit $\exp(A):=\lim_{n\to \infty}\sqrt[n]{c_n(A)}$ exists and is always a nonnegative integer called the (ordinary) exponent of $A$. As a consequence, it follows that the codimensions of an algebra are either polynomially bounded or grow exponentially. Given a variety $\textsc V$, its growth is defined as the growth of $c_n(A)$ for any $A$ generating $\textsc V$. Then one says that variety $\textsc V$ has almost polynomial growth if its growth is not polynomially bounded but every proper subvariety of $\textsc V$ has polynomial growth. In the ordinary setting, $G$ and $UT_2(F)$ are the only algebras generating varieties of almost polynomial growth (\cite{Kemer1979}).

Analogous definitions of $P_n^L$, $c_n^L(A)$ and $\exp^L(A)$ can be given in the differential setting. In \cite{Gordienko2013} Gordienko proved that, in case $A$ is finite dimensional, $\exp^L(A)$ indeed exists and is a nonnegative integer called the $L$-exponent of $A$, which allows to likewise define the concept of almost polynomial growth in this case. Moreover, since $U(L)$ is a unital algebra, we can identify $P_n$ with a subspace of $P_n^L$ in a natural way and hence we have $c_n(A)\leq c_n^L(A)$ for all $n\geq 1$, from what it is clear that $\exp(A)\leq \exp^L(A)$. In \cite{GordienkoKochetov2014}, Gordienko and Kochetov proved that in case $L$ is a finite-dimensional semisimple Lie algebra then $\exp(A)=\exp^L(A)$; in \cite{RizzodosSantosVieira2021}, it was shown that if $A$ is finite dimensional and $L$ is any Lie algebra then $\exp^L(A)=1$ if and only if $\exp(A)=1$; and in \cite{Rizzo2023}, the author proved that $\exp^L(A)$ coincides with $\exp(A)$ for any Lie algebra $L$. In case $L$ is finite dimensional and solvable, the only two finite-dimensional $L$-algebras generating $L$-varieties of almost polynomial growth are $UT_2(F)$ with trivial $L$-action, and $UT_2(F)^\varepsilon$ with $L$ the one-dimensional Lie algebra generated by the inner derivation $\varepsilon$ induced by the matrix unit $e_{11}$ (\cite[Corollary 5.5]{Rizzo2023}). The assumption that $L$ is solvable is crucial; in fact, in this paper we present an infinite family of finite-dimensional $L$-algebras of almost polynomial growth for a simple Lie algebra $L$. This points out that the structural properties of the acting Lie algebra deeply affect the growth of the related varieties.

As matrix algebras are of great importance for both mathematics and its applications, the identities satisfied by matrices have been an attractive object of study from the very origins of polynomial identities theory. Concerning matrices with additional structure, to the best of our knowledge, so far the only known results are on graded identities of the matrix algebras $\M$ for cross-product gradings  (\cite{Vasilovsky98,Vasilovsky99} for gradings by $\mathbb{Z}_k$ and $\mathbb{Z}$, \cite{BahturinDrensky2002} for gradings by an arbitrary group), and on the trace identities of the full matrix algebras (see \cite{Procesi1976,Razmyslov1974}).

The main purpose of this paper is to determine the differential identities of the algebra $\M$ of $k\times k$ matrices for $k\geq 2$ over a field $F$ of characteristic zero, when acted by its Lie algebra of all derivations $\Der(\M)$, making all computations explicit along the way. 
To do so, in order to have a finite-dimensional algebra at our disposal, for any $L$-algebra $A$ we call $U$ to the image of the representation of $U(L)$ in $\E_F(A)$ and define two related free $L$-algebras, $\FU$ and $\FLU$ (with their corresponding notions of $U$- and $(L,U)$-polynomials and $T_U$- and $T_{L,U}$-ideals). These algebras allow us to make computations with $U$, and between the two they model the desired properties of $\FL$: roughly speaking, $\FU$ inherits the linear behavior of $\FL$, while $\FLU$ inherits its $L$-action behavior. In Section \ref{Section:GeneralSetting} we conduct a careful analysis of these algebras and their relations, and develop the general setting of the variety of $(L,U)$-algebras (which we define), for which $\FLU$ is the free algebra. In this way we show that we can study differential identities, codimensions, and growth by considering $U$- and $(L,U)$-polynomials and the variety of $(L,U)$-algebras.

In Section \ref{Section:MatrixSetting} we particularize to the case $L:=\Der(\M)\cong\Sl_k(F)$ and, via the representation theory of $L$ (see Theorem \ref{fullenvelopingalgebraTHEOREM}), describe $U$-polynomials of $\FU$ as being composed of variables either of the form $x^{\vp_{ab}}=x^{ab}$ for $a,b$ elements of the standard Cartan-Weyl basis $\mathcal S:=\{h_1,\ldots,h_{k-1},e_{12},\ldots,e_{kk-1}\}$ of $\Sl_k(F)$, or of the form $x^{\vp_{gg}}=x^{gg}$, with $g=I_k$ the identity matrix, with the exponent endomorphisms satisfying $\vp_{ab}\vp_{cd}=\delta_{bc}\vp_{ad}$ and $1=\sum_{a\in\mathcal S}\vp_{aa} + \vp_{gg}\in\E_F(\M)$. It is this partition of unity into orthogonal idempotents what allows us to circumvent the apparition of ordinary PIs in our computations. Moreover, the definitions of the endomorphisms $\vp_{ab}$ allow us to directly translate identities from $\mathcal S\cup\{g\}$ to $U$-identities of $\M$ (see Lemma \ref{multiplicationtablegivesidentitiesLEMMA}),
e.g. $e_{12}e_{31}=0$ implies $x^{e_{12}e_{12}}y^{e_{12}e_{31}}=0$, with the second exponent index carrying the weight of the identities. In Section \ref{Section:DifferentialIdentities} we use this idea together with the linear structure of $\FU$ to show a set of generators of $\I_U(\M)$, in 2 variables and with at most 3 terms, which afterwards we reduce to a minimal set of $4$ generators (in 2 variables and 2 terms) with the aid of the $L$-action of $\FU$, which allows to modify the second index of an exponent; finally we show, through the result from the representation theory of $L$ that we call the primitive element lemma (Lemma \ref{primitiveelementLEMMA}), that $\I_{L,U}(\M)$ is principal.

In order to translate this last result to $\FL$ in an explicit way, if $\phi$ is the homomorphism sending $U(L)$ to $U$, we need to compute some preimages $\phi^{-1}(\vp_{ab})\in U(L)$ of the endomorphisms $\vp_{ab}\in U$, and also some generators of $\ker\phi$, what we also do in Section \ref{Section:MatrixSetting}. For the preimages, we just find expressions formed with polynomials of degree at most $6$ in the elements $e_{ij}\in U(L)$. For the kernel, we recall that the center of $U(L)$ is a polynomial ring in $k-1$ indeterminates $c_i$ which on each irreducible representation $\rho$ of $L$ act as scalars $\lambda^\rho_i$, from which each $c_i-\lambda^\rho_i$ is in the kernel of $\rho$.\footnote{We also compute explicitly the values of the eigenvalues of a standard set of Casimir generators of $\Sl_k(F)$ for the adjoint representation, a result which may be of independent interest.} On the other hand, we know that $e_{12}^3\in\ker\phi$ and that $\phi$ is the direct sum of the trivial and the adjoint representations of $L$. From these facts, the algebraic geometry of $U(L)$ (Gröbner bases, primitive spectrum), and the primitive element lemma, we show that $\ker\phi$ is principal (Theorem \ref{KerPhiGeneratorsTHEOREM}). Then, as $\FLU\cong\FL/\ILU$ with $\ILU$ the $T_L$-ideal generated by $x^z$ for $z\in\ker\phi$, we get as our main result, Theorem \ref{differentialidentitiesTHEOREM}, that the differential identities of $\M$ are generated by 1 identity in 1 variable (coming from $\ker\phi$ and depending on $k$) and 1 identity in 2 variables (coming from $\I_{L,U}(\M)$ and not depending on $k$ except for $k=2$).

In addition, in Section \ref{Section:MatrixSetting} we also show a special kind of symmetry that holds for the $U$-identities of any $(L,U)$-algebra $A$ and that we use profusely thereafter, roughly speaking, that changes in the first exponent index leave $T_U(A)$ invariant. More concretely, consider the space $P^U_{\mathcal I,\mathcal J,(a_1,\ldots,a_{n-r})}$ with $(\mathcal I,\mathcal J)$ a partition of $\{1,\ldots,n\}$ and $|\mathcal I|=r$, of those multilinear $U$-polynomials in which variable $x_i$ always appears paired with first exponent $g$ for $i\in\mathcal I$ and variable $x_j$ always appears paired with first exponent $a_j\in\mathcal S$ for $j\in\mathcal J$. Then $P^U_{\mathcal I,\mathcal J,(a_1,\ldots,a_{n-r})}$ is linearly isomorphic to $P^U_{\mathcal I',\mathcal J',(a'_1,\ldots,a'_{n-r'})}$ if and only if $r=r'$, and $U$-identities of $A$ map to $U$-identities. Moreover, defining an action of $S_r\times S_{n-r}$ by permutations of variables together with their first exponents, the linear isomorphisms are in fact isomorphisms of $S_r\times S_{n-r}$-modules. Since then all of them are isomorphic to $P^U_{r,n-r}:=P^U_{\{1,\ldots,r\},\{r+1,\ldots,n-r\},(a,\ldots,a)}$ for fixed $a$, we can restrict to the study of these $S_r\times S_{n-r}$-modules for each $0\leq r\leq n$.
With these ideas at hand, we show a combinatorial formula for the $U$-codimensions (Formula \eqref{formula(C)}), arising from $P^U_{r,n-r}(A)$, that is used in Section \ref{Section:DifferentialIdentities} together with the $U$-identities of $\M$ to find a closed formula for $c^L_n(\M)$ (see Theorem \ref{differentialidentitiesTHEOREM}). In particular, the associated generating function is rational; in contrast, when $k\geq3$ is odd, the generating function of (ordinary) $c_n(\M)$ is not algebraic (\cite[Theorem 12.4]{Regev2012}). As an aside, this proof also constitutes a simple way of showing that $\exp(\M)=k^2$ (as $\exp(\M)=\exp^L(\M)=k^2$), which was originally shown by Regev by resorting to the asymptotics of trace identities (\cite{Regev1984}), and can also be proved by Wedderburn's decomposition (\cite[Theorem 6.6.1]{GiambrunoZaicevbook}).\\
Now let $P_{(n;r)}^U(A)$ be the direct sum of all $P_{\mathcal I,\mathcal J,(a_1,\ldots,a_{n-r})}(A)$ such that $|\mathcal I|=r$; it is an $S_r\times S_{n-r}$-module whose character $\chi_{(n;r)}^U(A)$, which we call the $(n,r)$th $U$-cocharacter of $A$, is a multiple of that of $P^U_{r,n-r}(A)$ (Formula \eqref{formula(chi)}). In Section \ref{Section:Ucocharacter} we show, by a counting argument, that $\chi_{(n;r)}^U(\M)$ is a multiple of the irreducible $S_r\times S_{n-r}$-cocharacter $\chi_{(r)}\otimes\chi_{(n-r)}$ for each $0\leq r\leq n$ (Theorem \ref{UcocharacterTHEOREM}).

 Lastly, in Section \ref{Section:AlmostPolynomialGrowth} we prove a result which is, in our view, one of the most interesting and unexpected PI results of this paper: unlike the ordinary case, the variety $\V^L(M_k(F))$ of differential algebras with $L$-action generated by $M_k(F)$ has almost polynomial growth for all $k\geq 2$, i.e., $\V^L(M_k(F))$ has exponential growth but any of its proper $L$-subvarieties has polynomial growth (Theorem \ref{AlmostPolynomialGrowthTHEOREM}). To show it, we prove that if an $L$-subvariety $\textsc V$ satisfies any $L$-identity not belonging to $T_L(\M)$, then it must satisfy all $U$-identities of the form $x_1^{a_1a_2}\cdots x_t^{a_ta_{2t}}$ for $a_i\in\mathcal S$ for some $t$, implying that $c_{r,n-r}^L(\textsc V) = c_{r,n-r}^U(\textsc V)=0$ whenever $n-r\geq N$ for some $N$.

\section{General setting}
\label{Section:GeneralSetting}

\subsection{Preliminaries}
\mbox{}

\medskip

Throughout this paper, $F$ will denote a field of characteristic zero, $A$ an associative algebra, and $(L,[\cdot\, ,\cdot]_L)$ a Lie algebra. All algebras and vector spaces have $F$ as their underlying field. Although we work with varieties of nonunital associative algebras, all results can be easily adapted to unital associative algebras as well.  All notations, once introduced, will maintain their meanings in the ensuing sections of the paper.

\begin{pgraph}\textbf{Associative algebras.}\label{associativealgebras}
 Given a set $S\subseteq A$, by $\la S\ra$ we denote the ideal generated by $S$. $A$ is \emph{split semisimple} (over $F$) if it is a direct sum of matrix algebras over $F$. Given a unital associative algebra $U$ with product $\cdot$, the \emph{opposite algebra} $U^{\op}$ is the underlying vector space of $U$ endowed with the \emph{opposite product} $a\cdot^{\op} b:=b\cdot a$ for $a,b\in U$; $U^{\op}$ is antiisomorphic to $U$ as unital associative algebras (with $(U^{\op})^{\op}=U$) through the map $^{\op}:U\to U^{\op}$ such that $a^{\op}:=a$; in particular, any subset of $U$ is mapped to itself. If $\phi:U_1\to U_2$ is a homomorphism of unital associative algebras, then $\phi^{\op}:U_1^{\op}\to U_2^{\op}$ defined by $\phi^{\op}(a):=\phi(a)$ is a homomorphism of unital associative algebras.
Given a vector space $V$, a \emph{left (resp. right) (algebra) $U$-action} of a unital associative algebra $U$ on $V$ is a map $\cdot:U\times V\rightarrow V$ (resp. $\cdot:V\times U\rightarrow V$) such that $1\cdot x=x$, $a\cdot(\lambda x+y)=\lambda(a\cdot x)+a\cdot y$, $(\lambda a+b)\cdot x = \lambda(a\cdot x)+b\cdot x$, and $(ab)\cdot x = a\cdot(b\cdot x)$ (resp. $x\cdot 1=x$, $(\lambda x+y)\cdot a=\lambda(x\cdot a)+y\cdot a$, $x\cdot(\lambda a+b) = \lambda(x\cdot a)+x\cdot b$, and $x\cdot (ab) = (x\cdot a)\cdot b$) for $a,b\in U$, $x,y\in V$ and $\lambda\in F$. Let $\End_F(V)$ be the algebra of the endomorphisms of $V$ acting on the left of $V$. Then a left (resp. right) $U$-action on $V$ produces a \emph{left (resp. right) representation} of $U$, i.e., a homomorphism of unital algebras $\phi:U\to\End_F(V)$ (resp. $\phi:U^{\op}\to\End_F(V)$) and vice versa. Any left action $\cdot$ of $U$ on $V$ has an associated right action $\cdot^{\op}$ of $U^{\op}$ on $V$ given by $x\cdot^{\op}a:=a\cdot x$ for $a\in U^{\op}$, $x\in V$, and vice versa (with $(\cdot^{\op})^{\op}=\cdot$). If there is a (left, right) $U$-action on $A$ we say that $A$ is a \emph{(left, right) $U$-algebra} (for this action).
Throughout this paper we define endomorphisms as acting on the left, but we use exponential notation to denote their actions: hence we see any left $U$-algebra as a right $U^{\op}$-algebra (notice that the associated representation $\phi$ is the same), with exponents living in $U^{\op}$; in addition, we denote the opposite products appearing in exponents just by juxtaposition. Moreover, by abuse of notation we may also denote $\phi^{\op}:U^{\op}\to\End_F(A)^{\op}$ by $\phi$. For example, if $\phi:U\to\End_F(A)$ with $\phi(u_i)=\phi_i$, $u_i\in U$ for $i=1,2$ and associated left action denoted by $\sbullet$, then we write
\[(a^{u_1})^{u_2} = a^{u_1u_2} = a^{\phi(u_1)\phi(u_2)} = a^{\phi_1\phi_2} =\phi_2(\phi_1(a)) = u_2\sbullet(u_1\sbullet a) = (u_2u_1)\sbullet a = (u_1\cdot^{\op} u_2)\sbullet a\]
for $a\in A$ and, in the exponents, $u_1,u_2\in U^{\op}$, $\phi_1,\phi_2\in\End_F(A)^{\op}$. Since for any set $S\subseteq U$ we have $S^{\op}=S$ inside $U^{\op}$, if no confusion may arise, when picking exponents we may write $s\in S$ instead of $s\in S^{\op}$.
\end{pgraph}
\begin{pgraph}\textbf{Lie algebras.}
Given $L$, the \emph{opposite Lie algebra} $L^{\op}$ is the underlying vector space of $L$ endowed with the \emph{opposite product} $[a,b]_{L^{\op}}:=-[a,b]_L$ ($L^{\op}$ is isomorphic to $L$). If $\vp:L\rightarrow M$ is a homomorphism of Lie algebras, then $\vp^{\op}:L^{\op}\rightarrow M^{\op}$ defined by $\vp^{\op}(a):=\vp(a)$ is a homomorphism of Lie algebras. 
The underlying vector space of $A$ endowed with the \emph{commutator product} $[a,b]:=ab-ba$ for all $a,b\in A$ is a Lie algebra, denoted by $A^-$. We have $(A^{\op})^- = (A^-)^{\op}$. A linear endomorphism $\delta:A \to A$ is a \emph{derivation} of $A$ if it satisfies  $(ab)^\delta=a^\delta b + a b^\delta$ for all $a,b\in A$. If $a\in A$, the endomorphism $\ad_a:A\to A$ defined by $\ad_a(b):=[a,b]$ for all $b\in A$ is a derivation of $A$, called the \emph{inner derivation} induced by $a$. For ease of reading, given an element in $A$ denoted by a lowercase letter, at times we denote the inner derivation this element induces by the corresponding uppercase letter, e.g. $E:=\ad_e\in\End_F(A)$ for $e\in A$. The vector space of all derivations of $A$ endowed with the commutator product is a Lie algebra denoted by $\Der(A)\subseteq\E_F(A)^-$, with the subspace $\ad (A)\subseteq\Der(A)$ of inner derivations of $A$ being a Lie ideal of $\Der(A)$.
 A semisimple Lie algebra is \emph{split} (over $F$) if it has a Cartan subalgebra $H$ such that the eigenvalues of $\ad_h$ lie in $F$ for all $h\in H$, called a \emph{splitting} Cartan subalgebra; hence a split-semisimple Lie algebra has a root system (\cite[Chapter IV]{JacobsonLie}). $L$ is \emph{split simple} if it is simple and split semisimple.
\end{pgraph}
\begin{pgraph}\textbf{$L$-modules.} 
A \emph{(left) representation} of $L$ on the vector space $V$ is a Lie algebra homomorphism $\rho:L\rightarrow\End_F(V)^-$; we also say that that $V$ has a \emph{(left) $L$-action} given by $d\cdot x:=\rho(d)x$ for $d\in L$, $x\in V$, and that $V$ (together with $\rho$) is a \emph{(left) $L$-module}. We similarly define the same notions for the right side, with any left $L$-action $\cdot$ giving rise to a right $L^{\op}$-action $\cdot^{\op}$ by the rule $x\cdot^{\op} d:=d\cdot x$; accordingly, we will write left $L$-actions as right $L^{\op}$-actions with exponential notation.
The \emph{trivial representation} on $V$ is given by $\rho_0:=0$. The \emph{adjoint representation} of $L$ is given by $\ad$ on $L$. A representation $\rho$ on $V$ is \emph{irreducible} if it has no proper nontrivial subrepresentation $\rho|_W$ on $W\subseteq V$. Weyl's theorem states that, when $L$ is finite dimensional semisimple, every finite-dimensional $L$-module is \emph{completely reducible}, i.e., a direct sum of irreducible $L$-modules (\cite[III.7 Theorem 8]{JacobsonLie}). 
For the rest of this paragraph let $L$ be a finite-dimensional split-semisimple Lie algebra (over a field of characteristic $0$), with a fixed splitting Cartan subalgebra $H\subseteq L$ and a fixed set of positive roots; then $L$ has a \emph{triangular decomposition} $L=N_-\oplus H\oplus N_+$ where $N_{-,+}$ are the linear spans of the negative and positive root spaces, respectively. Any finite-dimensional irreducible $L$-module is \emph{absolutely irreducible}, i.e., irreducible for any extension of scalars of the base field (\cite[p.223]{JacobsonLie}).
A \emph{weight} of $L$ is an algebra homomorphism $\lambda\in\Hom(H,F)$. A vector $v$ of the $L$-module $V$ is a \emph{weight vector} of weight $\lambda$ if $\rho(h)v=\lambda(h)v$ for all $h\in H$, a \emph{highest weight vector} if in addition $\rho(N_+)v=0$, in which case $\lambda$ is a \emph{highest weight} of $V$. An $L$-module is a \emph{(cyclic) highest weight module} if it is generated by a single highest weight vector. Given a weight $\lambda$ of $L$ we can build a universal highest weight $L$-module with $\lambda$ as highest weight, the \emph{Verma module} $W_\lambda$, such that any highest weight $L$-module with highest weight $\lambda$ is a quotient of $W_\lambda$ (\cite[Chapter VII.2]{JacobsonLie}).

 Given a representation $\rho:L\to\E_F(V)^-$, the \emph{enveloping (associative) algebra} of $\rho$ is the unital associative subalgebra of $\E_F(V)$ generated by $\rho(L)$, denoted here by $[\rho(L)]$. The following result is known (see e.g. \cite[Theorem III.5.10]{JacobsonLie} for a close result); we include a proof here since it is a key tool of our paper.

\begin{theorem}[\textbf{Full matrix algebras as enveloping algebras}]\label{fullenvelopingalgebraTHEOREM}
\mbox{}

\noindent Let $L$ be a finite-dimensional split-semisimple Lie algebra over a field of characteristic $0$ and $\rho$ be a finite-dimensional representation of $L$. Then the enveloping associative algebra $[\rho(L)]$ of $\rho$ is split semisimple. Moreover, if $\rho$ is irreducible of dimension $d$, then $[\rho(L)]$ is a matrix algebra of dimension $d^2$.
\end{theorem}
\begin{proof}
Put $\rho:L\to\End_F(V)^-$ with $\dim_F V=d$. Since $\rho$ is finite dimensional, by Weyl's theorem it is completely reducible, so it suffices to assume that $\rho$ is irreducible. Denote $U:=[\rho(L)]$, let $K$ be an algebraic closure of $F$, and consider the following objects produced by extension of scalars:
$L_K:=L\otimes_F K, \, V_K:=V\otimes_F K, \E_F(V)\otimes_F K = \E_K(V_K),$
the representation $\rho_K:=\rho\otimes_F K:L_K\rightarrow\E_K(V_K)^-$, and its enveloping algebra $U_K:=[\rho_K(L_K)] = U\otimes_F K$.
Since $L$ is a finite-dimensional split-semisimple Lie algebra and $\chr(F)=0$, $\rho$ is absolutely irreducible. Therefore $\rho_K$ is irreducible, hence $V_K$ is a finite-dimensional irreducible $U_K$-module (since $U_K$ is generated by $L_K$) which is also faithful (as $U_K\subseteq\E_K(V_K)$); by Jacobson's density theorem (\cite[pp.197,199]{JacobsonAlgebra}), $U_K=\End_D(V_K)$ with $D:=\End_{U_K}(V_K)$ a finite-dimensional division algebra (by Schur's lemma) over the algebraically closed field $K$, which forces $D=K$ and $U_K = \End_K(V_K)\cong M_d(K)$. Now
$U\otimes_F K = U_K =  \E_K(V_K) = \E_F(V)\otimes_F K$
implies $U=\E_F(V)$ (since $U\subseteq\E_F(V)$), so $U\cong M_d(F)$.
\end{proof}
\end{pgraph}
\begin{pgraph}\textbf{Universal enveloping algebras.}
The \emph{universal enveloping algebra} $U(L)$ of $L$ is the quotient of the unital tensor algebra generated by $L$ by the ideal generated by relations $a\otimes b-b\otimes a-[a,b]_L$ for all $a,b\in L$, i.e., for $a,b\in L$ we have $[a,b]_{U(L)}=[a,b]_L$, hence $[a,b]_{\Uop}=-[a,b]_L = [a,b]_{L^{\op}}$.
The universal enveloping algebra satisfies the following universal property: there is a monomorphism $\gamma:L\rightarrow U(L)^-$ of Lie algebras such that for any homomorphism $\vp:L\rightarrow A^-$ of Lie algebras there exists a unique homomorphism $\phi:U(L)\rightarrow A$ of associative algebras such that $\vp=\phi\circ\gamma$. Since its universal property determines $U(L^{\op})$ up to isomorphism, and $U(L)^{\op}$ is easily seen to satisfy said universal property, we have $U(L^{\op})\cong U(L)^{\op}$ as unital associative algebras.
By abuse of notation, we identify $L$ with $\gamma(L)$ inside $U(L)$. Then the Poincaré-Birkhoff-Witt (PBW) theorem asserts that if $\{e_i\}_{i\in I}$ is an ordered basis of $L$, the set $B:=\{e_{i_1}^{k_1}\cdots e_{i_j}^{k_j} \ | \ j\in\N, e_{i_1}<\cdots< e_{i_j}, k_1,\ldots,k_j\in\N^*\}$ is a basis of $U(L)$ (including $1\in B$). In particular, $U(L)=:U^*(L)\oplus F\cdot1$, where $U^*(L)$ is the \emph{nonunital universal enveloping algebra} of $L$ and the \emph{augmentation ideal} of $U(L)$ (which is a maximal ideal). Note that $U(L)$ is always an infinite-dimensional algebra, even if $L$ is finite dimensional. Universal enveloping algebras of finite-dimensional Lie algebras are Noetherian rings (\cite[V.3 Theorem 6(2)]{JacobsonLie}) and the theory of Gröbner bases is available for $U(L)$ (see \cite[Section 6.3]{deGraafLie}); we are interested in a particular application. Let $L$ be finite dimensional, fix an ordered basis $\{e_1,\ldots,e_d\}$ with $e_i<e_j$ if $i<j$, and let $B$ denote the corresponding basis of monomials of $U(L)$ given by the PBW theorem. The \emph{deglex order} extends $<$ to $B$ by the rule $m_1:=e_1^{k_1}\cdots e_d^{k_d}< e_1^{l_1}\cdots e_d^{l_d}=:m_2$ if either $\deg(m_1)<\deg(m_2)$ or $\deg(m_1)=\deg(m_2)$ and the first nonzero entry of $(k_1-l_1,\ldots,k_d-l_d)$ is positive. Given an element $f\in U(L)$, its \emph{leading monomial} $\LM(f)$ with respect to $<$ is the largest monomial of $f$ with nonzero coefficient. If $I$ is an ideal of $U(L)$, its set of leading monomials is $\LM(I):=\{\LM(f) \ | \ f\in I\}$ and its set of \emph{normal words} is $N(I):=\{m\in B \ | \ m\not\in\LM(I)\}$. Then
$U(L)=I\oplus \spn N(I)$ (\cite[p.231]{deGraafLie}).

We can make $U(L)$ an $L$-module by extending the adjoint action of $L$ to $U(L)$, which is the restriction to $L$ of the adjoint action of $U(L)$ on itself (given by $\ad_x(y)=[x,y]$ for $x,y\in U(L)$). By the universal property of $U(L)$, any representation $\rho:L\rightarrow\End_F(V)^-$ of $L$ extends to an associative left $U(L)$-representation $\rho:U(L)\rightarrow\End_F(V)$ (note the abuse of notation). A $U(L)$-representation is irreducible if and only if it is irreducible as an $L$-representation (since $U(L)$ is generated as an algebra by $L$).
If $L$ is finite dimensional split semisimple, the Verma module associated to weight $\lambda$ of $L$ can be constructed as $M_\lambda\cong U(L)/I_\lambda$, with $I_\lambda$ the left ideal of $U(L)$ generated by $\{h-\lambda(h)1 \ | \ h\in H\}\cup N^+$ (\cite[Chapter VII.2]{JacobsonLie}, \cite[Chapter 9]{Hall2015}).

The \emph{center} $Z(U(L))$ of $U(L)$ is the set of elements $c\in U(L)$ such that $[c,U(L)]=0$, which are called \emph{Casimir elements}. For the rest of this paragraph let $L$ be a finite-dimensional semisimple Lie algebra over an algebraically closed field (of characteristic $0$), and denote $Z:=Z(U(L))$. By the Harish-Chandra isomorphism (\cite[Lemma 36]{HarishChandra1951}), $Z$ is isomorphic to the algebra of polynomials in rank$(L)$ indeterminates (see also \cite[Theorem 7.3.8]{DixmierEnvelopingAlgebras}); we call a set of \emph{Casimir generators} to any set of rank$(L)$ algebraically independent elements of $Z$. By Schur's lemma, if $c\in Z$ and $\rho$ is an irreducible representation of $L$ on $V$ then $\rho(c)$ acts as a scalar on $V$. A \emph{central character} of $L$ is an algebra homomorphism $\chi:Z\rightarrow F$. If $\rho$ is an irreducible representation of $L$ of dimension $d$ then $\chi_\rho(c):=\frac1d\tr(\rho(c))$ for $c\in Z$ is the central character \emph{associated to $\rho$}; we call $\chi_\rho(c)$ the \emph{eigenvalue} of $c$ for $\rho$. We clearly have $U(L)\ker\chi_\rho\subseteq\ker\rho$. If $\rho,\rho'$ are two finite-dimensional irreducible representations of $L$ and $\chi_\rho=\chi_{\rho'}$ then $\rho\cong\rho'$ (\cite[Section 2]{Cartier1955}). Note that, for a fixed set of Casimir generators, the maximal ideals of $Z$ are in 1-to-1 correspondence with the central characters of $L$. Now fix $F=\Co$, the field of complex numbers.  We will need the following results from the algebraic geometry of $U(L)$. Let $\Max(A), \Prim(A)$ respectively denote the sets of maximal and primitive ideals of the algebra $A$, and for $I\in\Prim(U(L))$ let $\pi(I):=I\cap Z$ and say that $I$ \emph{lies over} $J$ for an ideal $J$ of $Z$ if $I\in\pi^{-1}(J)$. Then $\pi(I)\in\Max(Z)$, and \emph{Dixmier's theorem} asserts the following: given $M\in\Max(Z)$, the set of primitive ideals of $U(L)$ lying over $M$ is finite and contains a minimal and a maximal element, with $\pi:\Max(U(L))\rightarrow\Max(Z)$ being a bijection (\cite{Dixmier1970,Dixmier1971}, see also \cite[3.2-3.3]{Borho1978} and \cite[Theorem 3.1.3]{Catoiu2000}). \emph{Joseph's theorem on principal series submodules} states that if the ideal $I$ of $U(L)$ satisfies $I\cap Z(L)\in\Max(Z(L))$ then $I$ is the annihilator of the $L$-module $W_\lambda/IW_\lambda$ for some weight $\lambda$ of $L$ (\cite[Theorem 5.1]{Joseph1979}).
We can say more when the ideal is of finite codimension:

\begin{lemma}\label{maximalideal}
Let $L$ be a complex finite-dimensional semisimple Lie algebra and $I\neq U(L)$ be an ideal of finite codimension of $U(L)$ such that $I\cap Z(U(L))\in\Max(Z(U(L)))$. Then $I\in\Max(U(L))$.
\end{lemma}

\begin{proof}
Denote $Z:=Z(U(L))$. Since $I\cap Z\in\Max(Z)$, $I$ is the annihilator of the $U(L)$-module $N:=W_\lambda/IW_\lambda$ for some weight $\lambda$ of $L$ by Joseph's theorem on principal series submodules. Since $W_\lambda\cong U(L)/I_\lambda$ as $U(L)$-modules, we get $N\cong\frac{U(L)/I_\lambda}{(I+I_\lambda)/I_\lambda}\cong U(L)/(I+I_\lambda)$ as $U(L)$-modules, with $I+I_\lambda$ of finite codimension, proving $N$ a finite-dimensional $U(L)$-module. By Weyl's theorem $N$ is completely reducible, $N=\bigoplus_{i=1}^n N_i$ with $N_i$ irreducible and finite dimensional for $1\leq i\leq n$. The primitive ideals $M_i:=\Ann N_i$ are maximal by Jacobson's density theorem (\cite[pp.197,199]{JacobsonAlgebra}), in particular pairwise coprime, so $I=\Ann N=\Ann(\bigoplus_{i=1}^n N_i) = \bigcap_{i=1}^n\Ann N_i =\bigcap_{i=1}^n M_i$ and (by the Chinese remainder theorem) $U(L)/I\cong \prod_{i=1}^n U(L)/M_i$ is a sum of $n\geq1$ matrix algebras ($n>0$ because $I\neq U(L)$). Suppose $n\geq2$; then there are at least two different maximal ideals $M_1,M_2$ of $U(L)$ containing $I$, with $I\cap Z$ maximal in $Z$ forcing $M_1\cap Z = I\cap Z = M_2\cap Z$. But Dixmier's theorem implies the uniqueness of the maximal ideal of $U(L)$ lying over $I\cap Z$, so $M_1=M_2$, a contradiction which forces $n=1$. So $U(L)/I$ is simple and $I$ is maximal.
\end{proof}
\end{pgraph}
\begin{pgraph}\textbf{Hopf algebras.} 
Consider a Hopf algebra $H$ with comultiplication $\Delta:H\rightarrow H\otimes H$ (in particular $\Delta(xy)=\Delta(x)\Delta(y)$ in $H\otimes H$). $H$ is \emph{cocommutative} if $\tau\circ\Delta=\Delta$ for the \emph{twist map} $\tau:H\otimes H\rightarrow H\otimes H$ defined by $\tau(a\otimes b):=b\otimes a$. If $H$ is cocommutative then $H^{\op}$ is also a Hopf algebra with the same comultiplication, counit, and antipode.
We write comultiplications in Sweedler's notation, $\Delta(h)=:\sum h_1\otimes h_2$. For $n\in\N^*$, the \emph{$(n-1)$th iterated comultiplication} is the operation $\Delta_{n-1}:H\rightarrow \overbrace{H\otimes\cdots\otimes H}^n$ given iteratively by \[\Delta_{n-1}(h):=\sum \Delta(h_1)\otimes\cdots\otimes h_{n-1} =: \sum h_1\otimes\cdots\otimes h_n\] (which is well defined by the coassociativity axiom, see \cite[1.4.2]{Montgomery1993}).

Given $H$, we say that $A$ has a right \emph{Hopf (algebra) $H$-action}, or that $A$ is a right \emph{$H$-module algebra}, if there is a right algebra action of $H$ on $A$, $\cdot:A\otimes H\rightarrow A$, such that $(ab)\cdot h=\sum(a\cdot h_1)(b\cdot h_2)$ and $1\cdot h=\varepsilon(h)1$, for all $a,b\in A$ and all $h\in H$, with $\Delta(h)=\sum h_1\otimes h_2$. When $A$ has a right Hopf action, then, for $h\in H$,
\[(a_1\ldots a_n)^h = \sum a_1^{h_1}\cdots a_n^{h_n}, \, \text{ with }\Delta_{n-1}(h)=\sum h_1\otimes\cdots\otimes h_n.\]

The universal enveloping algebra $U(L)$ of $L$ (and hence $\Uop$) becomes a cocommutative Hopf algebra when endowed with comultiplication $\Delta(\delta):=\delta\otimes 1 + 1\otimes \delta$, counit $\varepsilon(\delta):=0$ and antipode $S(\delta):=-\delta$ for $\delta\in L$, and extended to $U(L)$ via the PBW theorem: $\Delta(e_{i_1}^{k_1}\cdots e_{i_j}^{k_j} ):=\Delta(e_{i_1})^{k_1}\cdots\Delta(e_{i_j})^{k_j}$, $S(e_{i_1}^{k_1}\cdots e_{i_j}^{k_j} ):=S(e_{i_1})^{k_1}\cdots S(e_{i_j})^{k_j}$. E.g., for $\delta_1,\delta_2\in L$ we have $\Delta(\delta_1\delta_2)=\delta_1\delta_2\otimes 1 + \delta_1\otimes\delta_2 + \delta_2\otimes\delta_1 + 1\otimes\delta_1\delta_2$ and $\Delta_2(\delta_1\delta_2)= \delta_1\delta_2\otimes 1\otimes 1 + \delta_1\otimes\delta_2\otimes 1 + \delta_1\otimes 1\otimes\delta_2 + \delta_2\otimes\delta_1\otimes 1 + \delta_2\otimes 1\otimes\delta_1 + 1\otimes\delta_1\delta_2\otimes 1 + 1\otimes\delta_1\otimes\delta_2 + 1\otimes\delta_2\otimes\delta_1 + 1\otimes 1\otimes\delta_1\delta_2$.
\end{pgraph}

\subsection{The variety of $L$-algebras}
\mbox{}

\begin{pgraph}\textbf{$L$-algebras.} Given $L$, we say that $A$ is an \emph{$L$-algebra} or that $L$ \emph{acts on $A$ by derivations}, if there exists a homomorphism of Lie algebras $\varphi:L\to\Der(A)$, hence $A$ has a right $L^{\op}$-action satisfying $(a_1a_2)^\delta = a_1^\delta a_2+a_1a_2^\delta$ for $a_1,a_2\in A$, $\delta\in  L$. From now on, when we say that $A$ has an $L$-action it will imply that $A$ is an $L$-algebra for that $L$-action. Note that when $L$ is simple either $\varphi=0$ or $\varphi$ is a monomorphism.
By the universal property of $U(L)$, an $L$-action on $A$ can be uniquely extended to a right Hopf $\Uop$-action (which by abuse of notation we also call an $L$-action), by extending $\varphi$ to the homomorphism of unital associative algebras $\phi:U(L)\to\E_F(A)$ such that $\phi(ab):=\vp(a)\vp(b)$ (recall that we also denote the opposite homomorphism $\phi^{\op}:\Uop\to\E_F(A)^{\op}$ by $\phi$, see \ref{associativealgebras}). In this way, $A$ becomes a right $\Uop$-module algebra (with action in exponential notation and opposite product denoted by juxtaposition).
More explicitly, the $L$-action on $A$ satisfies
\[(a_1a_2\cdots a_n)^\delta =a_1^\delta a_2\cdots a_n + a_1a_2^\delta\cdots a_n + \cdots + a_1a_2\cdots a_n^\delta\]
for $a_1,\ldots,a_n\in A$ and $\delta\in  L$, and, if $\Delta_{n-1}(u)=\sum u_1\otimes\cdots\otimes u_n$ for $u\in \Uop$, then
\[(a_1a_2\cdots a_n)^u=\sum a_1^{u_1}a_2^{u_2}\cdots a_n^{u_n}. \tag{U}\label{formula(U)}\]
Note that when $\varphi(L)=0$, then $\phi(U^*(L))=0$, $\phi(1)=1$, and the Hopf $\Uop$-action is just the linear action of $F$.

For fixed $L$ the class of $L$-algebras is equational, and so it is a variety in the sense of universal algebra (see e.g. \cite{BS2012}). This variety is nontrivial, as it contains $A$. Ideals of $L$-algebras (\emph{$L$-ideals}) are understood to be invariant under the $\Uop$-action, and homomorphisms $f:A\rightarrow B$ between $L$-algebras $A,B$ (\emph{$L$-homomorphisms}) must satisfy $f(a^u)=f(a)^u$ for $a\in A$, $u\in\Uop$. The $L$-ideal generated by elements $a_1,\ldots,a_n\in A$ we denote by $\la a_1,\ldots,a_n\ra_L$. In the next result, which we call the \emph{primitive element lemma}, it is shown that some $L$-ideals are principal; it is a direct generalization of \cite[Lemma 3.1.5 and Corollary 3.1.6]{Catoiu2000} by Catoiu.
\begin{lemma}[\textbf{Primitive element}]\label{primitiveelementLEMMA}
Let $L$ be a finite-dimensional split-semisimple Lie algebra and $A$ be an $L$-algebra. If the $L$-ideal $I$ of $A$ is generated by weight vectors $a_1,\ldots,a_n\in A$ of different weights, then $I$ is generated by $a_1+\cdots+a_n$.
\end{lemma}

\begin{proof}(\cite{Catoiu2000})
Let $H$ be a fixed Cartan subalgebra of $L$ with basis $\{h_1,\ldots,h_m\}$ and let $\lambda_{ij}$ denote the eigenvalue of $h_i$ for $a_j$. We proceed by induction on $n$. If $n=1$ then the result is trivial. Suppose $n>1$ and that the conclusion is true for all $r<n$. Since $a_1,\ldots,a_n$ have different weights, there exists $1\leq i\leq m$ such that $z_{in}:=h_i-\lambda_{in}$ does not kill all $a_j$. Reorder the $a_j$ so that $a_j^{z_{in}}=0$ if and only if $j>r$ for some $1\leq r<n$ (it kills at least $a_n$). Put $a:=a_1+\cdots+a_n$. Then $a^{z_{in}} = \alpha_1a_1+\cdots + \alpha_r a_r\in \la a\ra_L$ with $0\neq\alpha_1,\ldots,\alpha_r\in K$. By the inductive hypothesis, $\la a_1,\ldots,a_r\ra_L = \la \alpha_1a_1,\ldots,\alpha_ra_r\ra_L = \la \alpha_1a_1+\cdots +\alpha_ra_r\ra_L \subseteq \la a\ra_L$, which implies that $a_{r+1}+\cdots+a_n\in \la a\ra_L$ since $a=a_1+\cdots+a_n$. By the inductive hypothesis again, $\la a_{r+1},\ldots,a_n\ra_L = \la a_{r+1}+\cdots+a_n\ra_L\subseteq \la a\ra_L$. Therefore $\la a \ra_L=I$.
\end{proof}

\end{pgraph}
\begin{pgraph}\textbf{$\FL$.}
The variety of $L$-algebras contains the \emph{free (nonunital associative) $L$-algebra} $\FL$, freely generated by the countably infinite set of variables $X:=\{x_1,x_2, \dots\}$, which satisfies the following universal property: each map $\gamma:X\rightarrow A$ to an $L$-algebra $A$ can be uniquely extended to an $L$-homomorphism $\FL\rightarrow A$, which we call the \emph{evaluation} of $\FL$ at elements $\gamma(x_1),\gamma(x_2)\ldots$ from $A$.
We can describe $\FL$ as follows: $\FL$ is generated as an algebra by the set $\{x_i^u \ | \ i\in \N^*, u\in\Uop\}$, subject to the relations $x_i^1=x_i$, $x_i^{\lambda u+v}=\lambda x_i^u + x_i^v$, $(\lambda x_i+x_j)^u =\lambda x_i^u+x_j^u$ for all $u,v\in\Uop$, $\lambda\in F$ and $i,j\in \N^*$. Note that given a basis $\mathcal B :=\{e_i\}_{i\in\mathcal I}$ of $\Uop$, $\FL$ is generated as an algebra by the set $\{x_i^{e_j} \ | \ i\in\N^*, j\in\mathcal I\}$ and, moreover, the set
\[\left\{x_{i_1}^{e_{j_1}}\cdots x_{i_n}^{e_{j_n}} \ | \ n\in \N^*, i_1,\ldots,i_n\in \N^*, e_{j_1},\ldots,e_{j_n}\in\mathcal B\right\}\]
is a basis of $\FL$.
The free $L$-algebra is endowed with an $L$-action determined  (as in \eqref{formula(U)}) by
\[(x_{i_1}^{e_{j_1}}\cdots x_{i_n}^{e_{j_n}})^u := \sum x_{i_1}^{e_{j_1}u_1}\cdots x_{i_n}^{e_{j_n}u_n}\]
for $e_{j_1},\ldots,e_{j_n}\in\mathcal B$ and $u\in\Uop$ with $\Delta_{n-1}(u)=\sum u_1\otimes\cdots\otimes u_n$.

The elements of the free $L$-algebra are called \emph{differential polynomials} or \emph{$L$-polynomials}. A \emph{$T_L$-ideal} of the free $L$-algebra is an $L$-ideal which in addition is invariant under all $L$-endomorphisms of $\FL$ or \emph{substitutions},  which send $x_i\mapsto f_i$ for $i\in\N^*$ and $f_i\in\FL$; e.g., there is a substitution sending $x_1$ to $x_1x_2$ and $x_i$ to $x_i$ for $i\neq1$. Special substitutions are those mapping $x_i\mapsto x_{\sigma(i)}^{u_i}$  for $i\in I$ and $x_j\mapsto x_j$ for $j\not\in I$, for given $I:=\{i_1,\ldots,i_n\}$, $\sigma\in S_n$ acting on $I$ and $u_i\in\Uop$ for $i\in I$, which we call substitutions \emph{swapping variables}. When referring to elements of a $T_L$-ideal in at most two variables we write them with ``generic'' variables $x,y$; since $T_L$-ideals are closed under substitutions, $x,y$ may be replaced by any $L$-polynomials $f,g\in\FL$.
Given a set $S\subseteq\FL$, by $\la S\ra_{T_L}$ we denote the smallest $T_L$-ideal containing $S$.
\end{pgraph}

\subsection{The variety of $(L,U)$-algebras}
\mbox{}

\medskip

We want to avoid, as much as possible, the infinite dimensionality of $U(L)$ in the determination of the differential identities of an $L$-algebra. For this reason, we introduce $(L,U)$-algebras.

Fix $L$, an $L$-algebra $A$, and the homomorphism $\phi: U(L)\to\E_F(A)$ corresponding to the  right Hopf $ \Uop$-action on $A$. We denote $ U:=\phi(U(L))\subseteq\E_F(A)$; note that $U$ is a unital associative algebra, which is finite dimensional when $A$ is, but not necessarily a Hopf algebra.  In the following we omit $\phi$ from the notation, but the reader should be aware of the fact that the concepts defined below depend not only on $L$ and $U$ but also on $\phi$.

\begin{pgraph}\textbf{$(L,U)$-algebras.} The $ \Uop$-action on $A$ induces a right action of $ U^{\op}$ as a unital associative algebra on $A$; this action is not necessarily a Hopf action (it is a generalized Hopf algebra action as defined by Gordienko in \cite[Section 2]{Gordienko2013}; see also Berele's \cite[Remark in p.878]{Berele1996}), but satisfies $a^u=a^{\phi(u)}$ for all $u\in\Uop$ and $a\in A$.
Accordingly, we say that an $L$-algebra $B$ is an \emph{$(L,U)$-algebra} if it is endowed with a right algebra $ U^{\op}$-action such that $b^u=b^{\phi(u)}$ for all $u\in\Uop$ and $b\in B$, which we call an \emph{$(L,U)$-action}.

If $B,C$ are two associative algebras and $C$ has a right $ U^{\op}$-action then the associative algebra $B\otimes_F C$ has a right $U^{\op}$-action given by $(b\otimes c)^u:=b\otimes c^u$ for $u\in U^{\op}$, $b\in B$, $c\in C$, while if $\delta$ is a derivation of $C$ then $1\otimes\delta$ is a derivation of $B\otimes_F C$. Therefore, if $C$ is an $(L,U)$-algebra then $B\otimes_F C$ is an $(L,U)$-algebra with $L$-action given by $(b\otimes c)^u:=b\otimes c^u$ for $u\in\Uop$, $b\in B$, $c\in C$.

The class of $(L,U)$-algebras is a variety that contains $A$, denoted by $\textsc V^{L,U}$. Ideals of $(L,U)$-algebras are closed under the $ U^{\op}$-action (equivalently, the $L$-action), and homomorphisms $f:B\rightarrow C$ between $(L,U)$-algebras $B,C$ must satisfy $f(b^u)=f(b)^u$ for $b\in B$, $u\in U^{\op}$ (equivalently, $u\in\Uop$). 
The variety of $(L,U)$-algebras contains the \emph{free (nonunital associative) $(L,U)$-algebra} $\FLU$ freely generated by $X$, which is isomorphic to the quotient $\FL/\ILU$ where
\[\ILU:=\langle\{f^z \ | \ f\in\FL, z\in\ker\phi\}\rangle, \tag{I}\label{formula(I)}\]
which is a $T_L$-ideal.
The structure of $\ILU$ is strongly dependent on the algebraic structure of $U$. For example, if $d:=\phi(\delta)$ with $\delta\in  L$ satisfies $d^2=0$, then $x^{d^2}=x^0=0$ for any $x\in\FLU$, and so
\[2x_1^dx_2^d = x_1^{d^2}x_2 +2x_1^dx_2^d +x_1x_2^{d^2} = (x_1x_2)^{d^2} = 0\]
in $\FLU$.

We refer to the elements of $\FLU$ as \emph{$(L,U)$-polynomials}. The \emph{$T_{L,U}$-ideals} of $\FLU$ are defined analogously to the $T_L$ ideals of $\FL$, with $\la S\ra_{T_{L,U}}$ denoting the smallest $T_{L,U}$-ideal containing the set $S$.
\end{pgraph}
\begin{pgraph}\textbf{$\FU$.}\label{FUdefinition}
Observe that the free $(L,U)$-algebra is generated as an associative algebra by the set $\{x_i^u \ | \ i\in \N^*, u\in\mathcal U\}$ for a basis $\mathcal U$ of $U$, albeit not freely, and the construction of a well-behaved basis of $\FLU$ may prove challenging. To circumvent this issue we introduce $\FU$, an algebra with more freeness than $\FLU$ and a better linear parallelism to $\FL$, which is an $L$-algebra but has no specified $U$-action.
We call $\FU$ the \emph{free algebra with $U$-exponents} and define it as follows: $\FU$ is generated as an algebra by the set $\{x_i^u \ | \ i\in \N^*, u\in U^{\op}\}$, subject to the relations $x_i^1=x_i$, $x_i^{\lambda u+v}=\lambda x_i^u + x_i^v$, $(\lambda x_i+x_j)^u =\lambda x_i^u+x_j^u$ for all $u,v\in U^{\op}$, $\lambda\in F$ and $x_i,x_j\in X$. Note that given a basis $\mathcal U:=\{u_1,\ldots,u_N\}$ of $U$, $\FU$ is freely generated as an associative algebra by $\{x_i^{u_j} \ | \ i\in \N^*, u_j\in \mathcal U\}$ (so this is the same algebra considered by Gordienko in \cite[Section 2]{Gordienko2013}). Moreover, the set
\[\{x_{i_1}^{u_{j_1}}\cdots x_{i_n}^{u_{j_n}} \ | \ n\in\N^*, i_1,\ldots,i_n\in \N^*, u_{j_1},\ldots,u_{j_n}\in\mathcal U\}\]
is a basis of $\FU$. In addition,  setting the subspaces $M_n^U:= \spn\{x_{i_1}^{u_{j_1}}\cdots x_{i_n}^{u_{j_n}} \ | \ i_1,\ldots,i_n\geq1, \ | u_{j_1},\ldots,u_{j_n}\in\mathcal U\}$ of \emph{$U$-monomials of degree $n$} for $n\in\N^*$, we get the grading $\FU=\bigoplus_{n\in \N^*} M_n^U$.

The $L$-action on $\FU$ is determined by
\[(x_{i_1}^{u_{j_1}}\cdots x_{i_n}^{u_{j_n}})^v := \sum x_{i_1}^{u_{j_1}\phi(v_1)}\cdots x_{i_n}^{u_{j_n}\phi(v_n)}\]
for $u_{j_1},\ldots,u_{j_n}\in\mathcal U$ and $v\in\Uop$ with $\Delta_{n-1}(v)=\sum v_1\otimes\cdots\otimes v_n$.
In addition, the operation $(x_i^u)^v:=x_i^{uv}$ for $u,v\in U^{\op}$ is well defined, which allows defining the $L$-endomorphisms of $\FU$ which we call \emph{substitutions swapping variables}, that map $x_i\mapsto x_{\sigma(i)}^{u_i}$ for $i\in I$ and $x_j\mapsto x_j$ for $j\not\in I$, for given $I:=\{i_1,\ldots,i_n\}$, $\sigma\in S_n$ acting on $I$ and $u_i\in U^{\op}$ for $i\in I$.

Although it has no specified $U$-action, and we are not considering any variety of $U$-algebras which would contain it as a free algebra, the algebra $\FU$ satisfies the following universal property: each map $\gamma:X\rightarrow B$ to an algebra $B$ with right $ U^{\op}$-action can be uniquely extended to a homomorphism of  associative algebras $\gamma:\FU\rightarrow B$ such that $\gamma(x_i^u)=\gamma(x_i)^u$ for all $x_i\in X$ and $u\in U^{\op}$. More importantly, although $\FU$ is not an $(L,U)$-algebra, it satisfies $x_i^u=x_i^{\phi(u)}$ for all $u\in\Uop$ and $x_i\in X$. Therefore $\FU$ also satisfies the following universal property: each map $\gamma:X\rightarrow B$ to an $(L,U)$-algebra $B$ can be uniquely extended to an $L$-homomorphism $\gamma:\FU\rightarrow B$ such that $\gamma(f^u)=\gamma(f)^{\phi(u)}$ for all $f\in\FU$ and $u\in\Uop$, which we call an \emph{evaluation} of $\FU$ at elements $\gamma(x_1),\gamma(x_2)\ldots$ from $B$.

We refer to the elements of $\FU$ as \emph{$U$-polynomials}. A \emph{$T_U$-ideal} of $\FU$ is an $L$-ideal (so, invariant under the Hopf $ \Uop$-action) which in addition is invariant  under all $L$-endomorphisms of $\FU$; in particular, under all linear substitutions of the form $x_j\mapsto \sum_{i\in I} \alpha_i x_i$ with fixed $j$, finite $I\subseteq\N^*$ and $\alpha_i\in F$ for $i\in I$, and all the substitutions swapping variables. Note that not every substitution of the variables by $U$-polynomials is valid, as not all are $L$-endomorphisms of $\FU$, with this phenomenon depending on the algebraic structure of $U$: e.g., if there is $\delta\in L$ such that $0\neq d:=\phi(\delta)$ satisfies $d^2=0$ then the substitution $\vp$ mapping $x_1\mapsto x_1x_2$ is not an $L$-endomorphism, as $\vp(x_1^{\delta^2}) = \vp(x_1^{d^2}) = 0 \neq 2x_1^dx_2^d = \vp(x_1)^{\delta^2}$.\\
The $T_U$-ideal generated by $f_1,\ldots,f_m\in\FU$ is the set of $U$-polynomials of the form
\[\sum_{j=1}^r g_j f_{i_j}^{u_j}(p_{j1},\ldots,p_{jk_j}) h_j = \sum_{j=1}^r g_j f_{i_j}(p_{j1},\ldots,p_{jk_j})^{u_j} h_j\tag{TU}\label{formula(TU)}\]
with $r\in\N^*$, $i_j\in\{1,\ldots,m\}$ (where we may have $i_j=i_k$ for $j\neq k$), $u_j\in\Uop$, $g_j,h_j\in\FU$ or $g_j=1$ or $h_j=1$, and $\vp_j(f):=f(p_{j1},\ldots,p_{jk_j})$ with $p_{j1},\ldots,p_{jk_j}\in\FU$ being an $L$-endomorphism of $\FU$ which maps $x_{t_i}\mapsto p_{ji}$ for some $x_{t_i}\in X$ and $1\leq i\leq k_j$.\\
When referring to elements of a $T_U$-ideal in at most two variables we write them with ``generic'' variables $x,y$; since $T_U$-ideals are closed under substitutions swapping variables, $x,y$ may be replaced by any variables $x_i^u,x_j^v$ with $x_i,x_j\in X$ and $u,v\in U^{\op}$.
Given a set $S\subseteq\FU$, by $\la S\ra_{T_U}$ we denote the smallest $T_U$-ideal containing $S$.

\begin{remarks}[\textbf{Action of $L$ on $\FU$}]\label{FUhasnoUaction}
\mbox{}

\enum
\item The free algebra with $U$-exponents is not an $(L,U)$-algebra in general: due to their respective universal properties with respect to $(L,U)$-algebras, if both $\FLU$ and $\FU$ were $(L,U)$-algebras, then they would be isomorphic as $(L,U)$-algebras (see \cite[Theorem 10.7]{BS2012}), in particular as $L$-algebras, which they are not in general.
\item Since $\FU$ is not an $(L,U)$-algebra, attention must be paid to the application of the $L$-action: exponents must be in $\Uop$ in general, and can only be taken from $U^{\op}$ when applied directly on ``isolated'' variables $x_i$. A expression like $(x_1\cdots x_n)^u$ with $n>1$ and $u\in U^{\op}$ makes no sense in $\FU$, and $(x_1\cdots x_n)^u$ with $u\in\Uop$ expands to $\sum x_1^{v_1}\cdots x_n^{v_n}$ for $\Delta_{n-1}(u)=\sum u_1\otimes\cdots\otimes u_n$ and $v_i:=\phi(u_i)$ for $i\in\{1,\ldots,n\}$. 
\item The Hopf $U(L)$-action on $\FU$ does not necessarily produce a $U$-action on $\FU$ via $x^{\phi(u)}:=x^u$ for $x\in\FU$, $u\in\Uop$, as we may have $u,v\in\Uop$ generating different actions on $\FU$ and such that $\phi(u)=\phi(v)$. E.g, if $\delta\in L$ satisfies $\phi(\delta^2)=0=\phi(0)$ but $\phi(\delta)\neq0$, then
    \[(x_1x_2)^{\delta^2} = 2x_1^{\phi(\delta)}x_2^{\phi(\delta)} \neq 0 = (x_1x_2)^0\]
    (recall that we have actually designed $\FU$ for this to happen).
\eenum
\end{remarks}

\end{pgraph}

\subsection{Identities and growth}
\mbox{}

\begin{pgraph}\textbf{$L$-identities, $L$-codimensions and $L$-exponent.} A differential polynomial $f(x_1, \dots , x_n)\in F\langle X| L\rangle$ is a \emph{differential identity} or  \emph{$L$-identity} of the $L$-algebra $B$ if $f(b_{1},\dots,b_{n})=0$ for any $b_1,\ldots,b_n\in B$ ($f$ vanishes under all evaluations of $\FL$ at elements from $B$). We denote by $\I^L(B)$ the set of differential identities of $B$, which is a $T_L$-ideal of the free $L$-algebra (in particular $\I^L(B)$ is closed under the Hopf $\Uop$-action and under substitutions). Note that $\I^L(B)$ is the intersection of all kernels of evaluations of $\FL$ from $B$.

For $n\geq 1$ we denote by $P_n^L$ the vector space of \emph{multilinear differential polynomials} in the variables $x_{1},\dots,x_{n}$, so that
\[P_n^L:=\spn_F\{x_{\sigma(1)}^{e_{i_1}}\cdots x_{\sigma(n)}^{e_{i_n}} \ | \ \sigma\in S_{n} ,e_{i_1},\ldots,e_{i_n}\in \mathcal{B} \},\]
where $S_n$ denotes the symmetric group acting on $\{1,\ldots,n\}$.
As in the ordinary case, since $F$ has characteristic zero, a Vandermonde plus linearization argument shows that the $T_L$-ideal $\I^L(B)$ is completely determined by its multilinear $L$-polynomials (see \cite[Proposition 4.2.3]{Drenskybook}).
We also consider the vector space
\[P_n^L(B):= \dfrac{P_n^L}{P_n^L \cap \I^L(B)}.\]
When the action of $\Uop$ is finite dimensional, i.e., when $U$ is a finite-dimensional algebra, the $n$th \emph{differential codimension} of $B$ is $c_n^L(B):=\dim_F P_n^L(B)$. Moreover, if $B$ is finite dimensional then the limit $\exp^L(B):=\lim_{n\to \infty}\sqrt[n]{c_n^L(B)}$ exists and is a nonnegative integer called the \emph{$L$-exponent of $B$} (\cite{Gordienko2013}).
\end{pgraph}
\begin{pgraph}\textbf{Growth of $L$-varieties.} Given a variety $\textsc V$ of $L$-algebras, the growth of $\textsc V$ is defined as the growth of the sequence of differential codimensions of any $L$-algebra $B$ generating $\textsc V$, i.e., $\textsc V=\V^L(B)$. In this case we set $c_n^L(\textsc V):=c_n^L(B)$, $n\geq 1$, and $\exp^L(\textsc V):=\exp^L(B)$. Then we say that $\textsc V$ \emph{has polynomial growth} if there exist $C,t>0$ such that $c_n^L(\textsc V) \leq C n^t$, i.e., $\exp^L(\textsc V)\leq 1$, and that $\textsc V$ \emph{has almost polynomial growth} if $c_n^L(\textsc V)$ is not polynomially bounded, i.e., $\exp^L(\textsc V)>1$, but every proper subvariety of $\textsc V$ has polynomial growth.
\end{pgraph}
\begin{pgraph}\textbf{Analogues for $U$-algebras and $(L,U)$-algebras.}
Mutatis mutandis, for $B$ an associative algebra with  right $ U^{\op}$-action (resp. an $(L,U)$-algebra), inside $\FU$ (resp. $\FLU$) we define the \emph{$U$-identities} (resp. \emph{$(L,U)$-identities}) of $B$, the $T_U$-ideal $\I^U(B)$ closed under the Hopf $\Uop$-action and the  valid substitutions (resp. the $T_{L,U}$-ideal $\I^{L,U}(B)$ closed under the Hopf $\Uop$-action and substitutions), the vector space of \emph{multilinear $U$-polynomials} $P_n^U$ (resp. of \emph{multilinear $(L,U$)-polynomials} $P_n^{L,U}$), the quotient $P_n^U(B)$ (resp. $P_n^{L,U}(B)$), the $n$th \emph{$U$-codimension} $c_n^U(B)$ (resp. the \emph{$(L,U)$-codimension} $c_n^{L,U}(B)$)  when $U$ is finite dimensional. If $\exp^U(B):=\lim_{n\to\infty}\sqrt[n]{c_n^U(B)}$ \big(resp. $\exp^{L,U}(B):=\lim_{n\to\infty}\sqrt[n]{c_n^{L,U}(B)}$\big) exists, we call it the \emph{$U$-exponent} (resp. \emph{$(L,U)$-exponent}) of $B$; see Lemma \ref{Ldata2UdataLEMMA} below.
Similarly, for a variety of $(L,U)$-algebras we define the growth and the notions of polynomial growth and almost polynomial growth.
\end{pgraph}

\subsection{Computing $L$-data through $U$-data}
\mbox{}

\medskip

In this section we relate $L$-identities to $U$-identities, $L$-varieties to $(L,U)$-varieties, and $L$-cocharacters to $U$-cocharacters.
Since $\FU$ is an $L$-algebra and $\FL$ is the free $L$-algebra, we can consider the $L$-homomorphism $\Psi:\FL\to\FU$ sending $x_i\mapsto x_i$, which is defined by $\Psi(x_i^u):=x_i^{\phi(u)}$ for all $x_i\in X$ and $u\in\Uop$. We denote $\IU:=\ker\Psi$. Analogously we have the $L$-homomorphism $\Theta:\FL\to\FLU$ sending $x_i\mapsto x_i$, defined by $\Theta(x_i^u):=x_i^{\phi(u)}$, which satisfies $\ker\Theta=\ILU$. We have $\IU\subseteq\ILU$. In addition, since $\FLU$ is an $(L,U)$-algebra, we have the $L$-homomorphism $\Gamma:\FU\to\FLU$ sending $x_i\mapsto x_i$, which satisfies $\Theta=\Gamma\circ\Psi$.

Let $B$ be any $(L,U)$-algebra. By definition $\Psi(\I^L(B))=\I^U(B)$, $\IU\subseteq \I^L(B)$, and $\Psi(P_n^L)=P_n^U$. Hence by the isomorphism theorems we get $\dfrac{\I^L(B)}{\IU} \cong \I^U(B)$, $\dfrac{\FL}{\I^L(B)} \cong \dfrac{\FU}{\I^U(B)}$(as $L$-algebras) and $P_n^L(B)\cong P_n^U(B)$ (as vector spaces). We get analogous results for $\FLU$ from $\Theta$, proving the next elementary lemma.

\begin{lemma}[\textbf{Relating $L$-identities to $U$-identities}]\label{Ldata2UdataLEMMA}
\mbox{}

Let $B$ be an $(L,U)$-algebra.
\enum
\item Let $G=\{g_i\}_{i\in\mathcal I}$ be a system of generators of $\I^U(B)$ (resp. $\I^{L,U}(B)$) as a $T_L$-ideal (resp. $T_{L,U}$-ideal), and for each $g_i\in G$ pick a fixed preimage $f_i\in\Psi^{-1}(g)$ (resp. $f_i\in\Theta^{-1}(g)$). Let $F:=\{f_i\}_{i\in\mathcal I}$. Then $\I^L(B)=\la F\ra_{T_L}+\IU$ (resp. $\I^L(B)=\la F\ra_{T_L}+\ILU$).
\item  If $U$ is finite dimensional, then $c_n^L(B)=c_n^U(B)=c_n^{L,U}(B)$ for all $n\geq 1$. Moreover, if $B$ is finite dimensional, then $\exp^U(B)$, $\exp^{L,U}(B)$ exist and $\exp^L(B)=\exp^U(B)=\exp^{L,U}(B)$.
\eenum
\end{lemma}

\begin{remark}[\textbf{$\IU$ is not a $T_L$-ideal}]
\mbox{}

Although $\ILU$ is a $T_L$-ideal, $\IU$ is not a $T_L$-ideal in general: it may not be invariant under the substitution $x_1\mapsto x_1x_2$, $x_i\mapsto x_i$ for $i\neq1$, as $x_1^z\in\IU$ for $z\in\ker\phi$ but $(x_1x_2)^z$ may not belong to $\IU$. E.g., $z:=\delta^2$ with $\delta\in L$, $\delta\not\in\ker\phi$ and $\phi(\delta)^2=0$ satisfies $x_1^z\in\IU$ and $(x_1x_2)^z = x_1^zx_2 + 2x_1^\delta x_2^\delta + x_1x_2^z\not\in\IU$ since $x_1^zx_2, x_1x_2^z\in\IU$ but $x_1^\delta x_2^\delta\not\in\IU$.

Nevertheless, $\IU$ is invariant under substitutions swapping variables (since $\FU\cong\FL/\IU$).
\end{remark}

Therefore, since $\IU\subseteq\I^L(A)$, which is a $T_L$-ideal, the $T_L$-ideal $\langle\IU\rangle_{T_L}$ may contain some interesting $L$-identities of $A$, obtained from substitutions in elements of $\IU$. More concretely, we have:

\begin{proposition}[\textbf{Structure of $\IU$}]\label{kernelPsiPROPOSITION}
\mbox{}

\enum
\item $\IU$ is generated as an associative algebra ideal by the set $\{x_i^z \ | \ x_i\in X, z\in\ker\phi\}$, and as an ideal with substitutions swapping variables by the identity $x^z$ for any $z\in\ker\phi$.
\item $\langle\IU\rangle_{T_L} = \ILU = \displaystyle\bigcap_{B\in\textsc V^{L,U}}\I^L(B)$.
\eenum
\end{proposition}
\begin{proof}
\mbox{}

\enum
\item Since the restriction $\phi: U(L)\rightarrow U$ is an epimorphism, we can write $ U=V\oplus\ker\phi$ for some $V$ such that the restriction $\phi:V\rightarrow U$ is an isomorphism of vector spaces. Let $\mathcal V$, $\mathcal K$ be bases of $V$ and $\ker\phi$, respectively, and set $\mathcal B:=\mathcal V\cup\mathcal K$, which is a basis of $\Uop$. Then $\mathcal F_L:=\{x_{i_1}^{e_{j_1}}\cdots x_{i_n}^{e_{j_n}} \ | \ n\in \N^*, i_1,\ldots,i_n\in \N^*, e_{j_1},\ldots,e_{j_n}\in\mathcal B\}$ is a basis of $\FL$, and we can write $\mathcal F_L=M^V\cup M^K$, where $M^V$ is the set of monomials whose variables have all its exponents in $V$ and $M^K$ is the set of monomials which have at least one variable with exponent in $\ker\phi$. Hence for $f\in\FL$ we can write $f=\sum_{i\in\mathcal I}\alpha^V_i m^V_i +\sum_{j\in\mathcal J}\alpha^K_j m^K_j$, where $\mathcal I,\mathcal J$ are finite sets, $\alpha^V_i,\alpha^K_j\in F$ for all $i\in\mathcal I, j\in\mathcal J$, $m^V_i\in M^V$ for $i\in\mathcal I$ and $m^K_j\in M^K$ for $j\in\mathcal J$. On the other hand, $\mathcal F_U:=\left\{x_{i_1}^{u_{j_1}}\cdots x_{i_n}^{u_{j_n}} \ | \ n\in \N^*, i_1,\ldots,i_n\in \N^*, u_{j_1},\ldots,u_{j_n}\in\mathcal U\right\}$ is a basis of $\FU$ such that $\Psi(M^V)=\mathcal F_U$. Then, since $\Psi(M^K)=0$, $\Psi(f)= \sum_{i\in\mathcal I}\alpha^V_i\Psi(m^V_i)$ is a linear combination of monomials from the basis $\mathcal F_U$ and thus $\Psi(f)=0$ implies $\alpha^V_i=0$ for all $i\in\mathcal I$, i.e., $f\in\spn M^K$ as we wanted to show.

 Since $\IU$ is invariant under substitutions swapping variables, the second claim of this item follows.
\item By definition, $\ILU=\langle\{x^z \ | \ z\in\ker\phi\}\rangle_{T_L}$, and clearly $\langle\{x^z \ | \ z\in\ker\phi\}\rangle_{T_L} = \langle\{x_i^z \ | \ x_i\in X, z\in\ker\phi\}\rangle_{T_L} = \langle\IU\rangle_{T_L}$ by item (1).
    On the other hand, given $f\in\ILU$ in $n$ variables we can write $f=\sum_{i\in\mathcal I} h_if_i^{z_i}g_i$ with $h_i,f_i,g_i\in\FL$ and $z_i\in\ker\phi$ for a finite set $\mathcal I$. Then, for any $(L,U)$-algebra $B$ and any $b_1,\ldots,b_n\in B$ we have $f(b_1,\ldots,b_n)= \sum_{i\in\mathcal I} h_i(b_1,\ldots,b_n)(f_i(b_1,\ldots,b_n))^{\phi(z_i)}g_i(b_1,\ldots,b_n) = 0$ since $\phi(z_i)=0$, what implies $\ILU\subseteq\I^L(B)$. Moreover, we have $\I^L(\FLU) = \ILU$ because $\FLU$ is the free $(L,U)$-algebra. Thus
    we get
\[\ILU\subseteq\bigcap_{\mathclap{B\in\textsc V^{L,U}}}\; \I^L(B)\subseteq\I^L(\FLU) = \ILU.\qedhere\]
\eenum
\end{proof}

\begin{figure}[h]
\begin{framed}
\centering
\[\xymatrix{\FL \ar[r]^-\Psi\ar@(dr,dl)[rr]_-\Theta & \FU\cong \FL/\IU \ar[r]^-\Gamma & \FLU\cong \FL/\ILU}\]
\medskip
\[\ILU = \langle\IU\rangle_{T_L} = \langle\{f^z \ | \ f\in\FL, z\in\ker\phi\}\rangle = \displaystyle\bigcap_{B\in\textsc V^{L,U}}\I^L(B)\]
\caption{\footnotesize Relationships between the different free algebras}
\label{Figure}
\end{framed}
\end{figure}

\medskip

From the previous results we derive the following \textbf{general strategy} for computing the differential identities of $A$:
\enumi
\item We determine its $U$-identities by exploiting the structure of the finite-dimensional algebra $ U\subseteq\End_F(A)$ and the good behavior of $\FU$. We find a system of generators of $\I^U(A)$ (which gives also a system of generators of $\I^{L,U}(A)$) and reduce it to a system $G$ by resorting to the $L$-action. We consider the system $F$ of some fixed preimages of $G$ in $\I^L(A)$.
\item We determine a system of generators $Z$ of the ideal $\ker\phi$, with the aid of the representation theory of $L$ applied to $U(L)$ and of the algebraic geometry of $U(L)$. Then $\ILU$ is generated by $K:=\{x^z, z\in Z\}$ by substitutions and the $L$-action, as $x^{uzv}=((x^u)^z)^v$ for $u,v\in\Uop$, $z\in\ker\phi$.
\item We check if any element of $F$ is generated by the others plus $\la\IU\ra_{T_L}=\ILU$. If so, we remove it and check again.
\item We find the (small) system of generators $F\cup K$ of $\I^L(A)$.
\eenum

\begin{remark}[\textbf{Same evaluations}]\label{sameevaluations}
\mbox{}

Let $B$ be an $(L,U)$-algebra. We not only have $\Psi(\I^L(B))=\I^U(B)$, but also $\Psi^{-1}(\I^U(B))=\I^L(B)$, and similarly $\Theta^{-1}(\I^{L,U}(B))=\I^L(B)$ and $\Gamma^{-1}(\I^{L,U}(B))=\I^U(B)$, since $f(b_1,\ldots,b_n)=\Psi(f)(b_1,\ldots,b_n)=\Theta(f)(b_1,\ldots,b_n)$ for all $f\in\FL$ and all $b_1,\ldots,b_n\in B$. In particular, $f\in \FL$ satisfies $f\in\I^L(B)$ if and only if $\Psi(f)\in\I^U(B)$ if and only if $\Theta(f)\in \I^{L,U}(B)$, and $f\in\FU$ satisfies $f\in\I^U(B)$ if and only if $\Gamma(f)\in\I^{L,U}(B)$.
\end{remark}

\begin{proposition}[\textbf{Relating $L$-varieties to $(L,U)$-varieties}]\label{equivalentvarietiesPROPOSITION}
\mbox{}

Let $B$ be an $(L,U)$-algebra and $C\in\V^L(B)$.
\enum
\item $C$ is an $(L,U)$-algebra such that $C\in\V^{L,U}(B)$.
\item  $\V^L(B)$ has almost polynomial growth if and only if $\V^{L,U}(B)$ has almost polynomial growth.
\item $\I^U(B) \subseteq \I^U(C)$, and $\V^{L,U}(C)$ is a proper subvariety of $\V^{L,U}(B)$ if and only if there exists a $U$-polynomial $f\in\I^U(C)\setminus\I^U(B)$.
\eenum
\end{proposition}
\begin{proof}
\mbox{}

\enum
\item Since $B$ is an $(L,U)$-algebra, $\ILU\subseteq\I^L(B)$ by Proposition \ref{kernelPsiPROPOSITION}(2), and since $C\in\V^L(B)$, $\I^L(B)\subseteq \I^L(C)$. Therefore $\ILU\subseteq\I^L(C)$, whence $x^z\in\I^L(C)$ for all $z\in\ker\phi$, so the right $ U^{\op}$-action $c^u:=c^v$ is well defined for any $c\in C$, $u\in U$ and $v\in\phi^{-1}(u)$ ($\phi(v_1)=\phi(v_2)$ implies $v_1-v_2\in\ker\phi$, so $c^{v_1}=c^{v_2}$), and is clearly an $(L,U)$-action. In addition, $\I^{L,U}(B)=\Theta(\I^L(B))\subseteq\Theta(\I^L(C))=\I^{L,U}(C)$, hence $C\in\V^{L,U}(B)$.
\item By item (1) and the fact that every $(L,U)$-algebra is an $L$-algebra we get $\V^L(B)=\V^{L,U}(B)$ as sets, and $c_n^L(C)=c_n^{L,U}(C)$ for every $C\in\V^{L,U}(B)$ by Lemma \ref{Ldata2UdataLEMMA}(2), which in particular implies that $C\in\V^L(B)$ generates a proper $L$-subvariety if and only if it generates a proper $(L,U)$-subvariety of $\V^{L,U}(B)$.
\item 
$\I^{L,U}(B) \subseteq \I^{L,U}(C)$ by item (1), so $\I^{U}(B) \subseteq \I^{U}(C) $ by Remark \ref{sameevaluations}. Moreover, $\V^{L,U}(C)$ is a proper subvariety of $\V^{L,U}(B)$ if and only if there exists $g\in \I^{L,U}(C)\setminus \I^{L,U}(B)$, if and only if there exists $f\in\Gamma^{-1}(g)$ such that $f\in\I^U(C)\setminus \I^U(B)$ by Remark \ref{sameevaluations}.\qedhere
\eenum
\end{proof}

Therefore we can study the growth of $\V^L(A)$ and its subvarieties by considering $(L,U)$-algebras, $U$-polynomials, and $U$-codimensions.

\medskip

\section{Matrix setting}\label{Section:MatrixSetting}

\subsection{Derivations of $\M$}
\mbox{}
\medskip

In this section, we describe the enveloping algebra $ U$ of the Lie algebra of derivations of $\M$ for $k\geq 2$. Let $Z_k(F)$ denote the center of $\M$  (i.e., the scalar multiples of the identity matrix $I_k$) and let $\Sl_k(F)$ denote the special Lie algebra of order $k$, that is, the set of traceless matrices inside $\M$ endowed with the bracket product.

\begin{pgraph}\textbf{$\bs{\Der(\M)}$ is isomorphic to $\bs{\SL}$.}
As a consequence of the Noether-Skolem theorem, all derivations of $\M$ are inner (\cite[p.100]{Herstein1968}), so $\ad:\M\rightarrow\Der(\M)$ is a surjective linear map between vector spaces. In addition $\ad(A)=\ad(B)$ if and only if $A-B\in Z_k(F)$, so $\ad:\M/Z_k(F)\rightarrow\Der(\M)$ is a linear isomorphism, which moreover satisfies $\ad([A,B])=[\ad_A,\ad_B]$, giving an isomorphism of Lie algebras between $\Der(\M)$ and $\M/Z_k(F)$. On the other hand, since $\chr(F)=0$ we have $\M= Z_k(F)\oplus\SL$ (direct sum of Lie ideals) and hence $\M/Z_k(F)\cong\SL$ as Lie algebras in a natural way.
\end{pgraph}
 From now on we identify $\Der(\M)$ with $\SL$ as the inner derivations arising from $\SL\subseteq\M$, and we fix $L:=\SL$ for the rest of this paper. Observe that $L$ is split simple (\cite[IV.5 Theorem 6]{JacobsonLie}).
\begin{pgraph}\textbf{Structure of $\bs{U}$.} \label{StructureUSL}
From the exposition of the previous paragraph, we infer that the $L$-action of $\SL$ on $\M= Z_k(F)\oplus\SL$ is the direct sum of the trivial action $\rho_0$ on the center and the adjoint action $\Ad$ on $\SL$, whence the image  $U=\phi(U(L))$ of the left representation $\phi$ of $U(L)$ on $\M$ is the direct sum $U=U_1\oplus U_2\subseteq\E_F(\M)$ with $U_1=\E_F(F\cdot I_k)\cong F$ and $U_2$ the enveloping algebra of the adjoint action. Since $\Ad$ is finite dimensional and irreducible (because subrepresentations of $\Ad$ correspond to ideals of $L$, which is simple), by Theorem \ref{fullenvelopingalgebraTHEOREM} we have $U_2=\E_F(\SL)$. Therefore
\[U =\E_F(F\cdot I_k)\oplus\E_F(\Sl_k(F)).\]
In particular, $U$ is a split-semisimple algebra of dimension $(k^2-1)^2+1$.

\end{pgraph}

\subsection{Explicit description of $U^{\op}$}
\mbox{}
\medskip

In this section and the next we describe how to operate with exponents coming from $U^{\op}$. We denote the product of $U^{\op}$ by juxtaposition.

\begin{pgraph}\textbf{Basis of $\M$.}\label{Sbasis}
Let $\{e_{ij}\}_{i,j=1}^k$ be the standard matrix units of $\M$ (with 1 as $(i,j)$ entry and $0$ elsewhere) and denote $h_i:=e_{ii}-e_{i+1,i+1}$ for $1\leq i<k$. Then a basis of $\SL$ is
\[\mathcal{S}:=\{e_{ij}|1\leq i\neq j\leq k\}\cup\{h_1,\ldots,h_{k-1}\},\]
which we expand to a basis $\mathcal{M}$ of $\M=\SL\oplus Z_k(F)$ by appending $g:=I_k$,
\[\mathcal M:=\mathcal S\cup\{g\}.\]
We will also refer to elements $h_{ij}:=e_{ii}-e_{jj}$ for $i,j\in\{1,\ldots,k\}$, $i\neq j$ (thus $h_i=h_{ii+1}$ for $1\leq i<k$). We have $h_{ij}=-h_{ji}$. Let us write $(-1)^{i>j}:=1$ if $i\leq j$ and $(-1)^{i>j}:=-1$ if $i>j$. Then in basis $\mathcal S$ we have
\[h_{ij}=(-1)^{i>j}\sum_{l=\min(i,j)}^{\max(i,j)-1} h_l.\]
\end{pgraph}
\begin{pgraph}\textbf{Basis of $U^{\op}$.}\label{matrixunitsendomorphisms}
Write $x\in\Mat_k(F)$ in basis $\mathcal M$ as $x=\sum_{a\in\mathcal M}\mu_a^x a$, i.e., $\mu_a^x$ denotes the coefficient of $x$ with respect to $a\in\mathcal M$. Then, given $a,b\in\mathcal M$ define $\varphi_{ab}\in\E_F(\M)^{\op}$ by
\[\varphi_{ab}(x):=\mu_a^x b,\]
i.e., $\varphi_{ab}$ is the endomorphism sending basis element $a$ to basis element $b$ and the remaining basis elements to $0$.
For example, if $x=e_{12}+2h_1+3h_2\in\Mat_4(F)$ then $x^{\varphi_{h_1e_{23}}}=2e_{23}$, $x^{\varphi_{h_3e_{23}}}=0$ and
\[x^{\varphi_{h_1e_{23}}\varphi_{e_{23}h_4}}=(x^{\varphi_{h_1e_{23}}})^{\varphi_{e_{23}h_4}} = (2e_{23})^{\varphi_{e_{23}h_4}} = 2h_4.\]
We also define endomorphisms $\vp_{ab}$ for any $a\in\mathcal S$ and $b\in\SL$ by linearity. In particular, for elements $h_{ij}$ and $a\in\mathcal S$, we define
\[\vp_{ah_{ij}}:=(-1)^{i>j}\sum_{l=\min(i,j)}^{\max(i,j)-1} \vp_{ah_l}.\]
Note that, in $\E_F(\M)^{\op}$, for $a,b,c,d\in\mathcal M$ we have
\[\vp_{ab}\vp_{cd}=\delta_{bc}\vp_{ad} \tag{F}\label{formula(F)}\]
where $\delta_{bc}$ is Kronecker's delta, since $(x^{\vp_{ab}})^{\vp_{cd}}=(\mu_a^x b)^{\vp_{cd}} = \delta_{bc}\mu_a^x d = \delta_{bc}x^{\vp_{ad}}$ for $x\in\M$.
 In fact, $\{\varphi_{ab}\}_{a,b\in\mathcal M}$ is nothing else than the standard set of matrix units of $\E _F(\M)^{\op}\cong\Mat_{k^2}(F)$ when basis $\mathcal M$ is fixed for $\M$. With this presentation, $U_2^{\op}\cong\Mat_{k^2-1}(F)$ has $\{\varphi_{ab}\}_{a,b\in\mathcal S}$ as a basis and $U_1^{\op}\cong F$ corresponds to the subspace of endomorphisms spanned by $\varphi_{gg}$. This is the presentation we will use in the following; hence from now on we fix the basis of $U^{\op}$
\[\mathcal U:=\{\vp_{ab}\}_{a,b\in\mathcal S}\cup\{\vp_{gg}\}.\]
Notice that $I_{k^2}=\sum_{a\in\mathcal M} \varphi_{aa}$, so in this way we avoid the explicit use of the identity endomorphism and thus the participation of the problematic ordinary polynomial identities.

To prevent the notation from becoming too cumbersome, throughout the rest of the paper we will omit the letters $\varphi$ from the exponents when applying endomorphisms of $U^{\op}$ to $\M$ or writing $U$- or $(L,U$)-polynomials; so, for example, $x^{h_1h_2}$ is shorthand for $x^{\varphi_{h_1h_2}}$. This notation of the form $x^{ab}$ with $a,b\in\mathcal\SL$ for polynomial $x^{\vp_{ab}}$ in $\FU$ or in $\FLU$ should never be confused with notation for $L$-polynomial $(x^a)^b$ in $\FL$ with $a,b\in L$; we will never write $L$-polynomials in the latter way.
\end{pgraph}

\subsection{Computations involving $U^{\op}$}
\mbox{}

\begin{pgraph}\textbf{Multiplication in $\SL$.}\label{multiplicationtableLie}
The Lie multiplication table of $\mathcal M$ is summarized by the following relations:
\begin{enumerate}
\item $[g,x]=0$ for any $x\in\mathcal M$.
\item $[e_{ij},e_{kl}]=0$ ($j\neq k, l\neq i$), $[e_{ij},e_{jk}]=e_{ik}$ ($k\neq i$), $[e_{ij},e_{ji}]=h_{ij}$.
\item $[h_i,e_{ij}] = e_{ij}$ ($j\neq i,i+1$), $[h_{i-1},e_{ij}] =-e_{ij}$ ($j\neq i-1,i$), $[h_{ij},e_{ij}]=2e_{ij}$.
\item $[h_i,h_j]=0$.
\end{enumerate}
\end{pgraph}
\begin{pgraph}\textbf{Computations involving inner derivations.}\label{innerderivations}
Recall that for $c\in\SL$ we write $C:=\ad_{c}\in U^{\op}$. Among the elements in $U^{\op}$ we find the inner derivations $E_{ij}$ generated by the elements $e_{ij}$ ($i\neq j$), which will play a special role in some results. Denote $\vp_{h_0a}:=0, \vp_{h_ka}:=0$ for $a\in\mathcal S$. Then we can write $E_{ij}$ in basis $\mathcal U$ as
\[E_{ij}=\sum_{l\neq i,j}\vp_{e_{jl}e_{il}} -\sum_{l\neq i,j}\vp_{e_{li}e_{lj}} + \vp_{e_{ji}h_{ij}}+\vp_{h_{i-1}e_{ij}}-\vp_{h_{i}e_{ij}}-\vp_{h_{j-1}e_{ij}}+\vp_{h_{j}e_{ij}}.\tag{E}\label{formula(E)}\]
By Formulas \eqref{formula(F)} and \eqref{formula(E)}, the product of two of these inner derivations is given by
\begin{align*}
E_{ij}E_{rs} =&
\boldsymbol{\delta_{is}}\left(\sum_{l\neq i,j,r}\vp_{e_{jl}e_{rl}} + \boldsymbol{(1-\delta_{jr})}(\vp_{e_{jr}h_{ri}}+\vp_{h_{i-1}e_{rj}}-\vp_{h_{i}e_{rj}}-\vp_{h_{j-1}e_{rj}}+\vp_{h_{j}e_{rj}})\right)+\\
+&\boldsymbol{\delta_{jr}}\left(\sum_{l\neq i,j,s}\vp_{e_{li}e_{ls}} - \boldsymbol{(1-\delta_{is})}(\vp_{e_{si}h_{js}}+\vp_{h_{i-1}e_{is}}-\vp_{h_{i}e_{is}}-\vp_{h_{j-1}e_{is}}+\vp_{h_{j}e_{is}})\right)+\\
+&\boldsymbol{\delta_{is}\delta_{jr}}\left(\vp_{h_{i-1}h_{ji}} -\vp_{h_{i}h_{ji}} -\vp_{h_{j-1}h_{ji}}+\vp_{h_{j}h_{ji}}\right)+\\
+&\boldsymbol{\delta_{ji+1}}\left(\boldsymbol{\delta_{ir-1}}\vp_{e_{i+1i}e_{i+1s}} -\boldsymbol{\delta_{ir}}\vp_{e_{i+1i}e_{is}} -\boldsymbol{\delta_{is-1}}\vp_{e_{i+1i}e_{ri+1}} +\boldsymbol{\delta_{is}}\vp_{e_{i+1i}e_{ri}}\right)+\\
+&\boldsymbol{(1-\delta_{is})(1-\delta_{jr})}\left(\boldsymbol{(1-\delta_{ir})}\vp_{e_{jr}e_{is}}+\boldsymbol{(1-\delta_{js})}\vp_{e_{si}e_{rj}}\right).\tag{EE}\label{formula(EE)}
\end{align*}
In particular, a useful identity derived from \eqref{formula(EE)} is
\[E_{ij}^2 = -2\vp_{e_{ji}e_{ij}}.\tag{E2}\label{formula(E2)}\]
Observe that $e_{ij}^2=0$ ($i\neq j$) implies $E_{ij}^3=0$, since for all $x\in\M$,
\[x^{E_{ij}^3}=[e_{ij},[e_{ij},[e_{ij},x]]] = e_{ij}^3x-3e_{ij}^2xe_{ij}+3e_{ij}xe_{ij}^2 -xe_{ij}^3=0.\]
We will also make use of the important \emph{bracket formula} \[\vp_{ab}C = \vp_{a[c,b]}\tag{B}\label{formula(B)}\]
for any $c\in\SL$ and $a,b\in\mathcal S$ or $a=g=b$, which is true because for all $x\in\M$,
\[x^{\vp_{ab}C} = (\mu_a^xb)^C = [c,\mu_a^x b] = \mu_a^x [c,b] = x^{\vp_{a[c,b]}}.\]
In particular, for any  left $U$-algebra $A$, $x\in A$, and $a\in\mathcal S$, we have
\[(x^{ae_{ij}})^{-E_{ji}}=x^{ah_{ij}}, \,\, (x^{ah_{ij}})^{-\frac12E_{ij}}=x^{ae_{ij}}.\]
The action of the power of a derivation on a product is given by \emph{Leibniz's rule}: for $x,y\in A$, $c\in\SL$ and $p\in\N$,
\[(xy)^{C^p} = \sum_{i=0}^p \binom pi x^{C^i}y^{C^{p-i}}.\]
As an example, let us compute the action of the square of derivation $E_{ij}$ on a product by using Leibniz's rule, the bracket formula, and Formula \eqref{formula(E2)}:
\begin{align*}
(x^{ab}y^{cd})^{E_{ij}^2} &= x^{(ab)E_{ij}^2}y^{cd} + 2x^{(ab)E_{ij}}y^{(cd)E_{ij}} + x^{ab}y^{(cd)E_{ij}^2} =
-2x^{(ab)(e_{ji}e_{ij})}y^{cd} + 2x^{a[e_{ij},b]}y^{c[e_{ij},d]} -2x^{ab}y^{(cd)(e_{ji}e_{ij})} =\\
&= -2(\delta_{be_{ji}} x^{ae_{ij}}y^{cd} - x^{a[e_{ij},b]}y^{c[e_{ij},d]} + \delta_{de_{ji}}x^{ab}y^{ce_{ij}}).
\end{align*}
The action of a general composition of derivations on a product is given by Formula \eqref{formula(U)}.
\end{pgraph}

\subsection{Explicit description of $\Uop$}
\mbox{}
\medskip

 In this section we describe how to operate with exponents coming from $\Uop$. We denote the product of $\Uop$ by juxtaposition.

Recalling that $U^{\op}=\phi(U(L)^{\op})\cong U(L)^{\op}/\ker\phi$, fix a unique preimage $\vr_{ab}\in\phi^{-1}(\vp_{ab})$ for each $\vp_{ab}\in\mathcal U$. Then $U(L)^{\op}=V\oplus\ker\phi$, with
\[\mathcal V:=\{\vr_{ab}\ | \ a,b\in\mathcal S\text{ or }a=g=b\}\]
being a basis of $V$. We extend the notation $\vr_{ab}$ to any $a\in\mathcal S$, $b\in\SL$ by linearity.

\begin{pgraph}\textbf{Preimages of the basis elements.}
We first show a valid assignment of the elements $\vr_{ab}\in U(L)^{\op}$, formed with polynomials of degree at most $6$ in the elements $e_{ij}\in U(L)^{\op}$.
\begin{proposition}\label{assignmentsPROPOSITION} Let $c\cdot v$ denote the scalar product of vectors $c\in F^{k-1}$, $v \in (U(L)^{\op})^{k-1}$ ($c\cdot v:=c_1v_1+\cdots+c_{k-1}v_{k-1}$). Then one valid assignment of $\mathcal V$ is
 \begin{alignat*}{2}
&\vr_{e_{rs}e_{ij}}:= \frac12 e_{sr}^2e_{rj}e_{is} \,\, \medmath{(r\neq j, s\neq i; r\neq s, i\neq j)}\\
&\vr_{e_{rs}e_{ir}}:=-\frac12 e_{sr}^2e_{is} \,\, \medmath{(i\neq s; r\neq s, i\neq r)}, &&\vr_{e_{rs}e_{sj}}:= \frac12 e_{sr}^2e_{rj} \,\, \medmath{(r\neq j; r\neq s, s\neq j)}\\
&\vr_{e_{rs}e_{rs}}:= \frac14 e_{sr}^2e_{rs}^2 \,\, \medmath{(r\neq s)}, &&\vr_{e_{rs}e_{sr}}:=-\frac12 e_{sr}^2 \,\, \medmath{(r\neq s)}
\end{alignat*}%
\begin{alignat*}{2}
&\vr_{e_{rs}h_i}:= \frac12 e_{sr}^2e_{i+1s}e_{ri}e_{ii+1} \,\, \medmath{(i\neq r-1; i\neq r, i\neq s-1; s\neq r)}, \quad &&\vr_{e_{rs}h_{r}}:=-\frac12 e_{sr}^2e_{r+1s}e_{rr+1} \,\, \medmath{(s\neq r+1; s\neq r)}\\
&\vr_{e_{rs}h_{s-1}}:= \frac12 e_{sr}^2e_{rs-1}e_{s-1s} \,\, \medmath{(s\neq r+1; s\neq r)}, &&\vr_{e_{rs}h_{r-1}}:= \frac12 e_{sr}^2e_{r-1s}e_{rr-1} \,\, \medmath{(s\neq r-1; s\neq r)}\\
&\vr_{e_{rr-1}h_{r-1}}:= \frac12 e_{r-1r}^2e_{rr-1}, &&\vr_{e_{r-1r}h_{r-1}}:=-\frac12 e_{rr-1}^2e_{r-1r}
\end{alignat*}
\begin{alignat*}{1}
\vr_{h_ie_{rs}}:=&-\frac1{2k}(c_{is}\cdot v_{rs})e_{rs}^2,\text{ where }\\
&v_{rs}:=(v_{rs1},\ldots,v_{rss-1},v_{rss+1},v_{rss+2},\ldots,v_{rsk})\text{ with }\\
&v_{rsj}:=e_{sj}e_{jr}\text{ for }j\neq r,s\text{ and } v_{rsr}:=-e_{sr},\text{ and }\\
&c_{pq}:=(\overbrace{-k+p,\ldots,-k+p}^{p-1},m,\overbrace{p,\ldots,p}^{k-p-1}), m:=-k+p\text{ if }p<q\text{ and }m:=p\text{ if }p\geq q.\\
\vr_{h_ih_j}:=&\frac1{k}(c_{ij}\cdot w_j)e_{j+1j}e_{jj+1},\text{ where }\\
&w_j:=(w_{j1},\ldots,w_{jj},w_{jj+2},w_{jj+3},\ldots,w_{jk})\text{ with }\\
&w_{jr}:=e_{jr}e_{rj}-\frac14e_{j+1r}^2e_{rj+1}^2\text{ for }r\neq j,j+1\text{ and }w_{jj}:=\frac12e_{jj+1}e_{j+1j}.\\
\vr_{gg}:=&1-\sum_{a\in\mathcal S}\vr_{aa}.
\end{alignat*}
\end{proposition}

For example, for $k=6$ we have $c_{35}=(-3,-3,-3,3,3)$, $v_{45}=(e_{51}e_{14},e_{52}e_{24},e_{53}e_{34},-e_{54},e_{56}e_{64})$ and
\[\vr_{h_3e_{45}}=\frac14(e_{51}e_{14}+e_{52}e_{24}+e_{53}e_{34}+e_{54}-e_{56}e_{64})e_{45}^2.\]

\begin{proof}
Proving the proposition is a matter of verifying that, for each assignment found in the statement of the form $\vr_{ab}:=\sum_{p,\ldots,q}\alpha_{p,\ldots,q} e_p\cdots e_q$ with $\alpha_{p,\ldots,q}\in F$ and $e_p,\ldots,e_q\in\SL$, the identity $\vp_{ab} = \sum_{p,\ldots,q}\alpha_{p,\ldots,q} E_p\cdots E_q$ holds in $U^{\op}$. Accordingly, we skip computations when they are straightforward.
\enum
\item The first 11 assignments of the statement are easily checked with Formulas \eqref{formula(E)}, \eqref{formula(E2)}, \eqref{formula(EE)} and \eqref{formula(F)}.
\item For $\vp_{h_ie_{rs}}$, first check that
\[V^{rsj}:= E_{sj}E_{jr}E_{rs}^2=2(-\vp_{h_{j-1}e_{rs}}+\vp_{h_je_{rs}}+\vp_{h_{s-1}e_{rs}}-\vp_{h_se_{rs}})\text{ for }j\neq r,s \tag{Ea}\label{formula(Ea)}\]
and
\[V^{rsr}:=-E_{sr}E_{rs}^2=2(-\vp_{h_{r-1}e_{rs}}+\vp_{h_re_{rs}}+\vp_{h_{s-1}e_{rs}}-\vp_{h_se_{rs}}).\tag{Eb}\label{formula(Eb)}\]
Then, for fixed $r,s$, write all identities in \eqref{formula(Ea)} for $1\leq j\leq k$, $j\neq r,s$ together with \eqref{formula(Eb)} as a $(k-1)\times(k-1)$  system of linear equations $V^{rs}=M(s)\cdot \vp^{rs}$, with vectors $\vp^{rs}:=(\vp_{h_1e_{rs}},\ldots,\vp_{h_{k-1}e_{rs}})$ and
\small\[
V^{rs}:=\frac12(V_{rs1},\ldots,V_{rsr-1},V_{rsr},V_{srr+1},\ldots,V_{rss-1},V_{rss+1},\ldots,V_{rsk}).\]
\normalsize
Compute $M(s)^{-1}$ to solve the system and find $\vp^{rs}=M(s)^{-1}V^{rs}$. The matrix of coefficients $M(s)$ is described as follows:
Suppose first $1<s<k$. For $i<s-1$ (corresponding to $j=i$ in (Ea) or $r=i$ in (Eb)) and for $i>s$ (corresponding to $j=i+1$ or $r=i+1$), the $i$th row has a $-1$ entry in columns $i-1$ and $s$, a $1$ entry in columns $i$ and $s-1$, and $0$ elsewhere. The $(s-1)$th row (corresponding to $j=s-1$ or $r=s-1$) has a $-1$ entry in columns $s-2$ and $s$ and a $2$ entry in column $s-1$. The $s$th row (corresponding to $j=s+1$ or $r=s+1$) has a $1$ entry in columns $s-1$ and $s+1$ and a $-2$ entry in column $s$. 
\footnotesize
\setcounter{MaxMatrixCols}{20}
\[M(s)=\begin{pmatrix*}[r]
1 & 0 & \cdots & \cdots & \cdots & 0 & 1 & -1 & 0 & \cdots & \cdots & \cdots & \cdots & 0\\
-1 & 1 & 0 & \cdots & \cdots & 0 & 1 & -1 & 0 & \cdots & \cdots & \cdots & \cdots & 0\\
0 &-1 & 1 & 0 & \cdots & 0 & 1 & -1 & 0 & \cdots & \cdots & \cdots & \cdots & 0\\
\vdots & \ddots & \ddots & \ddots & \ddots & \vdots & \vdots & \vdots & \vdots & \ddots & \ddots & \ddots & \ddots & \vdots\\
\vdots & \ddots & \ddots & \ddots & \ddots & \vdots & \vdots & \vdots & \vdots & \ddots & \ddots & \ddots & \ddots & \vdots\\
\vdots & \ddots & \ddots & \ddots & \ddots & \ddots & \vdots & \vdots & \vdots & \ddots & \ddots & \ddots & \ddots & \vdots\\
0 & \cdots & \cdots & \cdots & 0 & -1 & 2 & -1 & 0 & \cdots & \cdots & \cdots & \cdots & 0\\
0 & \cdots & \cdots & \cdots & \cdots & 0 & 1 & -2 & 1 & 0 & \cdots & \cdots & \cdots & 0\\
0 & \cdots & \cdots & \cdots & \cdots & 0 & 1 & -1 & -1 & 1 & 0 & \cdots & \cdots & 0\\
0 & \cdots & \cdots & \cdots & \cdots & 0 & 1 & -1 & 0 & -1 & 1 & 0 & \cdots & 0\\
\vdots & \ddots & \ddots & \ddots & \ddots & \vdots & \vdots & \vdots & \vdots & \ddots & \ddots & \ddots & \ddots& \vdots\\
\vdots & \ddots & \ddots & \ddots & \ddots & \vdots & \vdots & \vdots & \vdots & \ddots & \ddots & \ddots & \ddots& 0\\
0 & \cdots & \cdots & \cdots & \cdots & 0 & 1 & -1 & 0 & \cdots & \cdots \cdots & 0 & -1 & 1
\end{pmatrix*}.\]
\normalsize
Equivalently, if $i<s-1$ or $i>s$, the $i$th column $C_{is}$ of $M(s)$ has entries $1,-1$ in rows $i,i+1$ (resp. $i-1,i$); if $i=s-1$ (resp. $i=s$), entry $2$ (resp. $-2$) in row $s-1$ (resp. $s$) and $1$ (resp. $-1$) elsewhere. Let us show that $M(s)$ is invertible. It is straightforward to check that the row vector
\[c_{is}:=(\overbrace{-k+i,\ldots,-k+i}^{i-1},m,\overbrace{i,\ldots,i}^{k-i-1}), m:=-k+i\text{ if }i<s\text{ and }m:=i\text{ if }i\geq s\]
satisfies $c_{is}\cdot C_{js}=-k\delta_{ij}$, whence the matrix with rows $c_{1s},\ldots,c_{k-1s}$ scaled by $-1/k$ is the inverse of $M(s)$. Therefore $\vp_{h_ie_{rs}}=-\frac1{k}c_{is}\cdot V^{rs}$.\\
In the extreme cases, $s=1$ and $s=k$, the matrix $M(s)$ follows the same pattern with the obvious changes and the same formula gives the inverse.
\item For $\vp_{h_ih_j}$  check that, for $r\neq j,j+1$,
\[E_{jr}E_{rj}E_{j+1j}E_{jj+1} = \vp_{e_{rj+1}e_{rj+1}} +\vp_{h_{r-1}h_j}-\vp_{h_rh_j}-\vp_{h_{j-1}h_j}+\vp_{h_jh_j} \]
by showing first that
$E_{rj}E_{j+1j}=-\vp_{e_{jj+1}e_{rj}}-\vp_{e_{jr}e_{j+1j}}$
and
$(E_{rj}E_{j+1j})E_{jj+1}=\vp_{e_{jj+1}e_{rj+1}}-\vp_{e_{jr}h_j}.$
Next apply that $\vp_{e_{rj+1}e_{rj+1}}=\frac14E_{j+1r}^2E_{rj+1}^2 = \frac14E_{j+1r}^2E_{rj+1}^2E_{j+1j}E_{jj+1}$ to find
\[W^{jr}:=E_{rj}E_{j+1j}E_{jj+1} - \frac14E_{j+1r}^2E_{rj+1}^2E_{j+1j}E_{jj+1} = \vp_{h_{r-1}h_j}-\vp_{h_rh_j}-\vp_{h_{j-1}h_j}+\vp_{h_jh_j}\text{ for }r\neq j,j+1. \tag{Ec}\label{formula(Ec)}\]
Check also that \[W^{jj}:=E_{jj+1}E_{j+1j}^2E_{jj+1} = 2(-\vp_{h_{j-1}h_j}+2\vp_{h_jh_j}-\vp_{h_{j+1}h_j}).\tag{Ed}\label{formula(Ed)}\]
Now fix $j$ and proceed as in the previous case by solving for $\vp^j:=(\vp_{h_1h_j},\ldots,\vp_{h_{k-1}h_j})$ the $(k-1)\times(k-1)$ system of linear equations $W^j=-M(j)\vp^j$ generated by \eqref{formula(Ec)} and \eqref{formula(Ed)}, where $W^j:=(W^{j1},\ldots,W^{jj-1},\frac12W^{jj},W^{jj+2},\ldots,W^{jk})$
and $M(s)$ is (thankfully!) the matrix described in the previous item. Therefore $\vp_{h_ih_j}=\frac1k c_{ij}\cdot W^j$.
\item For $\vp_{gg}$ use that $1_{U^{\op}}=\vp_{gg} + \sum_{a\in\mathcal S}\vp_{aa}$.\qedhere
\eenum
\end{proof}
\end{pgraph}
\begin{pgraph}\textbf{Generator of $\ker\phi$.}
The ideal $\ker\phi$ of $\Uop$ is infinite dimensional, but is finitely generated; in this subsection, we show, through the primitive element lemma, that it is in fact a principal ideal. For most of this section we work with $U(L)$.

\medskip

Let $\Ad_k$ denote the adjoint representation of $\SL$ and $\chi_k$ its associated central character, and let $\phi_k$ be the representation of $\SL$ on $\M$ given by the action of $\ad$. We have $\phi_k=\Ad_k\oplus\rho_0$ with $\rho_0$ acting on $F\cdot g$. Fixing the Cartan subalgebra of diagonal traceless matrices and the set of positive roots giving $N_+=\spn\{e_{12},e_{23},\ldots,e_{k-1k}\}$, the highest weight vector of $\Ad_k$ is $e_{1k}$.

\medskip

Denote $x^i_j:=e_{ij}$ for $1\leq i\neq j\leq k$, and for $1\leq i\leq k$,
\[x^i_i:=\frac1k\sum_{j=1}^{k-1}\alpha_{ij}h_j,\,\,
\alpha_{ij}:=k-j\text{ if }j\geq i,\,\,
\alpha_{ij}:=-j\text{ if }j<i.\tag{X}\label{formula(X)}\]
The elements $x^i_j$ form a set of generators of $\SL$ satisfying $[x^i_j,x^r_s]=\delta_{jr}x^i_s-\delta_{is}x^r_j$. Then the Casimir elements
\[c_{p,k}:=\sum_{i_1,\ldots,i_p=1}^k x^{i_1}_{i_2}x^{i_2}_{i_3}\cdots x^{i_p}_{i_1}, \,\, 2\leq p\leq k,\tag{Ca}\label{formula(Ca)}\]
which have rational coefficients in the PBW basis, form a set of Casimir generators of $Z(U(L))$ (see \cite[(6),(64)]{PerelomovPopov1968}). For example,
\begin{align*}
c_{3,3}& = (x^1_1)^3+(x^1_1)^2x^1_2+(x^1_1)^2x^1_3+ x^1_1x^1_2x^2_2 +x^1_1x^1_2x^2_3 + \cdots + (x^3_3)^3 =\\
& = 2/9h_1^3 + 1/3h_1^2h_2 - 1/3h_1h_2^2 - 2/9h_2^3 + 2h_1^2 + h_1h_2 + 4h_1 + 2h_2 +\\
& + h_1e_{21}e_{12} + 2h_2e_{21}e_{12} - 2h_1e_{32}e_{23} - h_2e_{32}e_{23} + 3e_{31}e_{12}e_{23} + 3e_{21}e_{32}e_{13} + h_1e_{31}e_{13} - h_2e_{31}e_{13} +\\
& + 6e_{21}e_{12} + 3e_{31}e_{13}.
\end{align*}

Let $\lambda_{p,k}$ denote the eigenvalue of $c_{p,k}$ for $\Ad_k$. These eigenvalues are the following positive integers.
\begin{lemma}\label{eigenvalues} Put $\lambda_{1,k}:=0$ and $\lambda_{p,k}:=\chi_k(c_{p,k})$ for $2\leq p\leq k$. Then $\lambda_{p,k}=\lambda_{p-2,k}+k^{p-1}$ with $\lambda_{2,k}=2k$, thus
\[\lambda_{p,k}=\left\{\begin{array}{lc}\displaystyle
k\left(\frac{k^p-1}{k^2-1} + 1\right), & p\text{ even}\\\displaystyle
k^2\frac{k^{p-1}-1}{k^2-1}, & p\text{ odd}
\end{array}\right..\]
\end{lemma}

\begin{proof}
By \cite[(8),(14),(16)]{PerelomovPopov1968} we have $\lambda_{p,k}=\tr(A_k^pE_k)$, where $E_k$ is the $k\times k$ matrix full of ones and $A_k$ is the $k\times k$ upper triangular matrix with $(i,j)$ entries equal to $-1$ when $i<j$ and diagonal \[(m_1+k-1,m_2+k-2,\ldots,m_{k-1}+1,m_k),\]
where $m_i$ is the eigenvalue of $x_i^i$ for the highest weight vector of the adjoint representation, i.e., $x_i^ie_{1k}=:m_ie_{1k}$. A straightforward computation with Formula \eqref{formula(X)} produces $m_1=1$, $m_2=\cdots=m_{k-1}=0$, $m_k=-1$. By induction on $p$ with base case $p=1$ it is proven that $A_k^pE_k=A_k\cdot(A_k^{p-1}E_k)$ equals
\[\begin{pmatrix}
a_{p,k} & a_{p,k} & \cdots & a_{p,k}\\
0 & 0 & \cdots & 0\\
\vdots & \ddots & \ddots & \vdots\\
-1 & -1 & \cdots & -1
\end{pmatrix}\]
when $p$ is odd and
\[\begin{pmatrix}
a_{p,k} & a_{p,k} & \cdots & a_{p,k}\\
1 & 1 & \cdots & 1\\
\vdots & \ddots & \ddots & \vdots\\
1 & 1 & \cdots & 1
\end{pmatrix}\]
when $p$ is even, with $a_{1,k}=1$,
\[a_{p,k}=\left\{\begin{array}{lc}
ka_{p-1,k}+1, & p\text{ even}\\
ka_{p-1,k}-(k-1), & p>1\text{ odd}
\end{array}\right.,\]
and
\[\lambda_{p,k}=\tr(A^p_kE_k)=\left\{\begin{array}{lc}
a_{p,k}+k-1, & p\text{ even}\\
a_{p,k}-1, & p>1\text{ odd}
\end{array}\right..\]
Therefore
\[\lambda_{p,k}=\left\{\begin{array}{lc}
k(a_{p-1,k}+1) = k(\lambda_{p-1,k}+2), & p\text{ even}\\
k(a_{p,k}-1) = k(\lambda_{p-1,k}-k), & p>1\text{ odd}
\end{array}\right.,\]
with $\lambda_{2,k}=2k$, $\lambda_{1,k}=0$. Notice that $\lambda_{p,k}-\lambda_{p-2,k}=k(\lambda_{p-1,k}-\lambda_{p-3,k})$ regardless of whether $p$ is even or odd. By recursion $\lambda_{p,k}-\lambda_{p-2,k}=k^{p-3}(\lambda_{3,k}-\lambda_{1,k})=k^{p-3}k^2=k^{p-1}$, hence by recursion again we find
\[
\lambda_{p,k}=\left\{\begin{array}{lc}\displaystyle
\sum_{i=1}^{p/2-1} k^{2i+1} +2k= k\left(\frac{k^p-1}{k^2-1} + 1\right), & p\text{ even}\\\displaystyle
\sum_{i=1}^{(p-1)/2} k^{2i} = k^2\frac{k^{p-1}-1}{k^2-1}, & p\text{ odd}
\end{array}\right..\qedhere\]
\end{proof}

Clearly $c_{p,k}-\lambda_{p,k}\in\ker\chi_k\subseteq\ker\Ad_k$ for $2\leq p\leq k$, and we also have $e_{ij}^3\in\ker\Ad_k$ for $1\leq i\neq j\leq k$. We prove that these elements generate $\ker\Ad_k$ in $U(L)$ and show that $\ker\phi_k$ is a principal ideal. We build on ideas from \cite[Corollary 3.1.4 and Proposition 3.1.7]{Catoiu2000}, which solve the problem for $k=3$.

\begin{theorem}\label{KerPhiGeneratorsTHEOREM}
Denote $z_{p,k}:=c_{p,k}-\lambda_{p,k}$ for $2\leq p\leq k$ and $z'_{p,k}:=\lambda_{2,k}c_{p,k}-\lambda_{p,k}c_{2,k}$ for $3\leq p\leq k$.
\enum
\item $\ker\Ad_k = \la e_{12}^3, z_{2,k},\ldots,z_{k,k}\ra$.
\item $\ker\phi_k = \la e_{12}^3,e_{12}z_{2k}, z'_{3,k},\ldots,z'_{p,k}\ra$.
\item $\ker\phi_k = \la e_{12}^3+e_{12}z_{2,k} + e_{13}z_{3,k} +\cdots +e_{1k}z_{k,k}\ra$.
\eenum
\end{theorem}

\begin{proof}
Let $K$ be a field extension of $F$. Then $\Sl_k(K)=\SL\otimes_F K$, $U(\Sl_k(K))=U(L)\otimes_F K$, the adjoint representation of $\Sl_k(K)$ is $\Ad_k\otimes_F K$, and if $\rho$ is a representation of $\SL$ then $\rho\otimes_F K$ is a representation of $\Sl_k(K)$ such that
\[\ker_{U(\Sl_k(K))}(\rho\otimes_F K) = (\ker_{U(L)}\rho)\otimes_F K.\]
 In addition, if $I$ is an ideal of $U(L)$ such that $I\otimes_F K = \la g_1,\ldots,g_n\ra$ in $U(\Sl_k(K))$ with $g_i\in I$ for $1\leq i\leq n$ then $I=\la g_1,\ldots,g_n\ra$ in $U(L)$. Therefore, by extension and restriction of scalars, we can assume without loss of generality that $F=\Co$.
\enum
\item Clearly $I_k:=\la e_{12}^3, z_{2,k},\ldots,z_{k,k}\ra\subseteq\ker\Ad_k$. Let us show the opposite inclusion. First, we see that $I_k$ has finite codimension. Consider a deglex order on the set of monomials of $U(L)$ with \linebreak
    $h_i>e_{ii+1}>e_{i+1i}$ for $1\leq i<k$. Since $U(L)=I_k\oplus\spn N(I_k)$, the ideal $I_k$ has finite codimension if and only if $\spn N(I_k)$ has finite dimension, hence if and only if there are $n_1,\ldots,n_{k^2-1},m_1,\ldots,m_{k-1}\in\N$ such that $e_{ij}^{n_{ij}},h_l^{m_l}\in\LM(I_k)$ for all $1\leq i\neq j\leq k$ and $1\leq l<k$. In the next identities let $\ad$ denote the adjoint map of $U(L)$; since we have
    \begin{align*}
    &e_{21}^3 = -\frac1{6!}\ad_{e_{21}}^6(e_{12}^3)\text{ if } k=2,\\
    &e_{ij}^3 = \frac1{6}\ad_{e_{il}}^3(e_{lj}^3),\,\,\,\,
    e_{ij}^3 = -\frac1{6}\ad_{e_{lj}}^3(e_{il}^3)\,\,\text{ for }i\neq l\neq j\neq i\text{ if }k>2,
    \end{align*}
    starting from $e_{12}^3$ we can show $e_{ij}^3\in I_k$ for all $1\leq i\neq j\leq k$, for all $k\geq2$: for $k\geq3$, use the third identity to get $e_{1j}^3$ for all $3\leq j\leq k$ from $e_{12}^3$, the second identity to get $e_{i2}^3$ for all $3\leq i\leq k$ from $e_{12}^3$, then the second identity again to get $e_{21}^3$ from $e_{31}^3$; and so on. Also, since for all $k\geq2$ and all $1\leq i<k$ we have
    \[\frac16\ad^3_{e_{ii+1}}(e_{i+1i}^3) = h_i^3-6e_{i+1i}e_{ii+1}h_i -3h_i^2 +4h_i,\]
    we find $h_i^3\in\LM(I_k)$ for $1\leq i<k$.
    This proves that $I_k$ has finite codimension. Now, since $z_{2,k},\ldots,z_{k,k}$ is a set of Casimir generators, $I_k\cap Z(U(L))$ is a maximal ideal of $Z(U(L))$, and so Lemma \ref{maximalideal} shows that $I_k$ is a maximal ideal of $U(L)$, implying $I_k=\ker\Ad_k$.
\item The representation $\phi_k$ is the direct sum of the adjoint representation $\Ad_k$ and the trivial representation $\rho_0$, so $\ker\phi_k=\ker\Ad_k\cap\ker\rho_0=I_k\cap U^*(L)$, where $U^*(L)$ is the nonunital universal enveloping algebra of $L$; i.e., $\ker\phi_k$ is formed by those elements of $I_k$ which do not have a nonzero constant term.
We first change the set of Casimir generators to get rid of unnecessary constant terms in the generators of $I_k$. By Lemma \ref{eigenvalues} we have $\lambda_{2,k}=2k\neq0$, hence the matrix
\[\pma
1 & 0 & \cdots & \cdots & 0 \\
\lambda_{3,k} & -\lambda_{2,k} & 0 & \cdots & 0 \\
\lambda_{4,k} & 0 & -\lambda_{2,k} & \cdots & 0 \\
\vdots & \vdots & \ddots & \ddots & \vdots \\
\lambda_{k,k} & 0 & \cdots & 0 & -\lambda_{2,k}
\epma\]
is invertible. Therefore the central elements $c_{2,k}$ and $z'_{p,k}$ for $3\leq p\leq k$ form another set of Casimir generators such that $\chi_k(c_{2,k})=\lambda_{2,k}$, $\chi_k(z'_{p,k})=0$ for $3\leq p\leq k$. Then
\[\ker Ad_k=\la e_{12}^3,z_{2k}, z'_{3,k},\ldots,z'_{p,k}\ra\]
with $e_{12}^3, z'_{3,k},\ldots,z'_{p,k}\in U^*(\SL)$. Put $I:=\la e_{12}^3, z'_{3,k},\ldots,z'_{p,k}\ra$, $J:=\la z_{2,k}\ra$ and $M:=U^*(L)$. Then $\ker\phi_k=(I+J)\cap M = I+J\cap M$ because $I\subseteq M$. Since $M$ is a maximal ideal and $z_{2,k}\not\in M$, $J+M=U(L)$, whence
\[JM\subseteq J\cap M = U(L)(J\cap M) = (J+M)(J\cap M)\subseteq JM + MJ = JM\]
since $z_{2,k}$ is central. This shows $J\cap M = JM$. Moreover, since $\SL$ is simple, $U^*(L)=\la e_{12}\ra$, so $JM = \la z_{2,k})\I(e_{12}\ra = \la z_{2,k}e_{12}\ra$ because $z_{2,k}$ is central. Therefore
\[\ker\phi_k=I+JM = \la e_{12}^3,e_{12}z_{2,k},z_{3,k}',\ldots,z_{k,k}'\ra.\]
Now repeat the argument above with $I:=\la e_{12}^3\ra$, $J:=\la z_{2,k},\ldots,z_{k,k}\ra$ to arrive at
$\ker\phi_k = \la e_{12}^3,\allowbreak e_{12}z_{2,k},\ldots,e_{12}z_{k,k}\ra$. Since $[x,yz_{p,k}]=[x,y]z_{p,k}$ for all $x,y\in U(L)$ and $\la e_{1p}\ra=U^*(L)$ for all $2\leq p\leq k$, we get $\la e_{12}z_{p,k}\ra=\la e_{1p}z_{p,k}\ra$ for all $3\leq p\leq k$, so $\ker\phi_k = \la e_{12}^3,e_{12}z_{2,k},\ldots,e_{1k}z_{k,k}\ra$.
\item The elements $e_{1i}$ correspond to different roots $\alpha_i$ of $\SL$ and as such are weight vectors of different weights for the adjoint action of $\SL$ on $U(L)$. The identity $[h_i,e_{1p}z_{p,k}]=\alpha_p(h_i)e_{1p}z_{p,k}$ for $1\leq i<k$, $2\leq p\leq k$ shows that the elements $e_{1p}z_{p,k}$ are also weight vectors of $U(L)$ of different weights, which are also different from the weight $3\alpha_2$ of the weight vector $e_{12}^3$. Since ideals of $U(L)$ are $L$-ideals for the adjoint action, by the primitive element lemma (Lemma \ref{primitiveelementLEMMA}) we find
    $\ker\phi_k=\la e_{12}^3,e_{12}z_{2,k},\ldots,e_{1k}z_{k,k}\ra = \la e_{12}^3+e_{12}z_{2,k}+\cdots+e_{1k}z_{k,k}\ra$.
\qedhere
\eenum
\end{proof}

 Now we turn back to $\Uop$.
\begin{corollary}[\textbf{Generator of $\ker\phi$}]\label{kerphigeneratorCOROLLARY}
\[\ker\phi = \la e_{12}^3+e_{12}z_{2,k} + e_{13}z_{3,k} +\cdots +e_{1k}z_{k,k}\ra.\]
\end{corollary}
\begin{proof}
We have $\phi=(\phi_k)^{\op}$, so $\ker\phi=\ker\phi_k$ as sets with $z:=e_{12}^3+e_{12}z_{2,k} + e_{13}z_{3,k} +\cdots +e_{1k}z_{k,k}$ generating $\ker\phi_k$ in $U(L)$, and since $z_{p,k}\in Z(U(L))$ for $2\leq p<k$, $z$ generates $\ker\phi$ in $U(L)^{\op}$.
\end{proof}
\end{pgraph}

\subsection{$U$-polynomials and $U$-identities}
\mbox{}

\begin{pgraph}\textbf{Modifying the first index of an exponent.}\label{firstsuperindex}
As we will see below in \ref{generatingUidentities}, it is the second basis element in the subindex of the endomorphisms $\varphi_{ab}$ which carries the weight of the identities of $\M$. The first basis element is not that relevant, and in fact it can be freely changed by substitution: On the one hand, if $f(x_1,\ldots,x_n)$ is a $U$-identity of $(L,U)$-algebra $A$ and $u_1,\ldots,u_n\in U$, then $f(x_1^{u_1},\ldots,x_n^{u_n})$ is a $U$-identity of $A$. On the other hand,
$(x^{ab})^{bc}=x^{ac}$ for $a,b,c\in\mathcal S$. Therefore, if variable $x^{ab}$ with $a,b\in\mathcal S$ features in a $U$-identity of $A$, there is an analogous $U$-identity replacing $x^{ab}$ with variable $x^{cb} = (x^{ca})^{ab}$ for any $c\in\mathcal S$. Hence we can fix one element $a\in\mathcal S$ and assume that each endomorphism appearing as exponent in a $U$-identity is either $\vp_{gg}$ or of the form $\varphi_{ab}$ for some $b\in\mathcal S$; that is, when looking for the generating identities of the $T_U$-ideal of $A$ we may assume that all exponents different from $\vp_{gg}$ start with the same first basis element $a$. Taking this into account, for a fixed and previously specified $a\in\mathcal S$, we will write $x^b$ as a shorthand for $x^{ab}$ ($x^{\varphi_{ab}}$) with $b\in\SL$ (this should not be confused with element $x^b\in\FL$); we will also write $x^g$ as a shorthand for $x^{gg}$ ($x^{\vp_{gg}}$). 
 In particular, in this format, the bracket formula takes the simple form $(x^b)^C = x^{[c,b]}$ for $b,c\in\SL$ and $C=\ad_c$.
\end{pgraph}

\begin{pgraph}\textbf{Fixed-exponents components of $U$-identities.}\label{components}
We expand and add rigor to \ref{firstsuperindex}. Let us write multilinear $U$-polynomials of $P_n^U$ by grouping their terms with respect to the exponents of their variables, with one set of indices $\mathcal I$ for variables of the form $x_i^{gg}$ and another set of indices $\mathcal J$ for variables of the form $x_j^{a_jb_j}$ with $a_j,b_j\in\mathcal S$, and taking into account how the first exponent indices $a_j$ are paired with the variables $x_j$.
For $n\geq 1$, let $\stirlingtwo{[n]}{2}$ denote the set of pairs of sets $(\mathcal I,\mathcal J)$  such that $\{\mathcal I,\mathcal J\}$ is a partition of the set $\{1,\dots,n\}$ into two disjoint subsets, one of which may be empty  (observe that $(\mathcal I,\mathcal J)$ and $(\mathcal J,\mathcal I)$ are different elements belonging to $\stirlingtwo{[n]}{2}$). Then for $f\in P_n^U$ we have the decomposition
\[f=\sum_{(\mathcal I,\mathcal J)\in\stirlingtwo{[n]}{2}}\,\sum_{\bs a\in \mathcal{S}^{|\mathcal J|}} f_{\mathcal I, \bs a},\]
where for fixed $\mathcal I=\{i_1,\ldots,i_r\}$, $\mathcal J=\{j_1,\ldots,j_{n-r}\}$ and $\bs a=(a_1,\ldots,a_{n-r})$, the $U$-polynomial $f_{\mathcal I, \bs a}$ denotes the sum of all terms of $f$ in which only the variables
\[x_{i_1}^{gg}, \dots, x_{i_r}^{gg}\text{ and }x_{j_1}^{a_1b_1}, \dots, x_{j_{n-r}}^{a_{n-r}b_{n-r}}\text{ for any }(b_1,\dots, b_{n-r})\in \mathcal{S}^{n-r}\]
appear, in any order. Call $f_{\mathcal I,\bs a}$ the $(\mathcal I,\bs a)$ \emph{fixed-exponents component} of $f$. 

Any $U$-polynomial whose only nonzero fixed-exponents component is $(\mathcal I,\bs a)$ is of the form
\[f(\mathcal I,\bs a,\{\alpha_{\sigma,\bs b}\}):=\sum_{\sigma\in S_n}\sum_{\tiny\begin{array}{cc}\bs b=(b_1,\ldots,b_n)\in\mathcal M^n\\b_i=g\text{ if and only if }\sigma(i)\in \mathcal I\end{array}}\!\!\!\!\!\!\!\!
\alpha_{\sigma,\bs b} \, x_{\sigma(1)}^{a_{\sigma(1)}b_1}\cdots x_{\sigma(n)}^{a_{\sigma(n)}b_n},\]
where $a_{i}:=g$ for all $i\in \mathcal I$, $S_n$ is the symmetric group acting on $\{1,\ldots,n\}$, and $\alpha_{\sigma,\bs b}\in F$. If the first exponent index is \emph{homogeneous}, i.e., $\bs a=(a,\ldots,a)$ for $a\in\mathcal S$, then we write $f(\mathcal I, a,\{\alpha_{\sigma,\bs b}\})$. If $\mathcal I=\{1,\ldots,r\}$ then we write $f(r,\bs a,\{\alpha_{\sigma,\bs b}\})$ and say that $f$ \emph{has first $r$ $g$-exponents}.
We show that to study $P_n^U\cap\I^U(A)$ it is enough to study the $U$-identities with only one nonzero fixed-exponents component, with homogeneous first exponent index and first $r$ $g$-exponents.

\begin{lemma}\label{reductiontocomponentsLEMMA}
Let $A$ be an $(L,U)$-algebra, $(\mathcal I,\mathcal J)\in \stirlingtwo{[n]}{2}$ with $\mathcal I=\{i_1,\ldots,i_r\}$ and $\mathcal J=\{j_1,\ldots,j_{n-r}\}$, $\bs a = (a_{j_1},\ldots,a_{j_{n-r}})\in\mathcal S^{n-r}$, and $c\in\mathcal S$.
\enum
\item $f\in P_n^U$ is a $U$-identity of $A$ if and only if any nonzero fixed-exponents component of $f$ is a $U$-identity of $A$.
\item $f(\mathcal I,\bs a,\{\alpha_{\sigma,\bs b}\})$ is a $U$-identity of $A$ if and only if $f(\mathcal I,c,\{\alpha_{\sigma,\bs b}\})$ is a $U$-identity of $A$.
\item $f(\mathcal I,c,\{\alpha_{\sigma,\bs b}\})$ is a $U$-identity of $A$ if and only if $f(r,c,\{\alpha_{\sigma,\bs b}\})$ is a $U$-identity of $A$.
\eenum
\end{lemma}

\begin{proof}
Along the proof we use repeatedly that if $f\in \I^U(A)$ and $\rho$ is a substitution endomorphism swapping variables then $\rho(f)\in\I^U(A)$ (see  \ref{FUdefinition}).
\enum\item One implication is clear from the definition, since any multilinear $U$-polynomial is the sum of its fixed-exponents components. For the other one, fix $f\in P_n^U \cap \I^U(A)$ and for $j\in\mathcal J$ let $\rho_j$ denote the substitution endomorphism sending variable $x_j$ to $x_j^{a_ja_j}$. Then, since $\vp_{a_ja_j}\vp_{a_jb}=\vp_{a_jb}$ and $\vp_{a_ja_j}\vp_{cd}=0$ for $c\neq a_j$,
\[f_{\mathcal I,\bs a} = \rho_{j_1}\circ\cdots\circ\rho_{j_{n-r}}(f).\]    
\item For $j\in\mathcal J$ let $\rho^c_j$ denote the substitution endomorphism sending variable $x_j$ to $x_j^{ca_j}$. Then
\[f(\mathcal I,c,\{\alpha_{\sigma,\bs b}\}) = \rho^c_{j_1}\circ\cdots\circ\rho^c_{j_{n-r}}(f(\mathcal I,\bs a,\{\alpha_{\sigma,\bs b}\})).\]
Analogously, for $j\in\mathcal J$ let $\rho^j_c$ denote the substitution endomorphism sending variable $x_j$ to $x_j^{a_jc}$. Then
\[f(\mathcal I,\bs a,\{\alpha_{\sigma,\bs b}\}) = \rho^{j_1}_c\circ\cdots\circ\rho^{j_{n-r}}_c(f(\mathcal I,c,\{\alpha_{\sigma,\bs b}\})).\]
\item For $\sigma\in S_n$ let $\rho_\sigma$ denote the substitution endomorphism sending variable $x_i$ to $x_{\sigma(i)}$ for all $1\leq i\leq n$. Let $\tau$ be the permutation sending $i_s$ to $s$ for $1\leq s\leq r$ and $j_s$ to $s+r$ for $1\leq s\leq n-s$.
Then
\[f(r,c,\{\alpha_{\sigma,\bs b}\}) = \rho_\tau(f(\mathcal I,c,\{\alpha_{\sigma,\bs b}\}))\text{ and }
f(\mathcal I,c,\{\alpha_{\sigma,\bs b}\}) = \rho_{\tau^{-1}}(f(r,c,\{\alpha_{\sigma,\bs b}\})).\qedhere\]
\eenum
\end{proof}
Moreover, it is now clear that given a set of $U$-identities, we can homogeneously fix one single element $a\in\mathcal S$ as the first exponent index of all $U$-identities in the set, first exponent index which we may elide.

\end{pgraph}
\begin{pgraph}\textbf{Formula for the $U$-codimensions.}\label{codimensionsformula}
For $(\mathcal I,\mathcal J)\in\stirlingtwo{[n]}{2}$ and $\bs a\in\mathcal{S}^{|\mathcal J|}$, we denote by $P_{\mathcal I,\mathcal J,\bs a}^U$ the subspace of $P_n^U$ composed of the multilinear $U$-polynomials whose only nonzero fixed-exponents component is $(\mathcal I,\bs a)$ (see \ref{components}). For $0\leq r\leq n$ and $a\in\mathcal{S}$, we denote $P^U_{r,n-r,a}:=P^U_{\{1,\ldots,r\},\{r+1,\ldots,n\},(a,\ldots,a)}$, the vector space of multilinear $U$-polynomials with only one nonzero fixed-exponents component, with homogeneous first exponent index $a$ and first $r$ $g$-exponents. If $A$ is an $(L,U)$-algebra, we denote $P_{r,n-r,a}^U(A):= \displaystyle\dfrac{P_{r,n-r,a}^U}{P_{r,n-r,a}^U \cap \I^U(A)}$
and $c^U_{r,n-r}(A):=\dim_F P^U_{r,n-r,a}(A)$ for any $a\in\mathcal S$ (it is independent of $a$, as deduced from item (1) of the following proposition).
We show that to study $P_n^U\cap\I^U(A)$ it is enough to study $P^U_{r,n-r,a}\cap\I^U(A)$ for some fixed $a\in\mathcal S$. 
\begin{proposition}\label{P_{r,n-r,c}isomorphismPROPOSITION}
Let $A$ be an $(L,U)$-algebra.
\enum
\item For $(\mathcal I,\mathcal J)\in\stirlingtwo{[n]}{2}$ such that $r=|\mathcal I|$ and any $\bs a\in\mathcal S^{n-r}$ and $c\in\mathcal S$ there is a linear isomorphism between $P_{\mathcal I,\mathcal J,\bs a}^U$ and $P_{r,n-r,c}^U$ sending $P_{\mathcal I,\mathcal J,\bs a}^U \cap\I^U(A)$ to $P_{r,n-r,c}^U \cap\I^U(A)$.
\item \[c_n^U(A)= \sum_{r=0}^{n} \binom{n}{r} (k^2 -1)^{n-r} c^U_{r,n-r}(A). \tag{C}\label{formula(C)}\]
\eenum
\end{proposition}

\begin{proof}
\mbox{}

\enum
\item The proof of item (2) of Lemma \ref{reductiontocomponentsLEMMA} shows a linear isomorphism between $P_{\mathcal I,\mathcal J,\bs a}^U$ and $P_{\mathcal I,\mathcal J,(c,\ldots,c)}^U$ such that $P_{\mathcal I,\mathcal J,\bs a}^U \cap\I^U(A) \cong P_{\mathcal I,\mathcal J,(c,\ldots,c)}^U\cap\I^U(A)$, while the proof of item (3) shows a linear isomorphism between $P_{\mathcal I,\mathcal J,(c,\ldots,c)}^U$ and $P_{r,n-r,c}^U$ such that $P_{\mathcal I,\mathcal J,(c,\ldots,c)}^U\cap\I^U(A) \cong P_{r,n-r,c}^U\cap\I^U(A)$ (the isomorphisms being given by the invertible substitutions specified there).
\item By definition of fixed-exponents components we have 
\[P_n^U = \bigoplus_{(\mathcal I,\mathcal J)\in\stirlingtwo{[n]}{2}, \bs a\in \mathcal{S}^{|\mathcal J|}} P_{\mathcal I,\mathcal J,\bs a}^U.\]
This identity, combined with Lemma \ref{reductiontocomponentsLEMMA}(1), leads to
\[P_n^U\cap\I^U(A) = \bigoplus_{(\mathcal I,\mathcal J)\in\stirlingtwo{[n]}{2}, \bs a\in \mathcal{S}^{|\mathcal J|}} (P_{\mathcal I,\mathcal J,\bs a}^U\cap\I^U(A)).\]
So
\[P_n^U(A) = \dfrac{P_n^U}{P_n^U\cap\I^U(A)} \cong \bigoplus_{(\mathcal I,\mathcal J)\in\stirlingtwo{[n]}{2}, \bs a\in \mathcal{S}^{|\mathcal J|}} \dfrac{P_{\mathcal I,\mathcal J,\bs a}^U}{P_{\mathcal I,\mathcal J,\bs a}^U\cap\I^U(A)}.\]
Now fix $c\in\mathcal S$. By item (1), for each $(\mathcal I,\mathcal J)\in\stirlingtwo{[n]}{2}$ such that $|\mathcal I|=r$ and $\bs a\in\mathcal{S}^{|\mathcal J|}$ there is an isomorphism between $\dfrac{P_{\mathcal I,\mathcal J,\bs a}^U}{P_{\mathcal I,\mathcal J,\bs a}^U\cap\I^U(A)}$ and $\dfrac{P_{r,n-r,c}^U}{P_{r,n-r,c}^U\cap\I^U(A)}$. Therefore,
\[P_n^U(A) \cong \bigoplus_{0\leq r\leq n,\, (\mathcal I,\mathcal J)\in\stirlingtwo{[n]}{2} \, | \, |\mathcal I|=r,\, \bs a\in \mathcal{S}^{n-r}} \dfrac{P_{r,n-r,c}^U}{P_{r,n-r,c}^U\cap\I^U(A)} =  \bigoplus_{0\leq r\leq n}\binom nr (k^2-1)^{n-r} \dfrac{P_{r,n-r,c}^U}{P_{r,n-r,c}^U\cap\I^U(A)},\]
since for any $0\leq r\leq n$ we have $\binom nr$ choices of the first $r$ variables and $(k^2 -1)^{n-r}$ distinct elements in $\mathcal{S}^{n-r}$.
Hence Formula \eqref{formula(C)} follows.\qedhere
\eenum
\end{proof}

Fixed $a\in \mathcal{S}$, when no confusion may arise, we will write $P^U_{r,n-r}:=P_{r,n-r,a}^U$, $P^U_{r,n-r}(A):=P_{r,n-r,a}^U(A).$
\end{pgraph}
\begin{pgraph}{\bf $U$-cocharacter and its decomposition.}\label{U-character and its decomposition} In view of the previous sections, instead of considering the usual permutations of variables as the actions of the symmetric groups on the spaces of multilinear polynomials, we consider the actions which permute variables together with their first index exponent, which is more natural in this context.
Denoting by $S_{\mathcal I}$ and  $S_{\mathcal J}$ the symmetric groups acting on the sets $\mathcal I$ and $\mathcal J$ respectively, the group $S_{\mathcal I} \times S_{\mathcal J}$ acts on $P_{\mathcal I,\mathcal J,\bs a}^U$ on the left in the following way:
\[(\sigma,\tau)(x_{\rho(1)}^{a_{\rho(1)}b_{1}}\cdots x_{\rho(n)}^{a_{\rho(n)}b_{n}}) := x_{\pi\rho(1)}^{a_{\pi\rho(1)}b_{1}}\cdots x_{\pi\rho(n)}^{a_{\pi\rho(n)}b_{n}}\]
where $\rho\in S_n$, $a_{\rho(i)}=g=b_i$ if $\rho(i)\in \mathcal I$, $(\sigma,\tau)\in S_{\mathcal I}\times S_{\mathcal J}$, and $\pi(i):=\sigma(i)$ if $i\in \mathcal I$ while $\pi(i):=\tau(i)$ if $i\in \mathcal J$.
In this way  $P_{\mathcal I,\mathcal J,\bs a}^U$ becomes an $S_{\mathcal I} \times S_{\mathcal J}$-module. 
If $A$ is an $(L,U)$-algebra, then $P_{\mathcal I,\mathcal J,\bs a}^U\cap \I^U(A)$ is invariant under the $S_{\mathcal I} \times S_{\mathcal J}$-action, making
\[P_{\mathcal I,\mathcal J,\bs a}^U(A):=\frac{P_{\mathcal I,\mathcal J,\bs a}^U}{P_{\mathcal I,\mathcal J,\bs a}^U\cap \I^U(A)}\]
an $S_{\mathcal I} \times S_{\mathcal J}$-module with the induced action.

If $|\mathcal I|=r$, then $S_{\mathcal I} \times S_{\mathcal J}\cong S_{r} \times S_{n-r}$ where  $S_r$ and $S_{n-r}$ denote the symmetric groups acting on the sets $\{1,\dots , r\}$ and $\{r+1,\dots, n\}$, respectively. Thus, for any $(\mathcal I,\mathcal J)\in\stirlingtwo{[n]}{2}$ such that $|\mathcal I|=r$ and $\bs a\in\mathcal{S}^{n-r}$, $P_{\mathcal I,\mathcal J,\bs a}^U(A)$ can be regarded as an $S_r\times S_{n-r}$-module. As a consequence, the space
\[P_{(n;r)}^U(A):= \bigoplus_{(\mathcal I,\mathcal J)\in\stirlingtwo{[n]}{2} \, | \, |\mathcal I|=r,\, \bs{a}\in \mathcal{S}^{n-r}}P_{\mathcal I,\mathcal J,\bs a}^U(A)\]
is also an $S_r\times S_{n-r}$-module, whose character we denote by $ \chi_{(n;r)}^U(A)$ and call the \emph{$(n,r)$th $U$-cocharacter of A.}

Now, fixed $a\in \mathcal S$,  the space $P^U_{r,n-r}(A)$ (defined in \ref{codimensionsformula}) is an $S_r \times S_{n-r}$-module whose character we denote by $\chi_{r,n-r}^U(A)$ (it is independent of a, as deduced Proposition \ref{P_{r,n-r,c}isomorphismPROPOSITION}(1)). From the same result we know that for each $(\mathcal I,\mathcal J)\in\stirlingtwo{[n]}{2}$ such that $|\mathcal I|=r$ and $\bs a\in\mathcal{S}^{n-r}$ there is an isomorphism of vector spaces between $P_{\mathcal I,\mathcal J,\bs a}^U(A)$ and $P^U_{r,n-r}(A)$, and since the $S_r\times S_{n-r}$-action commutes with the isomorphism, it is in addition an isomorphism of $S_r\times S_{n-r}$-modules. Therefore
\begin{equation}\label{eq:P_{(n;r)}^U(A)}
      P_{(n;r)}^U(A) \underset{S_r\times S_{n-r}}{\cong\phantom{m}}\binom nr (k^2-1)^{n-r} P^U_{r,n-r}(A). 
\end{equation}

Recall that the irreducible $S_r \times S_{n-r}$-characters are the tensor products $\chi_\lambda \otimes \chi_\mu$ of the irreducible $S_r$- and $S_{n-r}$-characters $\chi_\lambda$ and $\chi_\mu$, where $\lambda\vdash r$ and $\mu \vdash n-r$ are partitions.
Since $\chr(F)=0$, by complete reducibility we can write
	\begin{equation} \label{S_r X S_n-r character}
		\chi_{r,n-r}^U(A)=\sum_{(\lambda,\mu)\vdash(r,n-r)}  m_{\lambda,\mu}\,\chi_\lambda \otimes \chi_\mu,
	\end{equation}
	where $\lambda\vdash r$, $\mu \vdash n-r$, and $m_{\lambda, \mu}\geq 0$ is the multiplicity corresponding to $\chi_\lambda \otimes \chi_\mu$. Thus, as a consequence of \eqref{eq:P_{(n;r)}^U(A)} and \eqref{S_r X S_n-r character}, the $(n,r)$th $U$-cocharacter of A can be decomposed as
 \[ \chi_{(n;r)}^U(A)=\sum_{(\lambda,\mu)\vdash(r,n-r)}  \binom nr (k^2-1)^{n-r} m_{\lambda,\mu}\,\chi_\lambda \otimes \chi_\mu. \tag{$\chi$}\label{formula(chi)}
 \]
 \end{pgraph}

\section{Differential identities of $\M$}\label{Section:DifferentialIdentities}

 In this section we determine the $U$-identities, $(L,U)$-identities, and $L$-identities of the algebra $\M$ for $k\geq 2$.

\subsection{$U$-identities of $\M$}
\mbox{}

\begin{pgraph}\textbf{Multiplication table of $\M$.}\label{multiplicationtableassociative}
The multiplication table arising from $\mathcal M$ (see \ref{Sbasis}), with results expressed in matrix units, is summarized by the following relations:
\begin{enumerate}
\item $gx=x=xg$ for any $x\in\mathcal M$.
\item $e_{ij}e_{jk}=e_{ik}$, $e_{ij}e_{kl}=0$ ($j\neq k$).
\item $h_ie_{ij} = e_{ij}$, $h_ie_{i+1,j}=-e_{i+1,j}$, $e_{ij}h_j = e_{ij}$, $e_{i,j+1}h_{j} = -e_{i,j+1}$, $h_ie_{jk}=0$ ($j\not\in\{i,i+1\}$), $e_{ij}h_k=0$ ($k\not\in\{j,j+1\}$).
\item $h_i^2=e_{ii}+e_{i+1,i+1}$, $h_ih_{i+1}=-e_{i+1,i+1}=h_{i+1}h_i$, $h_ih_j=0$ ($j\not\in\{i-1,i,i+1\}$).
\end{enumerate}
\end{pgraph}

\begin{pgraph}
\textbf{Generating $U$-identities.}\label{generatingUidentities}
Due to the nature of the endomorphisms $\varphi_{ab}$ (see Formula \eqref{formula(F)}), the identities in the multiplication table of $\M$ (see \ref{multiplicationtableassociative}) translate well to $U$-identities of $\M$: for example, if $x,y\in\M$ and $a,b\in\mathcal S$ then
\[x^{ae_{ij}}y^{be_{lm}}=\mu_a^x\mu_b^ye_{ij}e_{lm}=0\text{ if }j\neq l.\]

This idea provides us at once with the following $U$-identities of $\M$ in two variables.
\begin{lemma}[\textbf{$U$-identities from multiplication table}]\label{multiplicationtablegivesidentitiesLEMMA}
\mbox{}

\noindent Fix $a_1,a_2\in \mathcal M$. For $1\leq i\leq r$, fix $\sigma_i\in S_2$ and let $\alpha_i\in F$ and $m^i_1,m^i_2\in\mathcal M$ be such that $\sum_{i=1}^r \alpha_im^i_1m^i_2=0$ in $\M$, with $a_1=g$ (resp. $a_2=g$) forcing $m^i_{\sigma_i(1)}=g$ (resp. $m^i_{\sigma_i(2)}=g$) for all $1\leq i\leq r$. Then 
\[\sum_{i=1}^r \alpha_i x^{a_{\sigma_i(1)}m^i_1}_{\sigma_i(1)}x^{a_{\sigma_i(2)}m^i_2}_{\sigma_i(2)}\in\I^U(\M).\]
\end{lemma}

\begin{proof}
Evaluating in $x_1,x_2\in\M$, by definition of $\vp_{ab}$ we get
\[\sum_{i=1}^r \alpha_i x^{a_{\sigma_i(1)}m^i_1}_{\sigma_i(1)}x^{a_{\sigma_i(2)}m^i_2}_{\sigma_i(2)} = \sum_{i=1}^r \mu^{x_{\sigma_i(1)}}_{a_{\sigma_i(1)}}\mu^{x_{\sigma_i(2)}}_{a_{\sigma_i(2)}}\alpha_im^i_1m^i_2 = \mu^{x_1}_{a_1}\mu^{x_2}_{a_2}\sum_{i=1}^r \alpha_im^i_1m^i_2 = 0.\qedhere\]  
\end{proof}

We will show in the following that all $U$-identities of $\M$ can be generated from $U$-identities in two variables, with at most two terms, arising from its multiplication table as in Lemma \ref{multiplicationtablegivesidentitiesLEMMA}.
\end{pgraph}

Recall that notation $x^b$ with $x\in\FU$, $b\in\mathcal S$ is shorthand for the element $x^{\vp_{ab}}\in\FU$, for a fixed and elided first exponent index $a\in\mathcal S$ which we will not explicitly mention in the next results (see \ref{firstsuperindex}). Similarly we write $P^U_{r,n-r}$ instead of $P^U_{r,n-r,a}$ (see \ref{codimensionsformula}).

\smallskip

We start the description of the $U$-identities and codimensions of $M_k(F)$ by appealing to the linear structure of the $T_U$-ideal. We will tackle the simpler case $k=2$ separately.

\begin{proposition}\label{Theorem Id^U K=2}
The $T_U$-ideal of $U$-identities of $M_2(F)$ is generated by the following $U$-polynomials:
\[[x^g,y^g], \ [x^g, y^a],\ x^a y^b - y^a x^b, \ x^c y^c, \  x^{h_1} y^c +  x^c y^{h_1}, \ x^{e_{12}} y^{e_{21}} + x^{e_{21}} y^{e_{12}}-x^{h_1} y^{h_1},\]
where $a,b \in \{h_1, e_{12}, e_{21}\}$ and $c\in \{e_{12}, e_{21}\}$. In addition, $c_n^U(M_2(F))=4^{n+1}- 3(n+1).$
\end{proposition}
\begin{proof}
Firstly, note that $M_2(F)$ has no nontrivial $U$-identities of degree $1$ (as $x^\vp=0$ implies $\vp=0$ for any $\vp\in\E(M_2(F))$), and so $c_1^U(M_2(F)) = \dim_F(P_1^U) = \dim_F(U) =3^2 +1 = 4^2-3\cdot 2$ as needed. Hence in the following, we assume $n\geq2$.

Let $I$ be the $T_U$-ideal generated by the $U$-polynomials in the statement of the proposition. We will show that $I=\I^U(M_2(F))$.
Recall that $\mathcal S=\{h_1, e_{12}, e_{21}\}$ for $k=2$ by definition.  By Lemma \ref{multiplicationtablegivesidentitiesLEMMA} it follows that $I\subseteq \I^U(M_2(F))$. 

In order to prove the opposite inclusion let $f\in \I^U(M_2(F))$ with $\deg f=n$ and assume, as we may, that $f$ is multilinear and $f\in P^U_{r,n-r}$, where $0\leq r \leq n$ (see \ref{components} and \ref{codimensionsformula}). We will prove that $f\equi I 0$.
	
For all $ 1\leq i \leq n-r$, in order to simplify the notation, let us rename  $x_{r+i}^a$ to $y_i^a$,  $a\in \mathcal S$, so that variables $x_1,\ldots,x_r$ correspond to exponents $g$ and variables $y_1,\ldots,y_{n-r}$ correspond to exponents $a\in\mathcal S$. Since $ [x_1^g,x_2^g], \ [x^g, y^a]\in I$ for all $a\in\mathcal S$, $f$ modulo $I$ is a linear combination of $U$-monomials of type
\[x_1^g \cdots x_r^g y_{i_1}^{a_1}\cdots y_{i_{n-r}}^{a_{n-r}}\]
where $a_{1},\dots, a_{n-r}\in\mathcal S$. If $r=n$ we have $f\equi I \alpha x_1^g \cdots x^g_n$ for some $\alpha\in F$; by evaluating $x_i=g$ for $1\leq i\leq n$ we find, since $f$ is a $U$-identity of $M_2(F)$,
\[0=f(g,\ldots,g)\equi I \alpha(g)^{gg}\cdots (g)^{gg} = \alpha g^n = \alpha g,\]
hence $\alpha=0$ and $f\equi I 0$.
If $r=n-1$ then we have
\[f\equi I \alpha_1 x^g_1 \cdots x^g_{n-1}y_1^{h_1}+ \alpha_2 x^g_1 \cdots x^g_{n-1}y_1^{e_{12}}+ \alpha_3 x^g_1 \cdots x^g_{n-1}y_1^{e_{21}}\]
for some $\alpha_i\in F$, $1\leq i\leq 3$; by evaluating  $x_i=g$ for $1\leq i \leq n-1$ and $y_1=h_1 +e_{12}+ e_{21}$ we analogously get,
for some $a\in\{h_1,e_{12},e_{21}\}$,
\[\alpha_1(g)^{gg}\cdots (g)^{gg}(h_1+e_{12}+e_{21})^{ah_1} + \cdots  + \alpha_3(g)^{gg}\cdots (g)^{gg}(h_1+e_{12}+e_{21})^{ae_{21}} 
=\alpha_1h_1+\alpha_2e_{12}+\alpha_3e_{21}=0\]
and hence $\alpha_{i}=0$ for $1\leq i \leq 3$, and $f\equi I 0$.\\
Now let us assume that $0\leq r \leq n-2$. Since $y_1^a y_2^b - y_2^a y_1^b\in I$ with $a,b\in\mathcal S$, it is possible to reorder the variables $y_1,\ldots,y_r$ in each $U$-monomial of $f$ modulo $I$ without reordering their original exponents. Moreover, since $y_1^c y_2^c\in I$ for $c\in \{e_{12}, e_{21}\}$, modulo $I$ the nonzero terms of $f$ do not have two variables with exponent $e_{12}$ nor with exponent $e_{21}$ adjacent to each other. Since in addition $y_1^{h_1} y_2^c +  y_1^c y_2^{h_1}\in I$ for $c\in \{e_{12}, e_{21}\}$, we can permute the $h_1$ exponents with the $e_{12},e_{21}$ exponents; it then follows that $f$ modulo $I$ can be written as a linear combination of $U$-monomials of the following forms:
\begin{align*}
&x^g_1\cdots x^g_r y_{1}^{h_1} \cdots y_{m}^{h_1} y_{m+1}^{e_{12}} y_{m+2}^{e_{21}} \cdots y_{n-r-1}^{e_{12}} y_{n-r}^{e_{21}},\\
&x^g_1\cdots x^g_r y_{1}^{h_1} \cdots y_{m}^{h_1} y_{m+1}^{e_{21}} y_{m+2}^{e_{12}} \cdots y_{n-r-1}^{e_{21}} y_{n-r}^{e_{12}},\\
&x^g_1\cdots x^g_r y_{1}^{h_1} \cdots y_{m}^{h_1} y_{m+1}^{e_{12}} y_{m+2}^{e_{21}} y_{m+3}^{e_{12}}\cdots  y_{n-r-1}^{e_{21}} y_{n-r}^{e_{12}},\\
&x^g_1\cdots x^g_r y_{1}^{h_1} \cdots y_{m}^{h_1} y_{m+1}^{e_{21}} y_{m+2}^{e_{12}} y_{m+3}^{e_{21}} \cdots y_{n-r-1}^{e_{12}} y_{n-r}^{e_{21}},
\end{align*}
where $0\leq m \leq n-r$. Since $y_1^{e_{12}} y_2^{e_{21}} + y_1^{e_{21}} y_2^{e_{12}} -y_1^{h_1} y_2^{h_1} \in I$ it follows that, if $n-r\geq 2$, we can write $f$ modulo $I$ as a linear combination of $U$-monomials in which at most one exponent $h_1$ appears. Thus, if $n-r$ is even, we get that
\begin{align}
\label{n-r even}
f\equi I & \alpha_1 x^g_1\cdots x^g_r  y_{1}^{h_1} y_{2}^{e_{12}} y_{3}^{e_{21}}  y_{4}^{e_{12}}\cdots y_{n-r-1}^{e_{21}} y_{n-r}^{e_{12}}+\alpha_2 x^g_1\cdots x^g_r  y_{1}^{h_1} y_{2}^{e_{21}} y_{3}^{e_{12}}  y_{4}^{e_{21}} \cdots y_{n-r-1}^{e_{12}} y_{n-r}^{e_{21}}+\\
+& \alpha_3 x^g_1\cdots x^g_r  y_{1}^{e_{12}} y_{2}^{e_{21}}\cdots  y_{n-r-1}^{e_{12}} y_{n-r}^{e_{21}}
+ \alpha_4 x^g_1\cdots x^g_r  y_{1}^{e_{21}} y_{2}^{e_{12}} \cdots y_{n-r-1}^{e_{21}} y_{n-r}^{e_{12}} \notag
\end{align}
for some $\alpha_i\in F$, $1\leq i\leq 4$, whereas if $n-r$ is odd, then we have that
\begin{align*}
f\equi I & \alpha_1 x^g_1\cdots x^g_r  y_{1}^{h_1} y_{2}^{e_{12}} y_{3}^{e_{21}} \cdots y_{n-r-1}^{e_{12}} y_{n-r}^{e_{21}}+ \alpha_2 x^g_1\cdots x^g_r  y_{1}^{h_1} y_{1}^{e_{21}} y_{2}^{e_{12}} \cdots y_{n-r-1}^{e_{21}} y_{n-r}^{e_{12}}+\\
+& \alpha_3 x^g_1\cdots x^g_r  y_{1}^{e_{12}} y_{2}^{e_{21}} y_{3}^{e_{12}}\cdots  y_{n-r-1}^{e_{21}} y_{n-r}^{e_{12}}+ \alpha_4
x^g_1\cdots x^g_r  y_{1}^{e_{21}} y_{2}^{e_{12}} y_{3}^{e_{21}} \cdots y_{n-r-1}^{e_{12}} y_{n-r}^{e_{21}}
\end{align*}
for some $\alpha_i\in F$, $1\leq i\leq 4$.
Suppose that $f$ is as in \eqref{n-r even}. By making the evaluation $x_i=g$ for $1\leq i \leq r$ and $y_j=h_1 +e_{12}+ e_{21}$ for $1\leq j \leq n-r$ we get $\alpha_1 e_{12} -\alpha_2 e_{21}+ \alpha_3 e_{11}+ \alpha_4 e_{22}=0$. Thus $\alpha_i=0$ for $1\leq i \leq 4$, and $f$ is the zero $U$-polynomial modulo $I$. One can deal similarly with the other case. Thus $\I^U(M_2(F))=I$.
	
The argument above also proves that
\[c^U_{r,n-r}(M_2(F))=\begin{cases}
1,  & \mbox{ if } r=n, \\
3,  & \mbox{ if } r=n-1, \\
4, & \mbox{ if } 0\leq r\leq n-2.
\end{cases}\]
Therefore, by Formula \eqref{formula(C)},
\[c_n^U(M_2(F))= \sum_{r=0}^{n} \binom{n}{r} 3^{n-r} c^U_{r,n-r}(M_2(F))=4\sum_{r=0}^{n} \binom{n}{r} 3^{n-r} -3n-3= 4^{n+1}-3(n+1).\qedhere\]
\end{proof}

 Next lemma follows from simple computations.

\begin{lemma} \label{Remark e_ij}
If $k\geq 3$, then:
\enum
\item $x^{e_{il}}y^{e_{lj}} \equiv - x^{h_i}y^{e_{ij}}  \ (\md \langle x^{e_{i j}}y^{h_{j-1}}- x^{e_{i l}} y^{e_{lj}}, \ x^{h_i}y^{e_{ij}}+ x^{e_{ij}}y^{h_{j-1}} \rangle_{T_U})$, where $1\leq i \leq k-1$, $2\leq j \leq k$, $ j\neq i$, $1\leq l \leq k$;
\item $x^{e_{1j}}y^{e_{j1}} \equiv x^{h_1} y^{h_1}+ x^{h_1}y^{h_2} \ (\md \langle x^{e_{1j}} y^{e_{j1}} + x^{e_{2,l}} x_2^{e_{l,2}}-x^{h_1} x_2^{h_1}, \ x^{h_{1}} x_2^{h_2} + x^{e_{2j}} x_2^{e_{j2}} \rangle_{T_U})$, where $2\leq j \leq k$, $1\leq l \leq k$, $l \neq 2$;
\item $x^{e_{kj}}y^{e_{jk}} \equiv x^{h_{k-1}} y^{h_{k-1}}+ x^{h_{k-2}}y^{h_{k-1}} \ (\md \langle x^{e_{k-1,j}} y^{e_{j,k-1}} + x^{e_{k,l}} y^{e_{l,k}}-x^{h_{k-1}} y^{h_{k-1}}, \ x^{h_{k-2}} y^{h_{k-1}} + x^{e_{k-1,l}} y^{e_{l,k-1}} \rangle_{T_U})$, where $1\leq j \leq k-1$, $1\leq l \leq k$, $l\neq k-1$.
\eenum
\end{lemma}

\begin{proposition}\label{Theorem Id^U K>2}
The $T_U$-ideal of $U$-identities of $\M$, $k\geq 3$, is generated by the following $U$-polynomials:
\enum
\item $[x^g,y^g], \ \ [x^g, y^a], \ \ x^a y^b - y^a x^b$, where $a,b\in \mathcal{S}$;
\item $x^{e_{ij}}y^{e_{lm}}$, where $1\leq i,j,l,m\leq k$, $j\neq l$;
\item $x^{h_i}y^{e_{jl}}$, where  $1\leq i \leq k-1$, $1\leq  j, l \leq k$, $ j\neq i, i+1$, $l\neq j$;\label{Id 2a}
\item $x^{e_{jl}}y^{h_i}$ where $1\leq i \leq k-1$, $1\leq j,l\leq k$, $l\neq i, i+1$, $j\neq l$; \label{Id 2b}
\item $x^{h_i}y^{e_{ij}}+ x^{e_{ij}}y^{h_{j-1}} $, where $ 1\leq i \leq k-1$, $ 2\leq j \leq k$,  $j\neq i$; \label{Id 1a}
\item $ x^{h_{i-1}}y^{e_{ij}}+ x^{e_{ij}}y^{h_j}$, where $ 2\leq i \leq k$, $1\leq j \leq k-1$, $ j\neq i$; \label{Id 1b}
\item $x^{h_{i-1}}y^{e_{i j}}+x^{e_{i l}} y^{e_{lj}}$, where $2\leq i \leq k$, $1\leq j,l\leq k$, $ j\neq i$, $i,j\neq l$; \label{Id 2c}
\item $x^{h_{i-1}} y^{h_i} + x^{e_{ij}} y^{e_{ji}}$, where $2\leq i \leq k-1$, $1\leq j \leq k$, $j\neq i$; \label{Id 3}
\item $x^{h_i}y^{h_j}$, where $1\leq i,j\leq k-1$, $j\neq i-1,i,i+1$, if $k\geq 4$; \label{Id 5}
\item $[x^{h_{i}}, y^{h_{i+1}}]$, where  $1\leq i \leq k-2$;	\label{Id 4}
\item $x^{e_{ij}} y^{e_{ji}} + x^{e_{i+1l}}y^{e_{li+1}}-x^{h_i}y^{h_i}$, where $ 1\leq i \leq k-1$, $1\leq j,l\leq k$, $j\neq i$, $ l\neq i+1$; \label{Id 1c}
\item $x^{e_{i j}}y^{h_{j-1}}+ x^{e_{i l}} y^{e_{lj}}$, where $1\leq i,l\leq k$, $2\leq j \leq k$, $j\neq i$, $i,j\neq l$. \label{Id 2d}
\eenum
In addition, $c_n^U(\M)=k^{2(n+1)}- (k^2-1)(n+1).$
\end{proposition}

\begin{proof}
The proof of this result follows a scheme similar to that of Proposition \ref{Theorem Id^U K=2} for $2\times2$ matrices, with different computations. Firstly, note that there are no nontrivial identities of degree $1$ and thus $c_1^U(\M) = \dim_F(U) = (k^2-1)^2 +1 = k^4 - 2(k^2-1)$ as expected, so henceforth we assume $n\geq2$. Let $I$ be the $T_U$-ideal generated by the $U$-polynomials in the statement of the proposition.  By Lemma \ref{multiplicationtablegivesidentitiesLEMMA} it follows that $I\subseteq \I^U(\M)$. To prove the opposite inclusion, first, we find a set of generators of $P^U_{r,n-r}$ modulo $P^U_{r,n-r}\cap I$, for each $n \geq 1$ and $0\leq r\leq n$, and after that, we show, by evaluation, that the sets of generators found are actually bases of their corresponding vector spaces.
	
Let $f\in P^U_{r,n-r}$ be a multilinear $U$-polynomial of degree $n$. In order to simplify the notation, for all $ 1\leq i \leq n-r$ let us rename  $x_{r+i}^a$ to $y_i^a$,  $a\in \mathcal{S}$, so that variables $x_1,\ldots,x_r$ correspond to exponents $g$ and variables $y_1,\ldots,y_{n-r}$ correspond to exponents $a\in\mathcal S$. Since $[x^g_1,x^g_2], \ [x^g, y^a] \in I$ for all $a\in \mathcal{S}$, $f$ modulo $I$ is a linear combination of $U$-monomials of type
\[x_1^g \cdots x_r^g y_{i_1}^{a_1}\cdots y_{i_{n-r}}^{a_{n-r}}\]
where $a_{i_1},\dots, a_{i_{n-r}}\in \mathcal{S}$. So $P^U_{n,0}$ is generated modulo $P^U_{n,0}\cap I$ by the $U$-monomial $x_1^g \cdots x_n^g$. If $r=n-1$ then $f$ modulo $I$ is a linear combination of $U$-monomials
\begin{equation}
\label{generatori r=n-1}
x^g_1 \cdots x^g_{n-1}y_1^{h_i} \text{ for }1\leq i \leq k-1\text{ and }x^g_1 \cdots x^g_{n-1}y_1^{e_{jl}}\text{ for }1\leq j,l \leq k, j\neq l.
\end{equation}
It follows that $P^U_{n-1,1}$ is generated modulo $P^U_{n-1,1}\cap I$ by the $U$-monomials in \eqref{generatori r=n-1}.
	
Now suppose that $0\leq r \leq n-2$. By $U$-identities \ref{Id 2a}-\ref{Id 1b}, in each $U$-monomial of $f$ we can move to the left, modulo $I$, all the variables with exponent $h_i$ for $1\leq i \leq k-1$; moreover, since  $y_1^a y_2^b - y_2^a y_1^b\in I$ for all $a,b\in \mathcal{S}$, we can always reorder the indices of the variables with exponent in $\mathcal S$. Call (P1) to this moving and reordering procedure. From (P1) it follows that $f$ modulo $I$ is a linear combination of $U$-monomials of type
\[x^g_1\cdots x^g_r y_{1}^{h_{i_1}} \cdots y_{s}^{h_{i_s}} y_{s+1}^{b_1}  \cdots y_{n-r}^{b_{n-r}}\]
for $0\leq s\leq n-r$, $1\leq i_1, \dots, i_s\leq k-1$ and $b_1,\dots, b_{n-r}\in \{e_{ij}\ |  \ 1\leq i,j\leq k, i\neq j\}$.
Since $y_1^{e_{ij}}y_2^{e_{lm}}\in I$ for $1\leq i,j,l,m\leq k$, $j\neq l$, we can require in addition that the sequence of exponents $b_1,\dots, b_{n-r}$ has nonzero product $b_1\cdots b_{n-r}$, i.e., if $b_t=e_{ij}$ and $b_{t+1}=e_{lm}$ then $j=l$. Notice that the $U$-identities of Lemma \ref{Remark e_ij} also belong to $I$, thus by $U$-identities \ref{Id 2c},\ref{Id 3} and Lemma \ref{Remark e_ij} we can reduce each pair $y_t^{e_{il}}y_{t+1}^{e_{lj}}$:
\begin{itemize}
\item If $i\neq j,k$, to the pair $y_t^{h_i}y_{t+1}^{e_{ij}}$.
\item If $i=k\neq j$, to the pair $y_t^{h_{k-1}}y_{t+1}^{e_{kj}}$.
\item if $i=j$, to a linear combination of one or two $U$-monomials of the form $y_t^{h_l}y_{t+1}^{h_m}$ with $1\leq l\leq k-1$ and $m\in\{l,l+1\}$.
\end{itemize}
To this reduction procedure call (P2). By applying (P1) and (P2) repeatedly, we find that $f$ modulo $I$ is a linear combination of $U$-monomials of types
\[x^g_1\cdots x^g_r y_{1}^{h_{i_1}} \cdots y_{n-r}^{h_{i_{n-r}}}\text{ and }x^g_1\cdots x^g_r y_{1}^{h_{i_1}} \cdots y_{n-r-1}^{h_{i_{n-r-1}}} y_{n-r}^{e_{ij}}\]
for  $ 1\leq i_1, \dots, i_{n-r}\leq k-1$ and $1\leq i,j\leq k$, $i\neq j$. By $U$-identity \ref{Id 5} we can assume that the product of the exponents $h_{i_1}\cdots h_{i_{n-r}}$ is nonzero, and by $U$-identity \ref{Id 4} we can assume that $i_1 \leq \cdots \leq i_{n-r}$; hence we may assume that $i_{j+1}\in\{i_j,i_j+1\}$ for all $0\leq j<n-r$.
Moreover, since for $k\geq 4$, by $U$-identities \ref{Id 5},\ref{Id 4} we have
\[y_1^{h_{i}}y_2^{h_{i+1}}y_3^{h_{i+2}} = [y_1^{h_{i}}y_2^{h_{i+1}}]y_3^{h_{i+2}}+y_2^{h_{i+1}}(y_1^{h_{i}}y_3^{h_{i+2}})\in I\]
for $1\leq i \leq k-3$, we can assume that, for all $k\geq 3$, in each $U$-monomial of $f$ modulo $I$ there are at most two distinct (and consecutive) exponents $h_i$ ($1\leq i\leq k-1$), the rest of them being copies of one of those.
In addition, by $U$-identities \ref{Id 3} (applied twice) and \ref{Id 1c} we find
\begin{align*}
&y_1^{h_{i}}y_2^{h_{i+1}} +y_1^{h_{i-1}}y_2^{h_i} + y_1^{h_i}y_2^{h_i} =\\
&(y_1^{h_{i}}y_2^{h_{i+1}} + y_1^{e_{i+1j}}y_2^{e_{ji+1}}) + (y_1^{h_{i-1}}y_2^{h_i} + y_1^{e_{ij}}y_2^{e_{ji}})-(y_1^{e_{ij}} y_2^{e_{ji}} + y_1^{e_{i+1j}} y_2^{e_{ji+1}}-y_1^{h_i}y_2^{h_i})\in I
\end{align*}
for all $2\leq i \leq k-2$, so by recursion on $i$ we may suppose that, in each $U$-monomial of $f$, either all exponents $h_i$ are equal (for some $1\leq i\leq k-1$) or there are two distinct exponents, $h_1$ and $h_2$. Now notice that by $U$-identities \ref{Id 1c},\ref{Id 3},\ref{Id 2a} we have
\[y_1^{h_1}y_2^{h_2}y_3^{h_2} + y_1^{h_1} y_2^{h_1}y_3^{h_2} = y_1^{h_1}(y_2^{h_2}y_3^{h_2}-y_2^{e_{21}}y_3^{e_{12}}-y_2^{e_{31}}y_3^{e_{13}}) +y_1^{h_1}(y_2^{h_1}y_3^{h_2}+y_2^{e_{21}}y_3^{e_{12}})+ (y_1^{h_1}y_2^{e_{31}})y_3^{e_{13}}\in I,\]
so $y_1^{h_1}y_2^{h_2}y_3^{h_2} \equi I -y_1^{h_1} y_2^{h_1}y_3^{h_2}$. Hence $f$ modulo $I$ is a linear combination of $U$-monomials of types
\begin{alignat*}{2}
& x^g_1\cdots x^g_r y_{1}^{h_{i}} \cdots y_{n-r}^{h_{i}}, \quad && x^g_1\cdots x^g_r y_{1}^{h_{1}} \cdots y_{n-r-1}^{h_{1}}y_{n-r}^{h_2},\\
& x^g_1\cdots x^g_r y_{1}^{h_{i}} \cdots y_{n-r-1}^{h_i} y_{n-r}^{e_{lj}}, \quad &&  x^g_1\cdots x^g_r y_{1}^{h_{1}} \cdots y_{n-r-2}^{h_1} y_{n-r-1}^{h_2} y_{n-r}^{e_{lj}},
\end{alignat*}
for $1\leq i \leq k-1$, $1\leq j, l \leq k$, $l\neq j$. Finally, since $y_1^{h_{i-1}}y_2^{e_{ij}} + y_1^{h_i} y_2^{e_{ij}} \in \langle  y_1^{h_{i-1}}y_2^{e_{i j}}+ y_1^{e_{i l}} y_2^{e_{lj}}, \ y_1^{e_{i j}}y_2^{h_{j-1}}+ y_1^{e_{i l}} y_2^{e_{lj}}, \ y_1^{h_i}y_2^{e_{ij}}+ y_1^{e_{ij}}y_2^{h_{j-1}} \rangle_{T_U},$ where $2\leq i \leq k-1$, $1\leq j, l \leq k$, $j\neq i$, $l\neq i,j$, and by $U$-identity \ref{Id 2a}, it follows that $f$ modulo $I$ is a linear combination of the following $U$-monomials:
\begin{equation}\label{relativamente libera Mk}
\begin{split}
& x^g_1\cdots x^g_r y_1^{h_i} \cdots y_{n-r}^{h_i},\ \ x^g_1\cdots x^g_r y_1^{h_1} \cdots y_{n-r-1}^{h_1} y_{n-r}^{h_2},\\
& x^g_1\cdots x^g_r y_1^{h_i} \cdots y_{n-r-1}^{h_i} y_{n-r}^{e_{ij}},\ \
x^g_1\cdots x^g_r y_1^{h_{k-1}} \cdots y_{n-r-1}^{h_{k-1}} y_{n-r}^{e_{kl}},
\end{split}
\end{equation}
where $1\leq i, l \leq k-1$, $1\leq j \leq k$, $j\neq i$. Thus we have that $P^U_{r, n-r}$ modulo $P^U_{r,n-r}\cap I$ is generated by the $U$-monomials in \eqref{relativamente libera Mk}.
	
The $U$-monomial $x_1^g \cdots x_n^g$ can be seen to be nonzero modulo $I$ by evaluating $x_1=\cdots =x_n =g$. We next show that the $U$-monomials in \eqref{generatori r=n-1} and \eqref{relativamente libera Mk} are linearly independent modulo $\I^U(\M)$ if $r=n-1$ or $0\leq r\leq n-2$, respectively. To that end, let us assume first that $f\in \I^U(\M)$ is a linear combination of $U$-monomials in \eqref{generatori r=n-1}, i.e.,
\[f=\sum_{1\leq i \leq k-1} \alpha_i x^g_1 \cdots x^g_{n-1} y_1^{h_i}+\sum_{\substack{1\leq j,l \leq k\\ j\neq l}} \beta_{jl} x^g_1 \cdots x^g_{n-1} y_1^{e_{jl}}\]
for some $\alpha_i,\beta_{jl}\in F$. From the evaluation
$x_1=\cdots=x_{n-1}=g\text{ and } y_1=\sum_{a\in\mathcal S} a$
we get
\[\sum_{1\leq i \leq k-1} \alpha_i h_i + \sum_{\substack{1\leq j,l \leq k\\ j\neq l}} \beta_{jl} e_{jl}=0,\]
from which it follows, since $\mathcal S$ is a linearly independent set, that $\alpha_i=\beta_{lj}=0$ for all $1\leq i \leq k-1$, $1\leq j , l \leq k$, $j \neq l$. Therefore the $U$-monomials in \eqref{generatori r=n-1} are linearly independent modulo $\I^U(\M)$.
	
Let us assume now that $f\in \I^U(\M)$ is such that
\begin{align*}
f= & \sum_{1\leq i \leq k-1} \alpha_i x^g_1\cdots x^g_r y_1^{h_i} \cdots y_{n-r}^{h_i} + \beta x^g_1\cdots x^g_r y_1^{h_1} \cdots y_{n-r-1}^{h_1} y_{n-r}^{h_2}+ \sum_{\substack{1\leq i \leq k-1\\ 1\leq j \leq k\\ i\neq j}} \gamma_{ij} x^g_1\cdots x^g_r y_1^{h_i} \cdots y_{n-r-1}^{h_i} y_{n-r}^{e_{ij}}+\\
+&\sum_{1\leq l \leq k-1} \gamma_{kl}
x^g_1\cdots x^g_r y_1^{h_{k-1}} \cdots y_{n-r-1}^{h_{k-1}} y_{n-r}^{e_{kl}} .
\end{align*}
If we evaluate $x_1=\cdots=x_r=g$ and $y_1=\cdots=y_{n-r}=\sum_{a\in\mathcal S} a$, 
we get
\[\sum_{1\leq i \leq k-1} \alpha_i (e_{ii}\pm e_{i+1,i+1})+ \beta e_{22} +\sum_{\substack{1\leq i \leq k-1\\ 1\leq j \leq k\\ i\neq j}} \gamma_{ij} e_{ij}+ \sum_{1\leq l \leq k-1} (-1)^{n-r-1} \gamma_{kl} e_{kl} =0,\]
which produces $\alpha_{i}=\beta=\gamma_{lj}=0$ for all $1\leq i \leq k-1$, $1\leq l,j\leq k$. Therefore the elements in \eqref{relativamente libera Mk} are linearly independent modulo $P^U_{r,n-r}\cap \I^U(\M)$.
	
The fact that $P^U_{r,n-r}\cap \I^U(\M) \supseteq P^U_{r,n-r}\cap I$  for all $n\in\N$ and $0\leq r\leq n$ proves that $\I^U(\M)=I$, with $x_1^g\cdots x_n^g$ and the elements in \eqref{generatori r=n-1}, \eqref{relativamente libera Mk}  
forming a basis of $P^U_{r,n-r}$ modulo $P^U_{r,n-r}\cap \I^U(\M)$ for $r=n$, $r=n-1$ and $0\leq r \leq n-2$, respectively.
Thus, by counting we get
\[c^U_{r,n-r}(\M)=\begin{cases}
1,  & \mbox{ if } r=n, \\
k^2 -1,  & \mbox{ if } r=n-1, \\
k^2, & \mbox{ if } 0\leq r\leq n-2.
\end{cases}\]
Hence, by Formula \eqref{formula(C)} it follows that
\begin{align*}
c_n^U(\M) &= \sum_{r=0}^{n} \binom{n}{r} (k^2 -1)^{n-r} c^U_{r,n-r}(\M)=k^2 \sum_{r=0}^{n} \binom{n}{r} (k^2 -1)^{n-r} -(k^2 -1)n-(k^2 -1) = \\
&= k^{2(n+1)}-(k^2 -1)(n+1).\qedhere
\end{align*}
\end{proof}


\begin{pgraph}\textbf{Modifying the second index of an exponent.}
Through the $L$-action, the second basis element of the subindex of $\varphi_{ab}$ can also be changed in the search for identities, albeit the result is less straightforward. In particular, if $d\in L$ is a derivation then $(x^{ca}y^{cb})^d = x^{(ca)d}y^{cb} + x^{ca}y^{(cb)d}$. We will resort to this method to reduce the number of generators in the basis of $\I^U(\M)$.
\end{pgraph}

We are working in $\FU$, which has no $U$-action (see Remarks \ref{FUhasnoUaction}). Nevertheless, to simplify the notation, instead of writing exponents belonging to $\Uop$ whose action would eventually project to $U^{\op}$ once they landed on isolated variables, we will write the exponents directly in $U^{\op}$ with the caution of evaluating them only on isolated variables. For example, to apply the action of $e_{12}^2\in \Uop$ to $x^uy^v\in\FU$, we write
\[(x^uy^v)^{E_{12}^2} = x^{uE_{12}^2}y^v + 2x^{uE_{12}}y^{vE_{12}} + x^uy^{vE_{12}^2}\]
with $E_{12}\in U^{\op}$, and only then simplify the exponents $uE_{12}^2,uE_{12},vE_{12},vE_{12}^2$ by computing in $U^{\op}$.

\smallskip

\begin{theorem}\label{UidentitiesTHEOREM}
The $T_U$-ideal of $U$-identities of $\M$ is generated by the following $U$-polynomials:
\enum
\item Either {\Large $x^{e_{12}e_{12}}y^{e_{12}e_{12}}$} if $k=2$ or {\Large $x^{e_{12}e_{12}}y^{e_{12}e_{31}}$} if $k\geq3$,
\item {\Large $x^{e_{12}e_{12}}y^{e_{12}e_{21}}-y^{e_{12}e_{12}}x^{e_{12}e_{21}}$},
\item {\Large $[x^{gg},y^{gg}]$},
\item {\Large $[x^{gg},y^{e_{12}e_{12}}]$}.
\eenum
\end{theorem}

\begin{proof}
By Propositions \ref{Theorem Id^U K=2} and \ref{Theorem Id^U K>2}, the $T_U$-ideal of $U$-identities of $\M$ is generated by the following list (L) of identities, for fixed $a\in\mathcal S$:
\enum\label{identitieslist}
\item\label{eidentities} $x^{ae_{ij}}y^{ae_{lm}}$ with $1\leq i\leq k-1$, $2\leq j\leq k$, $j\neq i$.
\item\label{hidentitiesleft} $x^{ah_i}y^{ae_{jl}}$ with $1\leq i \leq k-1$, $1\leq  j, l \leq k$, $ j\neq i, i+1$, $l\neq j$;
\item\label{hidentitiesright} $x^{ae_{jl}}y^{ah_i}$ with $1\leq i \leq k-1$, $1\leq j,l\leq k$, $l\neq i, i+1$, $j\neq l$;
\item\label{hidentitiesboth}  $x^{ah_i}y^{ah_j}$ with $1\leq i,j\leq k-1$, $j\neq i-1,i,i+1$.
\item\label{hidentities2oneright} $ x^{ae_{ij}}y^{ah_{j-1}} + x^{ae_{il}} y^{ae_{lj}}$ with $1\leq i,l\leq k$, $2\leq j \leq k$, $l,j\neq i$, $j\neq l$;
\item\label{hidentities2oneleft} $x^{ah_{i-1}}y^{ae_{ij}} + x^{ae_{il}} y^{ae_{lj}}$ with $2\leq i \leq k$, $1\leq j,l\leq k$, $ l,j\neq i$, $ j\neq l$;
\item\label{hidentities2twoij-1} $x^{ah_i}y^{ae_{ij}}+ x^{ae_{ij}}y^{ah_{j-1}}$ with $ 1\leq i \leq k-1$, $ 2\leq j \leq k$,  $j\neq i$;
\item\label{hidentities2twoi-1j}  $ x^{ah_{i-1}}y^{ae_{ij}}+ x^{ae_{ij}}y^{ah_j}$ with $ 2\leq i \leq k$, $1\leq j \leq k-1$, $ j\neq i$;
\item\label{hidentities2twotogether} $x^{ah_{i-1}} y^{ah_i} + x^{ae_{ij}} y^{ae_{ji}}$ with $2\leq i \leq k-1$, $1\leq j \leq k$, $j\neq i$;
\item\label{hidentities3} $x^{ae_{ij}} y^{ae_{ji}} + x^{ae_{i+1,l}} y^{ae_{l,i+1}}-x^{ah_i} y^{ah_i}$ with $ 1\leq i \leq k-1$, $1\leq j,l\leq k$, $j\neq i$, $l\neq i+1$;
\item\label{identitiespermuted} $x^{ab} y^{ac} - y^{ab} x^{ac}$ with $b,c\in \mathcal{S}$;
\item\label{identitiesbracketsg} $[x^{gg},y^{gg}], \, [x^{gg} y^{ab}]$ with  $b\in \mathcal{S}$;
\item\label{identitiesbracketsh} $[x^{ah_{i}}, y^{ah_{i+1}}]$ with $1\leq i \leq k-2$;
\eenum
\noindent where some identities are not realized for $k=2$ and $k=3$. Let $J$ be the $T_U$-ideal of $\FU$ generated by either $x^{e_{12}e_{12}}y^{e_{12}e_{12}}$ if $k=2$ or $x^{e_{12}e_{12}}y^{e_{12}e_{31}}$ if $k\geq3$, $x^{e_{12}e_{12}}y^{e_{12}e_{21}}-y^{e_{12}e_{12}}x^{e_{12}e_{21}}$, $[x^{gg},y^{gg}]$ and
$[x^{gg},y^{e_{12}e_{12}}]$. Clearly $J\subseteq\I^U(\M)$. In the following, we will prove that $J=\I^U(\M)$ by showing that all identities in the list (L) belong to $J$, mainly by the action of inner derivations. From now on we change, as we may, the first exponent index $a$ of any variable $x^{ab}$ with $b\in\mathcal S$ to $e_{12}$ and elide it by writing $x^b$; we also write $x^g$ for $x^{gg}$.

By the bracket formula \eqref{formula(B)}, the action of the inner derivation $E_{ij}$ (generated by $e_{ij}$) on $x^{e_{ab}}y^{e_{cd}}$ gives
{\Large\[(x^{e_{ab}}y^{e_{cd}})^{E_{ij}}=\delta_{ja}x^{e_{ib}}y^{e_{cd}} -\delta_{bi}x^{e_{aj}}y^{e_{cd}} +\delta_{jc}x^{e_{ab}}y^{e_{id}} -\delta_{di}x^{e_{ab}}y^{e_{cj}}\]}
for $i,j,a,b,c,d\in \{1,\ldots,k\}$, $i\neq j, a\neq b, c\neq d$, where $\delta_{rs}$ denotes Kronecker's delta (see \ref{innerderivations} to review the key computational facts about inner derivations that we will need in the sequel). From now on we assume without further notice that any element of the form $e_{ab}$ or $E_{ab}$ satisfies $a\neq b$ and imposes this restriction wherever it appears.

For $U$-polynomials $f,g\in\FU$ and $u\in\Uop$ let us write $f\overset{u}{\longrightarrow} g$ to denote $g=f^u$; we call this kind of operation a \emph{deduction}, with $f$ its \emph{starting} $U$-polynomial and $g$ its \emph{ending} $U$-polynomial, which is \emph{deduced from} $f$. More in general, we extend the notation and write $f\xrightarrow{u(*)}g$ if some subtractions of other $U$-polynomials also deduced from $f$ are needed in the process of getting $g$ from $f$, in addition to the action of $u$. Notice that if the starting $U$-polynomial of a deduction belongs to a $U$-ideal $I$ of $\FU$, then the ending $U$ñpolynomial of the deduction also belongs to $I$, as $I$ is $\Uop$-invariant. In the following, we will use this fact to show, in sequential steps, that the list of identities (L) belongs to $J$. When in this process we deduce that some $U$-polynomial $g$ is in $J$ because $f\in J$, we say that we \emph{reach} $g$ \emph{starting from $f$}.
\begin{enumerate} [label=\bfseries (D\arabic*)]
\item\textbf{Identities with one $U$-monomial, with exponents of type $e$ and 2 distinct subindices:}\label{deductions1e2}
{\Large\begin{alignat}{1}\label{actions1}
&x^{e_{ab}}y^{e_{ab}}\xrightarrow{\frac1{2}E_{ca}^2} x^{e_{cb}}y^{e_{cb}} \,\, \medmath{(c\neq a)}\nonumber\\
&x^{e_{ab}}y^{e_{ab}}\xrightarrow{\frac1{2}E_{bc}^2} x^{e_{ac}}y^{e_{ac}} \,\, \medmath{(c\neq b)}\\
&x^{e_{ab}}y^{e_{ab}}\xrightarrow{\frac1{24}E_{ba}^4} x^{e_{ba}}y^{e_{ba}}\nonumber
\end{alignat}}%
To aid comprehension of the rest of the computations, let us show the details of the second and third deductions. Recall that the appearance of $e_{ac}$ as the ending $U$-polynomial of the second deduction implies the additional restriction $c\neq a$:
\begin{alignat*}{1}
&x^{e_{ab}}y^{e_{ab}}\xrightarrow{E_{bc} (c\neq a)} -x^{e_{ac}}y^{e_{ab}} -x^{e_{ab}}y^{e_{ac}}\xrightarrow{E_{bc} (c\neq a)} x^{e_{ac}}y^{e_{ac}} + x^{e_{ac}}y^{e_{ac}} = 2x^{e_{ac}}y^{e_{ac}}\,\, \medmath{(c\neq b)}\\
&x^{e_{ab}}y^{e_{ab}}\xrightarrow{E_{ba}^4}\\ &(x^{e_{ab}})^{E_{ba}^4}y^{e_{ab}}+4(x^{e_{ab}})^{E_{ba}^3}(y^{e_{ab}})^{E_{ba}}+6(x^{e_{ab}})^{E_{ba}^2}(y^{e_{ab}})^{E_{ba}^2} +4(x^{e_{ab}})^{E_{ba}}(y^{e_{ab}})^{E_{ba}^3} + (x^{e_{ab}})(y^{e_{ab}})^{E_{ba}^4} = \\
&= 6(x^{e_{ab}})^{E_{ba}^2}(y^{e_{ab}})^{E_{ba}^2} = 24(x^{e_{ab}})^{\vp_{e_{ab}e_{ba}}}(y^{e_{ab}})^{\vp_{e_{ab}e_{ba}}} = 24x^{e_{ba}}y^{e_{ba}}
\end{alignat*}
since $E_{ba}^3=0$ and $E_{ba}^2=-2\vp_{e_{ab}e_{ba}}$.\\
From \eqref{actions1} we deduce that, starting from $x^{e_{12}}y^{e_{12}}\in J$, we can reach any $U$-monomial identity of the form $x^{e_{ab}}y^{e_{ab}}$ with $a,b\in\{1,\ldots,k\}$, for any $k\geq2$:
\begin{itemize}
\item If $a,b\neq1,2$, apply the first and second deductions to reach the target $U$-monomial.
\item If $a=1, b\neq2$, apply the second deduction.
\item If $a\neq1, b=2$, apply the first deduction.
\item If $a=2$ or $b=1$ (or both), apply the third deduction, then the first or second as needed.
\end{itemize}

\smallskip

In the following, we show the deductions of the rest of the identities with less degree of detail. The computations are lengthy but straightforward to check.

\item \textbf{Identities with one $U$-monomial, with exponents of type $e$ and 3 distinct subindices:}\label{deductions1e3}

We assume $a,b,c,d\in\{1,\ldots,k\}$ are pairwise different.
{\Large\begin{alignat}{1}\label{actions2}
&x^{e_{ab}}y^{e_{ca}}\xrightarrow{-\frac12E_{ba}^2} x^{e_{ba}}y^{e_{ca}} \nonumber\\
&x^{e_{ab}}y^{e_{ca}}\xrightarrow{-\frac12E_{ac}^2} x^{e_{ab}}y^{e_{ac}} \nonumber\\
&x^{e_{ab}}y^{e_{ca}}\xrightarrow{-\frac12E_{bc}^2} x^{e_{ac}}y^{e_{ba}} \nonumber\\
&x^{e_{ab}}y^{e_{ca}}\xrightarrow{\frac12E_{ba}E_{ab}E_{ba}(*)} x^{e_{ba}}y^{e_{cb}}\\
&x^{e_{ab}}y^{e_{ca}}\xrightarrow{-E_{da}E_{ad}(*)} x^{e_{db}}y^{e_{cd}}\nonumber\\
&x^{e_{ab}}y^{e_{ca}}\xrightarrow{-E_{bd}} x^{e_{ad}}y^{e_{ca}}\nonumber\\
&x^{e_{ab}}y^{e_{ac}}\xrightarrow{-E_{cb}} x^{e_{ab}}y^{e_{ab}}\nonumber
\end{alignat}}%
Let us show the fifth deduction of \eqref{actions2} explicitly:
\[x^{e_{ab}}y^{e_{ca}}\xrightarrow{E_{da}} x^{e_{db}}y^{e_{ca}} \xrightarrow{-E_{ad}} -x^{e_{ab}}y^{e_{ca}}+x^{e_{db}}y^{e_{cd}};\]
adding $x^{e_{ab}}y^{e_{ca}}$ (which is the starting $U$-monomial of the deduction) we deduce $x^{e_{db}}y^{e_{cd}}$.\\
Now we explicit the fourth deduction:
\[x^{e_{ab}}y^{e_{ca}}\xrightarrow{-E_{ba}} x^{h_{ab}}y^{e_{ca}} \, \text{(a)} \xrightarrow{-E_{ab}} 2x^{e_{ab}}y^{e_{ca}} + x^{h_{ab}}y^{e_{cb}};\]
subtracting $2x^{e_{ab}}y^{e_{ca}}$ (twice the starting $U$-monomial of the deduction) we deduce $x^{h_{ab}}y^{e_{cb}}$ and continue with
\[x^{h_{ab}}y^{e_{cb}}\xrightarrow{\frac12E_{ba}} x^{e_{ba}}y^{e_{cb}}+\frac12x^{h_{ab}}y^{e_{ca}};\]
subtracting now $\frac12x^{h_{ab}}y^{e_{ca}}$ (which we deduced in (a)) we finally deduce $x^{e_{ba}}y^{e_{cb}}$.\\
From \eqref{actions3} we deduce that, starting from $x^{e_{12}}y^{e_{31}}\in J$, we can reach any $U$-monomial identity of the forms $x^{e_{ab}}y^{e_{ca}}$, $x^{e_{ab}}y^{e_{ac}}$, $x^{e_{ba}}y^{e_{ca}}$ with $a,b,c\in\{1,\ldots,k\}$ pairwise different, for any $k\geq3$: deductions 1-4 show that we can get any valid permutation of the indices involved without introducing new ones (a crucial property for $k=3$), deductions 5-6 allow to change the indices to others not appearing in the original identity, and deduction 7 (together with deduction 2) shows that we can reach $x^{e_{12}}y^{e_{12}}$ and thus all identities with only two distinct indices (by \ref{deductions1e2}), whence $x^{e_{12}}y^{e_{12}}$ is superfluous in the basis of identities for $k\geq3$.

\item \textbf{Identities with one $U$-monomial, with exponents of type $e$ and 4 distinct subindices:}\label{deductions1e4}

We assume $a,b,c,d,i\in\{1,\ldots,k\}$ are pairwise different.
{\Large\begin{alignat}{1}\label{actions3}
&x^{e_{ab}}y^{e_{ca}}\xrightarrow{-E_{ad}} x^{e_{ab}}y^{e_{cd}}\nonumber\\
&x^{e_{ab}}y^{e_{cd}}\xrightarrow{-\frac12E_{ba}^2} x^{e_{ba}}y^{e_{cd}}\nonumber\\
&x^{e_{ab}}y^{e_{cd}}\xrightarrow{-\frac12E_{dc}^2} x^{e_{ab}}y^{e_{dc}}\\
&x^{e_{ab}}y^{e_{cd}}\xrightarrow{-\frac12E_{bc}^2} x^{e_{ac}}y^{e_{bd}} \nonumber\\
&x^{e_{ab}}y^{e_{cd}}\xrightarrow{E_{ia}} x^{e_{ib}}y^{e_{cd}}\nonumber
\end{alignat}}%
From \eqref{actions3} we deduce that, starting from $x^{e_{12}}y^{e_{31}}\in J$, we can reach any $U$-monomial identity of the form $x^{e_{ab}}y^{e_{cd}}$ with $a,b,c,d\in\{1,\ldots,k\}$ all distinct: the first deduction allows to reach four distinct subindices, deductions 2-4 generate all permutations of the 4 subindices (through 3 transpositions) without the need of introducing new subindices (a crucial property for $k=4$), and the last deduction allows the introduction of new subindices.\\
Through \ref{deductions1e2},\ref{deductions1e3},\ref{deductions1e4} we have shown that identity \ref{eidentities} of list (L) is in $J$.

\item\textbf{Identities with one $U$-monomial, with exponents of type $h$:}\label{deductions1h}
{\Large\begin{alignat}{1}\label{actions4}
&x^{e_{ii+1}}y^{e_{ab}}\xrightarrow{-E_{i+1i}} x^{h_i}y^{e_{ab}} \,\, \medmath{(a\neq i; b\neq i+1)}\nonumber\\
&x^{e_{ii+1}}y^{e_{ai+1}}\xrightarrow{-E_{i+1i}} x^{h_i}y^{e_{ai+1}} +x^{e_{ii+1}}y^{e_{ai}} \equi J x^{h_i}y^{e_{ai+1}} \,\, \medmath{(a\neq i,i+1)}\nonumber\\
&x^{e_{ab}}y^{e_{ii+1}}\xrightarrow{-E_{i+1i}} x^{e_{ab}}y^{h_i} \,\, \medmath{(a\neq i; b\neq i+1)}\\
&x^{e_{ib}}y^{e_{ii+1}}\xrightarrow{E_{i+1i}} x^{e_{i+1b}}y^{e_{ii+1}}+x^{e_{ib}}y^{h_i}\equi J x^{e_{ib}}y^{h_i} \,\, \medmath{(b\neq i,i+1)}\nonumber\\
&x^{h_i}y^{e_{jj+1}}\xrightarrow{-E_{j+1j}} x^{h_i}y^{h_j} \,\, \medmath{(i\neq j-1,j,j+1)}\nonumber
\end{alignat}}%
In the second and fourth deductions we have used \ref{deductions1e3}.\\
From \eqref{actions4} we deduce that, starting from $x^{e_{12}}y^{e_{31}}\in J$, we can reach any $U$-monomial identity of the forms $x^{h_i}y^{e_{ab}}$ ($a\neq i,i+1$), $x^{e_{ab}}y^{h_i}$ ($b\neq i,i+1$), $x^{h_i}y^{h_j}$ ($i\neq j-1,j,j+1$), with $a,b\in\{1,\ldots,k\}$, $i,j\in\{1,\ldots,k-1\}$: by \ref{deductions1e4} we can reach the starting $U$-monomials of deductions 1-4 (with the corresponding restrictions on $a,b,i$), so we can reach their ending $U$-monomials; then the ending $U$-monomials of deductions 1-2 are the starting $U$-monomials of deduction 5. Therefore identities \ref{hidentitiesleft},\ref{hidentitiesright},\ref{hidentitiesboth} of list (L) are in $J$.

\item \textbf{Auxiliary identities:}\label{deductions2auxiliary}
{\Large\begin{alignat}{1}\label{actions5}
&x^{e_{ij}}y^{e_{li}}\xrightarrow{-E_{jl}} x^{e_{il}}y^{e_{li}}-x^{e_{ij}}y^{e_{ji}} \,\, \medmath{(l\neq j)}\nonumber\\
&x^{e_{il}}y^{h_j}\xrightarrow{-E_{lj}} x^{e_{ij}}y^{h_j}-x^{e_{il}}y^{e_{lj}} \,\, \medmath{(l\neq j+1)}\nonumber\\
&x^{e_{ij+1}}y^{e_{jj+1}}\xrightarrow{\frac12E_{j+1j}^2} x^{e_{ij}}y^{h_j}-x^{e_{ij+1}}y^{e_{j+1j}} \\
&x^{h_i}y^{e_{lj}}\xrightarrow{E_{il}} x^{h_i}y^{e_{ij}}-x^{e_{il}}y^{e_{lj}} \,\, \medmath{(l\neq i+1)}\nonumber\\
&x^{e_{i+1i}}y^{e_{i+1j}}\xrightarrow{\frac12E_{ii+1}^2} x^{h_i}y^{e_{ij}}-x^{e_{ii+1}}y^{e_{i+1j}}\nonumber
\end{alignat}}%
Since $x^{e_{ij}}y^{e_{li}}$ is an identity in $J$ for $l\neq j$ (by \ref{deductions1e3}), by deduction 1 we have
{\Large\[x^{e_{il}}y^{e_{li}}\equi J x^{e_{ij}}y^{e_{ji}} \tag{A1}\label{formula(A1)}\]}
for all valid $1\leq i,j,l\leq k$.
Since $x^{e_{il}}y^{h_j}$ ($l\neq j,j+1$) and $x^{e_{ij+1}}y^{e_{jj+1}}$ are identities in $J$ (by \ref{deductions1e3}, \ref{deductions1h}), by deductions 2-3 we get
{\Large\[x^{e_{ij}}y^{h_j}\equi J x^{e_{il}}y^{e_{lj}} \tag{A2}\label{formula(A2)}\]}
for all valid $1\leq i,j,l\leq k$.
Analogously, from deductions 4-5, we get
{\Large\[x^{h_i}y^{e_{ij}}\equi J x^{e_{il}}y^{e_{lj}}. \tag{A3}\label{formula(A3)}\]}

\item \textbf{Identities in 2 or 3 $U$-monomials with variables in the same order:}\label{deductions2same}
{\Large\begin{alignat}{1}\label{actions6}
&x^{e_{il}}y^{h_{j-1}}\xrightarrow{-E_{lj}} x^{e_{ij}}y^{h_{j-1}}+x^{e_{il}}y^{e_{lj}} \,\, \medmath{(l\neq j-1)}\nonumber\\
&x^{e_{ij-1}}y^{e_{jj-1}}\xrightarrow{-\frac12E_{j-1j}^2} x^{e_{ij}}y^{h_{j-1}}+x^{e_{ij-1}}y^{e_{j-1j}} \,\, \nonumber\\
&x^{h_{i-1}}y^{e_{lj}}\xrightarrow{E_{il}} x^{h_{i-1}}y^{e_{ij}}+x^{e_{il}}y^{e_{lj}} \,\, \medmath{(l\neq i-1)}\\
&x^{e_{i-1i}}y^{e_{i-1j}}\xrightarrow{-\frac12E_{ii-1}^2} x^{h_{i-1}}y^{e_{ij}}+x^{e_{ii-1}}y^{e_{i-1j}} \,\, \nonumber
\end{alignat}}%
Since $x^{e_{il}}y^{h_{j-1}}$ ($l\neq j-1,j$) and $x^{e_{ij-1}}y^{e_{jj-1}}$ are identities in $J$ (by \ref{deductions1e3}, \ref{deductions1h}), deductions 1-2 give identity \ref{hidentities2oneright} from list (L). Analogously, deductions 3-4 give identity \ref{hidentities2oneleft} from list (L). Now identities \ref{hidentities2twoij-1},\ref{hidentities2twoi-1j} from list (L) are linear combinations of \ref{hidentities2oneright},\ref{hidentities2oneleft} and auxiliary identities \eqref{formula(A2)},\eqref{formula(A3)}:
{\Large\begin{align*}
&x^{h_i}y^{e_{ij}}+x^{e_{ij}}y^{h_{j-1}} = (x^{h_i}y^{e_{ij}}-x^{e_{il}}y^{e_{lj}})+(x^{e_{ij}}y^{h_{j-1}}+x^{e_{il}}y^{e_{lj}}),\\
& x^{h_{i-1}}y^{e_{ij}} + x^{e_{ij}}y^{h_j} = (x^{h_{i-1}}y^{e_{ij}}+x^{e_{il}}y^{e_{lj}})+(x^{e_{ij}}y^{h_j}-x^{e_{il}}y^{e_{lj}}).
\end{align*}}%
In addition, by auxiliary identity \eqref{formula(A1)} we find
\[x^{h_{i-1}}y^{e_{ij}}+x^{e_{il}}y^{e_{lj}}\xrightarrow{-E_{ji}} x^{h_{i-1}}y^{h_{ij}}+x^{e_{il}}y^{e_{li}}-x^{e_{jl}}y^{e_{lj}} +x^{e_{ji}}y^{e_{ij}}\equi J x^{h_{i-1}}y^{h_{ij}}+x^{e_{il}}y^{e_{li}} \,\, \medmath{(j\neq i-1)}.\]
Since $x^{h_{i-1}}y^{e_{ij}}+x^{e_{il}}y^{e_{lj}}$ is an identity in $J$ for all $1\leq i,j,l\leq k$, particularizing for $j=i+1$ and changing $l$ to $j$ we get the identity
{\Large\[x^{h_{i-1}}y^{h_i}+x^{e_{ij}}y^{e_{ji}}\]}%
in $J$, which is identity \ref{hidentities2twotogether} in list (L).\\
Finally, starting from \eqref{formula(A3)} particularized with $j=i+1$ and changing $l$ to $j$ we get
\[x^{h_i}y^{e_{ii+1}}- x^{e_{ij}}y^{e_{ji+1}}\xrightarrow{-E_{i+1i}} x^{h_i}y^{h_i}-2x^{e_{i+1i}}y^{e_{ii+1}} +x^{e_{i+1j}}y^{e_{ji+1}} -x^{e_{ij}}y^{e_{ji}} \equi J x^{h_i}y^{h_i}-x^{e_{i+1l}}y^{e_{li+1}}-x^{e_{ij}}y^{e_{ji}},\]
since $x^{e_{i+1i}}y^{e_{ii+1}}\equi J x^{e_{i+1j}}y^{e_{ji+1}}\equi J x^{e_{i+1l}}y^{e_{li+1}}$ by \eqref{formula(A1)}. Hence identity \ref{hidentities3} from list (L) is in $J$.\\
All identities in this subsection have been reached starting from $x^{e_{12}}y^{e_{31}}$.

\item\textbf{Identities in 2 $U$-monomials with permuted variables:}\label{deductions2permuted}\\
If $x^ay^b\in\I(\M)$ with $a,b\in\mathcal S$ then $x^ay^b\in J$  by \ref{deductions1e2}-\ref{deductions1h} and hence $x^ay^b-y^ax^b\in J$, reached starting from $x^{e_{12}}y^{e_{12}}$ if $k=2$ and from $x^{e_{12}}y^{e_{31}}$ if $k\geq3$. Let us see that the rest of the identities of the form $x^ay^b-y^ax^b$ with $a,b\in\mathcal S$ are reached from $x^{e_{12}}y^{e_{21}}-y^{e_{12}}x^{e_{21}}$. Those of the form $x^{e_{ab}}y^{e_{bc}}-y^{e_{ab}}x^{e_{bc}}$ are reached just from $x^{e_{12}}y^{e_{21}}-y^{e_{12}}x^{e_{21}}$ by the application of the following deductions, similar to those in \ref{deductions1e2}. We assume $a,b,c\in\{1,\ldots,k\}$ are pairwise different:
{\Large\begin{alignat}{1}\label{actions7}
&x^{e_{ab}}y^{e_{ba}}\xrightarrow{-E_{ac}} x^{e_{ab}}y^{e_{bc}}\nonumber\\
&x^{e_{ab}}y^{e_{ba}}\xrightarrow{E_{ca}} x^{e_{cb}}y^{e_{ba}}\nonumber\\
&x^{e_{ab}}y^{e_{ba}}\xrightarrow{-E_{ca}E_{ac}(*)} x^{e_{cb}}y^{e_{bc}}\\
&x^{e_{ab}}y^{e_{ba}}\xrightarrow{\frac14E_{ba}E_{ab}E_{ba}E_{ab}(*)} x^{e_{ba}}y^{e_{ab}}\nonumber
\end{alignat}}%
E.g., we reach $x^{e_{21}}y^{e_{12}}-y^{e_{21}}x^{e_{12}}\in J$ applying deduction 4 to $x^{e_{12}}y^{e_{21}}-y^{e_{12}}x^{e_{21}}$ (by linearity), etc.
Let us explicit deduction 4:
\[x^{e_{ab}}y^{e_{ba}}\xrightarrow{-E_{ba}} x^{h_{ab}}y^{e_{ba}} \, \text{(a)} \xrightarrow{E_{ab}} -2x^{e_{ab}}y^{e_{ba}} + x^{h_{ab}}y^{h_{ab}};\]
adding $2x^{e_{ab}}y^{e_{ba}}$ (twice the starting $U$-monomial of the deduction) we deduce $x^{h_{ab}}y^{h_{ab}} \, \text{(b)}$ and continue with
\[x^{h_{ab}}y^{h_{ab}}\xrightarrow{\frac12E_{ba}} x^{e_{ba}}y^{h_{ab}}+x^{h_{ab}}y^{e_{ba}};\]
subtracting $x^{h_{ab}}y^{e_{ba}}$ (which we deduced in (a)) we deduce $x^{e_{ba}}y^{h_{ab}}$ and continue with
\[x^{e_{ba}}y^{h_{ab}}\xrightarrow{-\frac12E_{ab}} -\frac12x^{h_{ab}}y^{h_{ab}}+x^{e_{ba}}y^{e_{ab}};\]
adding now $\frac12x^{h_{ab}}y^{h_{ab}}$ (which we deduced in (b)) we finally deduce $x^{e_{ba}}y^{e_{ab}}$.

To reach $x^{e_{ij}}y^{h_j}-y^{e_{ij}}x^{h_j}$ and $x^{h_i}y^{e_{ij}}-y^{h_i}x^{e_{ij}}$, we respectively apply $x^{e_{ij}}y^{h_j}\equi J x^{e_{il}}y^{e_{lj}}\equi J x^{h_i}y^{e_{ij}}$ from auxiliary identities \eqref{formula(A2)},\eqref{formula(A3)} (reached in \ref{deductions2auxiliary}), then deduce $x^{e_{il}}y^{e_{lj}}$ from $x^{e_{12}}y^{e_{21}}$ as done in \eqref{actions7}.
To reach $x^{e_{ij}}y^{h_{j-1}}-y^{e_{ij}}x^{h_{j-1}}$ and $x^{h_{i-1}}y^{e_{ij}}-y^{h_{i-1}}x^{e_{ij}}$ we respectively apply $x^{e_{ij}}y^{h_{j-1}}\equi J -x^{h_i}y^{e_{ij}}$ and $x^{h_{i-1}}y^{e_{ij}}\equi J -x^{e_{ij}}y^{h_j}$ from identities \ref{hidentities2twoij-1},\ref{hidentities2twoi-1j} of list (L).
To reach $x^{h_{i-1}}y^{h_i}-y^{h_{i-1}}x^{h_i}$ we apply $x^{h_{i-1}}y^{h_i}\equi J -x^{e_{ij}}y^{e_{ji}}$ by identity \ref{hidentities2twotogether} of list (L), then deductions \eqref{actions7}.
To reach $x^{h_i}y^{h_i}-y^{h_i}x^{h_i}$ we apply $x^{h_i}y^{h_i}\equi J x^{e_{ij}}y^{e_{ji}}+x^{e_{i+1l}}y^{e_{li+1}}$ by identity \ref{hidentities3} of list (L).
Finally, starting from identity \ref{hidentities2oneright} of list (L) we reach
\[x^{e_{i+1i}}y^{h_{i-1}}+x^{e_{i+1l}}y^{e_{li}} \xrightarrow{E_{ii+1}} x^{h_i}y^{h_{i-1}} + x^{e_{i+1i}}y^{e_{ii+1}}+x^{e_{il}}y^{e_{li}} - x^{e_{i+1l}}y^{e_{li+1}} \equi J x^{h_i}y^{h_{i-1}} + x^{e_{il}}y^{e_{li}},\]
since $x^{e_{i+1i}}y^{e_{ii+1}} \equi J x^{e_{i+1l}}y^{e_{li+1}}$ by auxiliary identity \eqref{formula(A3)}. Therefore
\[x^{h_i}y^{h_{i-1}} \equi J -x^{e_{il}}y^{e_{li}} \tag{A4}\label{formula(A4)}\] and so
$x^{h_i}y^{h_{i-1}}-y^{h_i}x^{h_{i-1}}\in J$ by deductions \eqref{actions7}.\\
In this subsection, we have shown that identity \ref{identitiespermuted} is in $J$.

\item \textbf{Identities with brackets:}\label{deductions2brackets}\\
$U$-polynomial $[x^g,y^g]$ is already a generator of $J$. To reach $[x^g,y^a]$ with $a\in\mathcal S$ starting from $[x^g,y^{e_{12}}]$ we apply the following deductions. We assume $a,b,c\in\{1,\ldots,k\}$ are pairwise different:
{\Large\begin{alignat}{1}\label{actions8}
&x^gy^{e_{ab}}\xrightarrow{E_{ca}} x^gy^{e_{cb}} \nonumber\\
&x^gy^{e_{ab}}\xrightarrow{-E_{bc}} x^gy^{e_{ac}}\nonumber\\
&x^gy^{e_{ab}}\xrightarrow{-\frac12E_{ba}^2} x^gy^{e_{ba}}\\
&x^gy^{e_{ii+1}}\xrightarrow{-E_{i+1i}} x^gy^{h_i} \nonumber
\end{alignat}}%
In addition, by the previous subsection,
\[x^{h_{i-1}}y^{h_i}\equi J y^{h_{i-1}}x^{h_i} \equiv_J y^{h_i}x^{h_{i-1}},\]
since by identity \eqref{formula(A4)} and identity (11) of the list (L) we get
\[x^{h_{i-1}}y^{h_i}\equi J -x^{e_{il}}y^{e_{li}} \equi J x^{h_i}y^{h_{i-1}}. \tag{A5}\label{formula(A5)}\]
Therefore $[x^{h_{i-1}},y^{h_i}]\in J$. With this we have shown that identities \ref{identitiesbracketsg},\ref{identitiesbracketsh} of list (L) are in $J$.\qedhere
\end{enumerate}
\end{proof}

\begin{theorem}\label{minimalUgeneratorsTHEOREM}
The set of generators of $\I^U(\M)$ formed by
\[x^{e_{12}e_{12}}y^{e_{12}e}, \, x^{e_{12}e_{12}}y^{e_{12}e_{21}}-y^{e_{12}e_{12}}x^{e_{12}e_{21}}, \, [x^{gg},y^{gg}], \, [x^{gg},y^{e_{12}e_{12}}]\]
with $e:=e_{12}$ if $k=2$ and $e:=e_{31}$ if $k\geq3$, is minimal.
\end{theorem}

\begin{proof}
Let us show first that the $U$-polynomials $[x^{gg},y^{gg}]$ and $[x^{gg},y^{e_{12}e_{12}}]$ cannot be generated by nor help in the generation of the other three identities. Suppose $f\in P_2^U$ is of the form $f=f_{gg}+f_{ge_{12}}+f_{e_{12}e_{12}}$ with $f_{gg}\in\la [x^{gg},y^{gg}]\ra_{T_U}$, $f_{ge_{12}}\in\la [x^{gg},y^{e_{12}e_{12}}]\ra_{T_U}$ and $f_{e_{12}e_{12}}$ and $f$ in the ideal generated by the remaining two $U$-identities. Then by substitution (see \ref{firstsuperindex}, \ref{components}), since $\vp_{gg}\vp_{e_{12}e_{12}}=0$ while $\vp_{gg}^2=\vp_{gg}$ and $\vp_{e_{12}e_{12}}^2 = \vp_{e_{12}e_{12}}$, by Formula \eqref{formula(TU)} we get $0=f(x^{gg},y^{gg}) = f_{gg}$ and $0=f(x^{gg},y^{e_{12}e_{12}})=f_{ge_{12}}$, forcing $f=f_{e_{12}e_{12}}$. Analogously, if $f\in \la [x^{gg},y^{gg}]\ra_{T_U}$ then $f_{ge_{12}} = 0 = f_{e_{12}e_{12}}$, $f = f_{gg}$, and if $f\in\la [x^{gg},y^{e_{12}e_{12}}]\ra_{T_U}$ then $f_{gg} = 0 = f_{e_{12}e_{12}}$, $f = f_{ge_{12}}$.

Therefore it remains to show that $x^{e_{12}e_{12}}y^{e_{12}e}$ is not in the $T_U$-ideal generated by $x^{e_{12}e_{12}}y^{e_{12}e_{21}}-y^{e_{12}e_{12}}x^{e_{12}e_{21}}$ and vice versa. To this end, we first show  the general form of multilinear $U$-polynomials in variables $x_1,x_2$ belonging to the $T_U$-ideal of a multilinear $U$-polynomial $f$ in variables $x_1,x_2$ and without $g$-exponents. By Formula \eqref{formula(TU)}, any $U$-polynomial in $\la f\ra_{T_U}$ is of the form $G:=\sum_{i=1}^m g_if(p_{i1},p_{i2})^{u_i}g'_i$, with the $i$th substitution $L$-endomorphism sending, for $j\in\{1,2\}$, $x_j$ to $p_{ij}:=\sum_{a,b\in\mathcal S\text{ or }a=g=b}(\alpha^{ab}_{ij} x_1^{ab} +\beta^{ab}_{ij} x_2^{ab}) + q_{ij}$, with $\alpha^{ab}_{ij},\beta^{ab}_{ij}\in F$ and $q_{ij}\in\FU$ such that $\deg q_{ij}>1$. Write
\[f=\sum_{\sigma\in S_2;r,s,t,u\in\mathcal S} \alpha_{rstu}^{\sigma}\, x_{\sigma(1)}^{rs}x_{\sigma(2)}^{tu}\]
with $\alpha_{rstu}^{\sigma}\in F$. 
Then $f(p_{i1},p_{i2}) = \sum_{\sigma\in S_2;a,b,c,d\in\mathcal S}\alpha_{abcd}^{\sigma,i} f(x_{\sigma(1)}^{ab},x_{\sigma(2)}^{cd}) + q_i$, where $\alpha_{abcd}^{\sigma,i}\in F$ and $q_i\in\FU$ is of degree greater than $2$ or not multilinear, and
\begin{align*}
G = & \sum_{i=1}^m g_i\left(\sum_{\sigma\in S_2;a,b,c,d\in\mathcal S}\alpha_{abcd}^{\sigma,i} f(x_{\sigma(1)}^{ab},x_{\sigma(2)}^{cd}) + q_i\right)^{u_i}g'_i = \\
= &\sum_{1\leq i\leq m;\sigma\in S_2;a,b,c,d\in\mathcal S}\alpha_{abcd}^{\sigma,i} g_if(x_{\sigma(1)}^{ab},x_{\sigma(2)}^{cd})^{u_i}g'_i + g_iq_i^{u_i}g'_i.
\end{align*}
Now, due to the grading of $\FU$ and the fact that $g_iq_i^{u_i}g'_i$ is of degree greater than $2$ or not multilinear, when $G$ is multilinear in $x_1,x_2$ it must equal the sum of its multilinear terms in $x_1,x_2$ (the rest of terms must cancel out). Therefore, reordering if necessary, there are some $1\leq i\leq m'$ such that $g_i=1=g'_i$ and 
\[G = \sum_{1\leq i\leq m';\sigma\in S_2;a,b,c,d\in\mathcal S}\alpha_{abcd}^{\sigma,i} f(x_{\sigma(1)}^{ab},x_{\sigma(2)}^{cd})^{u_i} = \sum_{\sigma\in S_2;a,b,c,d\in\mathcal S} f(x_{\sigma(1)}^{ab},x_{\sigma(2)}^{cd})^{u_{abcd}^\sigma}\]
(with $u_{abcd}^\sigma:=\sum_{1\leq i\leq m'} \alpha_{abcd}^{\sigma,i}u_i$).
Hence, any multilinear element of $\la f\ra_{T_U}$ in $x_1,x_2$ is of the form
\[F(x_1,x_2):=\sum_{\sigma\in S_2;a,b,c,d\in\mathcal S; u\in\Uop}\alpha_{a,b,c,d,u}^{\sigma} f(x_{\sigma(1)}^{ab},x_{\sigma(2)}^{cd})^u,\]
where $\alpha_{a,b,c,d,u}^\sigma\in F$ and only a finite number of these coefficients is nonzero.
Now put
\[f_1(x,y):=x^{e_{12}e_{12}}y^{e_{12}e}\text{ and }f_2(x,y):=x^{e_{12}e_{12}}y^{e_{12}e_{21}}-y^{e_{12}e_{12}}x^{e_{12}e_{21}}\]
in $\FU$; then all $U$-polynomials in $\la f_1\ra_{T_U}$ (resp. $\la f_2\ra_{T_U}$) are of the form $F_1$ (resp. $F_2$).\\
Let us show that $f_1\not\in\la f_2\ra_{T_U}$. If $f_1\in\la f_2\ra_{T_U}$, then for some $F_2$ we have
\[f_1(x_1,x_2)=F_2(x_1,x_2)=\sum_{\sigma\in S_2; a,b\in\mathcal S; u\in\Uop}\!\!\!\!\alpha_{a,b,u}^{\sigma} (x_{\sigma(1)}^{ae_{12}}x_{\sigma(2)}^{be_{21}} -x_{\sigma(2)}^{be_{12}}x_{\sigma(1)}^{ae_{21}})^u,\]
 which is a skew-symmetric $U$-polynomial, so we have $f_1(x_1,x_2)=F_2(x_1,x_2)=-F_2(x_2,x_1)=-f_1(x_2,x_1)$, a contradiction since $f_1(x_1,x_2),f_1(x_2,x_1)$ are different elements of the basis of $\FU$.\\
To show that $f_2\not\in\la f_1\ra_{T_U}$ we consider the $(L,U)$-algebra $\Mat_2(F)\otimes_F \M$ with Hopf $U(L)$-action given by $(a\otimes b)^u:=a\otimes b^u$ for $u\in\Uop$, $a\in\Mat_2(F)$, $b\in\M$.
The evaluation at $u_1,u_2\in M_2(F)\otimes\M$ of a $U$-polynomial $F_1$ takes the form
\[F_1(u_1,u_2)=\sum_{\sigma\in S_2;a,b,c,d\in\mathcal S\text{ or }c=g=d}\!\!\!\!\alpha_{a,b,c,d}^{\sigma} (u_{\sigma(1)}^{ae_{12}}u_{\sigma(2)}^{be})^{cd}.\]
On the one hand, we have
\begin{align*}
&f_2(e_{12}\otimes e_{12},e_{21}\otimes e_{12})=\\
=&(e_{12}\otimes e_{12})^{e_{12}e_{12}}(e_{21}\otimes e_{12})^{e_{12}e_{21}}- (e_{21}\otimes e_{12})^{e_{12}e_{12}}(e_{12}\otimes e_{12})^{e_{12}e_{21}} =\\
=& (e_{12}\otimes e_{12}^{e_{12}e_{12}})(e_{21}\otimes e_{12}^{e_{12}e_{21}})- (e_{21}\otimes e_{12}^{e_{12}e_{12}})(e_{12}\otimes e_{12}^{e_{12}e_{21}}) =\\
=& (e_{12}\otimes e_{12})(e_{21}\otimes e_{21})-(e_{21}\otimes e_{12})(e_{12}\otimes e_{21}) = e_{11}\otimes e_{11} -e_{22}\otimes e_{11} = \\
=& h_1 \otimes e_{11}\neq0.
\end{align*}
On the other hand, we get
\begin{align*}
&F_1(e_{12}\otimes e_{12},e_{21}\otimes e_{12}) = \\
=&\sum_{\sigma\in S_2;a,b,c,d\in\mathcal S\text{ or }c=g=d}\!\!\!\!\alpha_{a,b,c,d}^{\sigma} ((e_{\sigma(1)\sigma(2)}\otimes e_{12})^{ae_{12}}(e_{\sigma(2)\sigma(1)}\otimes e_{12})^{be})^{cd} = \\
=&\sum_{\sigma\in S_2;a,b,c,d\in\mathcal S\text{ or }c=g=d}\!\!\!\!\alpha_{a,b,c,d}^{\sigma} (e_{\sigma(1)\sigma(2)}\otimes e_{12}^{ae_{12}})(e_{\sigma(2)\sigma(1)}\otimes e_{12}^{be})^{cd} = \\
=&\sum_{\sigma\in S_2;c,d\in\mathcal S\text{ or }c=g=d}\!\!\!\!\alpha_{e_{12},e_{12},c,d}^{\sigma} ((e_{\sigma(1)\sigma(2)}\otimes e_{12})(e_{\sigma(2)\sigma(1)}\otimes e))^{cd} = \\
=&\sum_{\sigma\in S_2;c,d\in\mathcal S\text{ or }c=g=d}\!\!\!\!\alpha_{e_{12},e_{12},c,d}^{\sigma} (e_{\sigma(1)\sigma(1)}\otimes e_{12}e)^{cd} = 0,
\end{align*}
as $e_{12}e_{12}=0=e_{12}e_{31}$.
Since $f_2(e_{12}\otimes e_{12},e_{21}\otimes e_{12})\neq F_1(e_{12}\otimes e_{12},e_{21}\otimes e_{12})$ for any $F_1\in\la f_1\ra_{T_U}$, we conclude that $f_2\neq F_1$ for any $F_1\in\la f_1\ra_{T_U}$, hence $f_2\not\in\la f_1\ra_{T_U}$.
\end{proof}

\begin{remark}[\textbf{Generators of $\I^U(\M)$ in the unital case}]
\mbox{}

In this remark we consider $\M$ in the variety of unital associative $(L,U)$-algebras and take $\FU$ to be unital. We have to take into account the $L$-endomorphisms of the form $x_i\mapsto \lambda$ for fixed $i$ and $\lambda\in F\cdot1$, which allow to reduce the degree of $U$-polynomials with factors of the form $x_i^{gg}$, since $\vp_{gg} = 1 - s$ with $s:=\sum_{a\in\mathcal S}\vp_{aa}$ in $U$ implies $\lambda^{gg}=\lambda$ in $\FU$, as $s$ is generated by derivations and $\lambda\in F\cdot1$. These \emph{reducing substitutions} do not affect the result of Theorem \ref{minimalUgeneratorsTHEOREM}, since they amount to the fact of allowing substitutions with $U$-polynomials with nonzero constant terms in the general expression given by Formula \eqref{formula(TU)}, and since in the bulk of the proof of Theorem \ref{minimalUgeneratorsTHEOREM} we only have to consider $U$-polynomials $f$ with no $g$-exponents, such constant terms get killed by the derivations.

On the other hand, the reducing substitutions allow to show that $\I^U(\M)$ is principal as a $T_U$-ideal, at the expense of using three variables, for it is generated by the $U$-identity
\[ x^gy^{e_{12}}z^e + x^{e_{12}}y^{e_{21}}z^g -y^{e_{12}}x^{e_{21}}z^g +[x^g,z^g]y^g + [x^g,z^{e_{12}}] y^g\]
(we get $y^{e_{12}}z^e$ from $x\mapsto 1$, $x^{e_{12}}y^{e_{21}} -y^{e_{12}}x^{e_{21}}$ from $z\mapsto 1$, and $[x^g,z^g], [x^g,z^{e_{12}}]$ from $y\mapsto 1$ then $z\mapsto z^g$ and $z\mapsto z^{e_{12}}$, respectively).
\end{remark}

\smallskip

\subsection{$(L,U)$-identities of $\M$}
\mbox{}

\begin{theorem}\label{(L,U)identitiesTHEOREM}
The $T_{L,U}$-ideal of $(L,U)$-identities of $\M$ is generated by the following identities:
\enum
\item {\Large $[x^{gg},y^{gg}]$},
\item {\Large $[x^{gg},y^{e_{12}e_{12}}]$},
\item {\Large $x^{e_{12}e_{12}}y^{e_{12}e_{21}}-y^{e_{12}e_{12}}x^{e_{12}e_{21}}$},
\item {\Large $x^{e_{12}e_{12}}y^{e_{12}e_{31}}$} if $k\geq3$.
\eenum
Moreover $\I^{L,U}(\M)$ is principal, generated as a $T_{L,U}$-ideal by the sum of identities (1)-(4).
\end{theorem}

\begin{proof}
The first assertion is a direct consequence of Theorem \ref{minimalUgeneratorsTHEOREM}, which provides generators of $\I^U(\M)$ and hence of $\I^{L,U}(\M)$, with the caveat that if $k=2$, the $(L,U)$-polynomial $x^{e_{12}e_{12}}y^{e_{12}e_{12}}$ is redundant because it is $0$ in $\FLU$: we can write
\[\large x^{e_{12}e_{12}}y^{e_{12}e_{12}} = \frac14 x^{(e_{12}e_{21})E_{12}^2}y^{(e_{12}e_{21})E_{12}^2} = \left(\frac1{24}x^{e_{12}e_{21}}y^{e_{12}e_{21}}\right)^{E_{12}^4}=0\]
because $E_{12}^3=0$.

The second assertion is a consequence of the primitive element lemma (\ref{primitiveelementLEMMA}): By the bracket formula \eqref{formula(B)}, since $H_1,\ldots,H_{k-1}\in U^{\op}$ are derivations, for $a,b,c,d\in S$ or $a=g=b$ or $c=g=d$, with $b,d$ weight vectors of $\M$ for the adjoint action of $L$ of respective weights $\alpha_b,\alpha_d$, we get 
\[(x^{ab}y^{cd})^{H_i} = x^{a[h_i,b]}y^{cd} + x^{ab}y^{c[h_i,d]} = (\alpha_b(h_i)+\alpha_d(h_i))x^{ab}y^{cd},\]
which shows that $x^{ab}y^{cd}$ is a weight vector of $\FLU$ for the $L$-action, of weight $\alpha_b+\alpha_d$. Now, the elements $g,e_{12},e_{21},e_{31}\in\M$ are weight vectors of $\M$ for the adjoint action of $L$, with respective weights $0,\alpha,-\alpha,\beta$.
Hence the $(L,U)$-polynomials (1)-(4) in the statement are weight vectors of $\FLU$, with (2),(4) having different nonzero weights and (1),(3) having zero weight. By Lemma \ref{primitiveelementLEMMA} we have $\la (1)+(2)+(3)+(4) \ra_{T_{L,U}} = \la (2),(4),(1)+(3)\ra_{T_{L,U}}$. In addition, applying the substitution $x\mapsto x^{e_{12}e_{12}}$ to $(1)+(3)$, as $\vp_{e_{12}e_{12}}\vp_{gg}=0$ and $\vp_{e_{12}e_{12}}^2 = \vp_{e_{12}e_{12}}$, we get $(3)$ as a consequence of $(1)+(3)$ (hence also $(1)$, by subtraction). Therefore $\la (1)+(3)\ra_{T_{L,U}} = \la (1),(3)\ra_{T_{L,U}}$ and the result is proven.
\end{proof}

\begin{question}[\textbf{Minimality of the set of generators}]
\mbox{}
Is the set of generators of $\I^{L,U}(\M)$ appearing in Theorem \ref{(L,U)identitiesTHEOREM} minimal? We can reason similarly to the proof of Theorem \ref{minimalUgeneratorsTHEOREM} to find that (1),(2) are independent of themselves and of the rest, and that (3) is independent of (4); but the proof of (4) being independent of (3) now fails because we don't know a basis of $\FLU$. So either (1)-(3) or (1)-(4) is a minimal set of generators. It could even happen that (4) is $0$ in $\FLU$ (hence redundant), but this we are not able to prove nor disprove at the moment, for we have found neither an expression of (4) in terms of the elements of $\ILU$, nor an $(L,U)$-algebra for which (4) is not an $(L,U)$-identity. Note that (1)-(3) are not $0$ in $\FLU$, since there are nonzero evaluations for them from the $(L,U)$-algebra $M_2(F)\otimes_F\M$ as in the proof of \ref{minimalUgeneratorsTHEOREM} (respectively, with elements $x:=e_{12}\otimes g,y:=e_{21}\otimes g$, \, $x:=e_{12}\otimes g,y:=e_{21}\otimes e_{12}$, and $x:=e_{12}\otimes e_{12},y:=e_{21}\otimes e_{12}$).
\end{question}

\subsection{$L$-identities of $\M$}
\mbox{}

\begin{theorem}\label{differentialidentitiesTHEOREM}
For $k\geq2$ let $L$ be the Lie algebra of all derivations of $\M$, $c_{p,k}$ for $2\leq p< k$ be its $p$th Casimir element (given in \eqref{formula(Ca)}), $\lambda_{p,k}$ the eigenvalue of $c_{p,k}$ for $\Ad_k$ (given in Lemma \ref{eigenvalues}), and denote $z_{p,k}:=c_{p,k}-\lambda_{p,k}$ and $z:=e_{12}^3+e_{12}z_{2,k}+\cdots+e_{1k}z_{k,k}$. Let $\rho_{ab}$ be a valid assignment of preimages of $\vp_{ab}$ in $\Uop$ (as in Proposition \ref{assignmentsPROPOSITION}).
The $T_L$-ideal of differential identities of $\M$ is generated by the following $L$-polynomials:
\enum
\item {\Large $x^z$},
\item {\Large $[x^{\rho_{gg}},y^{\rho_{gg}+\rho_{e_{12}e_{12}}}] + x^{\rho_{e_{12}e_{12}}}y^{\rho_{e_{12}(e_{21}+e_{31})}}-y^{\rho_{e_{12}e_{12}}}x^{\rho_{e_{12}e_{21}}}$},\\
the $\rho_{e_{12}e_{31}}$ exponent appearing only if $k\geq3$.
\eenum
Moreover, $c_n^L(\M)=k^{2(n+1)}- (k^2-1)(n+1).$
\end{theorem}

\begin{proof}
The generators arise from Lemma \ref{Ldata2UdataLEMMA}(1), Formula \eqref{formula(I)} together with Corollary \ref{kerphigeneratorCOROLLARY} (which provide the first identity as generator of $\ILU$), and Theorem \ref{(L,U)identitiesTHEOREM} (which gives the generator of $\I^{L,U}(\M)$). The $L$-codimensions formula arises from Lemmas \ref{Ldata2UdataLEMMA}(2), \ref{Theorem Id^U K=2} (which gives $c_n^U(M_2(F))$) and \ref{Theorem Id^U K>2} (which gives $c_n^U(\M)$ for $k>2$).
\end{proof}

\medskip

\noindent From $c_n^L(\M)=k^{2(n+1)}- (k^2-1)(n+1)$ we immediately get:
\begin{corollary}
\mbox{}

\enum
\item $\exp^L(\M)=k^2.$
\item The generating function of the differential codimensions, $C_k^L(x):=\sum_{n=0}^\infty c_n^L(\M)x^n$, is rational:
\[C_k^L(x)=\frac{k^2}{1-k^2x}-\frac{k^2-1}{(1-x)^2}.\]
\eenum
\end{corollary}

\section{Growth of differential codimensions of $\M$}\label{Section:AlmostPolynomialGrowth}
\setcounter{subsection}{1}

In this section, we shall prove that $\V^L(\M)$, $k\geq 2$, is a variety of almost polynomial growth. To this end it is enough to work with $\V^{L,U}(\M)$ and $U$-polynomials, by Propositions \ref{equivalentvarietiesPROPOSITION} and \ref{Ldata2UdataLEMMA}(2).  Recall that notation $x^b$ with $x\in\FU$, $b\in\mathcal S$ is shorthand for the element $x^{\vp_{ab}}\in\FU$, for a fixed and elided first exponent index $a\in\mathcal S$ which we will not explicitly mention in most of the next results (see \ref{firstsuperindex}).

We start by proving some results on proper subvarieties of $\V^{L,U}(\M)$.

\begin{lemma} \label{Remark substitution}
For $a\in\mathcal S$,
\[(x^{ae_{12}}y^{ae_{21}})^{gg}\equiv \dfrac{1}{k}\left( \sum_{i=1}^{k-1}x^{ah_i}y^{ah_i}+ \sum_{i=1}^{k-2}x^{ah_i}y^{ah_{i+1}}\right) \ (\md \I^U(\M)).\]
\end{lemma}
\begin{proof}
Since for any $x\in \M$ and $a,b\in \mathcal{S}$, $x^{ab}=\mu^{x}_a b$ (see \ref{matrixunitsendomorphisms}), for all $x, y\in \M$ we have
\[(x^{ae_{12}}y^{ae_{21}})^{gg}= \mu^{x}_{a} \mu^{y}_{a} (e_{11})^{gg}=\dfrac{1}{k} \mu^{x}_{a} \mu^{y}_{a} \left(  \sum_{i=1}^{k-1} (k-i) h_i + g \right)^{gg}= \dfrac{1}{k}\mu^{x}_{a} \mu^{y}_{a}  g\]
	since $e_{11}=\dfrac{1}{k} \left(  \sum_{i=1}^{k-1} (k-i) h_i + g \right)$. On the other hand, since $g=\dfrac{1}{k} \left( \sum_{i=1}^{k-1}h_i^2+ \sum_{i=1}^{k-2} h_i h_{i+1}\right)$,  we have
\[\dfrac{1}{k}\left( \sum_{i=1}^{k-1}x^{ah_i}y^{ah_i}+ \sum_{i=1}^{k-2}x^{ah_i}y^{ah_{i+1}}\right)= \dfrac{1}{k} \mu^{x}_{a} \mu^{y}_{a}  \left( \sum_{i=1}^{k-1}h_i^2+ \sum_{i=1}^{k-2} h_i h_{i+1}\right)= \dfrac{1}{k} \mu^{x}_{a} \mu^{y}_{a}  g\]
	and the proof is complete.
\end{proof}

We prove now a characterization of the proper subvarieties of $\V^{L,U}(M_2(F))$.

\begin{proposition}\label{subvariety of M_2}
	Let $\textsc V=\V^{L,U}(A)$ be a subvariety of $\V^{L,U}(M_2(F))$. Then $\textsc V$ is a proper subvariety if and only if there exists $t\geq 2$ such that $x_1^{e_{12}}x_2^{e_{21}}\cdots x_{t-1}^{e_{12}}x_t^{e_{21}}\in \I^U(A)$.
\end{proposition}
\begin{proof}
	By the proof of Proposition \ref{Theorem Id^U K=2}, $x_1^{e_{12}}x_2^{e_{21}}\cdots x_{t-1}^{e_{12}}x_t^{e_{21}}\notin \I^U(M_2(F))$ for any $t\geq 1$, hence one implication is clear by Proposition \ref{equivalentvarietiesPROPOSITION}(3). Now, suppose that $\textsc V$ is a proper subvariety of $\V^{L,U}(M_2(F))$. Then by Proposition \ref{equivalentvarietiesPROPOSITION}(3) there exists a $U$-polynomial $f\in\I^U(A)$ such that $f\notin \I^U(M_2(F))$.
 We may assume that $f$ is a multilinear $U$-polynomial of degree $n$. Moreover, by Lemma \ref{reductiontocomponentsLEMMA} and Proposition \ref{P_{r,n-r,c}isomorphismPROPOSITION} we may suppose that $f\in P^U_{r,n-r}\cap \I^U(A)$ for some $0\leq r\leq n$. In order to simplify the notation let us identify $x_{r+i}^a$ with $y_i^a$,  $a\in \mathcal S$, for all $ 1\leq i \leq n-r$, so that variables $x_1,\ldots,x_r$ correspond to exponents $g$ and variables $y_1,\ldots,y_{n-r}$ correspond to exponents $a\in\mathcal S$. Since by Proposition \ref{equivalentvarietiesPROPOSITION}(3) we have $\I^U(M_2(F))\subseteq \I^U(A)$, we may reduce $f$ modulo $\I^U(M_2(F))$. Thus, by the proof of Proposition \ref{Theorem Id^U K=2} we may suppose that either
	$f= x_1^g \cdots x_n^g $
	if $r=n$,
	\begin{equation}
		 \label{polynomial n-1}
		f= \alpha_1 x_1^g \cdots x_{n-1}^g y_1^{h_1}+ \alpha_2 x_1^g \cdots x_{n-1}^g y_1^{e_{12}}+ \alpha_3 x_1^g \cdots x_{n-1}^g y_1^{e_{21}}
	\end{equation}
	with $\alpha_i\in F$, $1\leq i \leq 3$, not all zero, if $r=n-1$,
	\begin{align}
	\label{polynomial n-r even}
	f= & \alpha_1 x_1^g\cdots x_r^g y_{1}^{h_1} y_{2}^{e_{12}} y_{3}^{e_{21}}  y_{4}^{e_{12}}\cdots y_{n-r-1}^{e_{21}} y_{n-r}^{e_{12}}+\alpha_2 x_1^g\cdots x_r^g y_{1}^{h_1} y_{1}^{e_{21}} y_{2}^{e_{12}}  y_{4}^{e_{21}} \cdots y_{n-r-1}^{e_{12}} y_{n-r}^{e_{21}}\\
	\notag &+ \alpha_3 x_1^g\cdots x_r^g y_{1}^{e_{12}} y_{2}^{e_{21}}\cdots  y_{n-r-1}^{e_{12}} y_{n-r}^{e_{21}}
	+ \alpha_4 y_1^g\cdots y_r^g x_{1}^{e_{21}} x_{2}^{e_{12}} \cdots x_{n-r-1}^{e_{21}} x_{n-r}^{e_{12}}
	\end{align}
	with $\alpha_i\in F$, $1\leq i \leq 4$, not all zero, if $0\leq r\leq n-2$ and $n-r$ is even, or
	\begin{align} \label{polynomial n-r odd}
	f= & \alpha_1 x_1^g\cdots x_r^g y_{1}^{h_1} y_{2}^{e_{12}} y_{3}^{e_{21}} \cdots y_{n-r-1}^{e_{12}} y_{n-r}^{e_{21}}+ \alpha_2 x_1^g\cdots x_r^g y_{1}^{h_1} y_{1}^{e_{21}} y_{2}^{e_{12}} \cdots y_{n-r-1}^{e_{21}} y_{n-r}^{e_{12}}\\
	\notag &+ \alpha_3 x_1^g\cdots x_r^g y_{1}^{e_{12}} y_{2}^{e_{21}} y_{3}^{e_{12}}\cdots  y_{n-r-1}^{e_{21}} y_{n-r}^{e_{12}}+ \alpha_4
	x_1^g\cdots x_r^g y_{1}^{e_{21}} y_{2}^{e_{12}} y_{3}^{e_{21}} \cdots y_{n-r-1}^{e_{12}} y_{n-r}^{e_{21}}
	\end{align}
		with $\alpha_i\in F$, $1\leq i \leq 4$, not all zero, if $0\leq r\leq n-2$ and $n-r$ is odd.
	
	Suppose that $f$ is as in \eqref{polynomial n-r even}. Let us assume that $\alpha_1\neq 0$; the remaining cases follow from similar arguments.  Multiplying $f$ on the right by $y_{n-r+1}^{e_{21}}$, since $y_1^{e_{21}} y_2^{e_{21}}\in \I^U(M_2(F))$  by Lemma \ref{multiplicationtablegivesidentitiesLEMMA}, we get that
	\[
	f':= \alpha_1 x_1^g\cdots x_r^g y_{1}^{h_1} y_{2}^{e_{12}} y_{3}^{e_{21}}  y_{4}^{e_{12}}\cdots y_{n-r-1}^{e_{21}} y_{n-r}^{e_{12}} y_{n-r+1}^{e_{21}}
	+ \alpha_4 y_1^g\cdots y_r^g x_{1}^{e_{21}} x_{2}^{e_{12}} \cdots x_{n-r-1}^{e_{21}} x_{n-r}^{e_{12}}y_{n-r+1}^{e_{21}}\in \I^U(A).
	 \]
	Since $[x^g,y^{e_{21}}],\ y_1^{e_{21}}y_2^{h_1}+ y_1^{h_1}y_2^{e_{21}}, \ y_1^{e_{21}} y_2^{e_{21}} \in \I^U(M_2(F))$ (Lemma \ref{multiplicationtablegivesidentitiesLEMMA}) and $\alpha_1 \neq 0$, if we multiply $f'$ on the left by $y_{n-r+2}^{e_{21}}$ we get that
 $x_1^g\cdots x_r^g y_{1}^{h_1} y_{n-r+2}^{e_{21}} y_{2}^{e_{12}} y_{3}^{e_{21}}  y_{4}^{e_{12}}\cdots y_{n-r-1}^{e_{21}} y_{n-r}^{e_{12}} y_{n-r+1}^{e_{21}}\in \I^U(A)$.
  We substitute variables to reorder them (we \emph{rename} variables)  and get
	\[f'':=x_1^g\cdots x_r^g y_{1}^{h_1} y_{2}^{e_{21}} y_{3}^{e_{12}}  y_{4}^{e_{21}}\cdots y_{n-r+1}^{e_{12}} y_{n-r+2}^{e_{21}}\in \I^U(A).\]
	Now for all $1\leq i \leq r$ we substitute in $f''$ the variable $x_i$ with $y_{n-r+2+i}^{e_{12}}y_{n+2+i}^{e_{21}}$. By Lemma \ref{Remark substitution}  and after renaming variables we obtain that
 $y_1^{h_1}\cdots y_{2r+1}^{h_1}y_{2r+2}^{e_{21}}  y_{2r+3}^{e_{12}}  y_{2r+4}^{e_{21}}\cdots y_{n+r+1}^{e_{12}} y_{n+r+2}^{e_{21}} \in \I^U(A).$
 
	 Hence if we multiply the last $U$-polynomial on the left by $y_{n+r+3}^{h_1}$, by  renaming the variables we obtain that
	 $
	 y_1^{h_1}\cdots y_{2r+2}^{h_1}y_{2r+3}^{e_{21}}  y_{2r+4}^{e_{12}}  y_{2r+5}^{e_{21}}\cdots y_{n+r+2}^{e_{12}} y_{n+r+3}^{e_{21}} \in \I^U(A)
	 $. Now recall that $  y_1^{e_{12}} y_2^{e_{21}} + y_1^{e_{21}} y_2^{e_{12}}-y_1^{h_1} y_2^{h_1}, \ y_1^c y_2^c\in \I^U(M_2(F))$ for $c\in \{e_{12}, e_{21}\}$ (Lemma \ref{multiplicationtablegivesidentitiesLEMMA}). Thus it follows that
	 $y_1^{e_{21}}  y_{2}^{e_{12}} \cdots y_{n+r+2}^{e_{12}} y_{n+r+3}^{e_{21}} \in \I^U(A)$.
   Multiplying the last $U$-polynomial on the left by $y_{n+r+4}^{e_{12}}$ and renaming variables we get that
	 \[ y_1^{e_{12}}y_2^{e_{21}}\cdots y_{n+r+3}^{e_{12}}y_{n+r+4}^{e_{21}}\in \I^U(A) \]
	 and we are done. One deals in a similar way with case \eqref{polynomial n-r odd}. 
	
	Suppose now that $f= x_1^g \cdots x_n^g$. If we substitute variable $x_i$ with $y_{i}^{e_{12}}y_{i+n}^{e_{21}}$ for all $1\leq i \leq n$, then by Lemma \ref{Remark substitution}  and after renaming variables we get that $y_{1}^{h_1}y_{2}^{h_1}\cdots y_{2n-1}^{h_1}y_{2n}^{h_1}\in \I^U(A)$.
 Since  $y_1^{e_{12}} y_2^{e_{21}} + y_1^{e_{21}} y_2^{e_{12}}-y_1^{h_1} y_2^{h_1}, \ y_1^c y_2^c\in \I^U(M_2(F))$ for $c\in \{e_{12}, e_{21}\}$, it follows that
	$y_{1}^{e_{12}}y_{2}^{e_{21}}\cdots y_{2n-1}^{e_{12}}y_{2n}^{e_{21}}+ y_{1}^{e_{21}}y_{2}^{e_{12}}\cdots y_{2n-1}^{e_{21}}y_{2n}^{e_{12}}\in \I^U(A)$.
	 Thus we obtain a $U$-polynomial of the form \eqref{polynomial n-r even} with $r=0$ and we are done.
	
	Finally, suppose that  $f$ is as in \eqref{polynomial n-1}. If we substitute in $f$ the variable $x_i$ with $y_{i+1}^{e_{12}}y_{i+n}^{e_{21}}$ for all $1\leq i \leq n-1$ then, by Lemma \ref{Remark substitution}, we obtain that
	\[ \alpha_1 y_{2}^{h_1}y_{n+1}^{h_1} \cdots y_{n}^{h_1}y_{2n-1}^{h_1} y_1^{h_1}+ \alpha_2 y_{2}^{h_1}y_{n+1}^{h_1} \cdots  y_{n}^{h_1}y_{2n-1}^{h_1} y_1^{e_{12}}+ \alpha_3 y_{2}^{h_1}y_{n+1}^{h_1} \cdots  y_{n}^{h_1}y_{2n-1}^{h_1} y_1^{e_{21}}\in \I^U(A). \]
	Thus, as we have done above, by reducing modulo $\I^U(M_2(F))$ we obtain a $U$-polynomial of the form \eqref{polynomial n-r odd}  with $r=0$. Now the proof is complete.
\end{proof}

Next, we prove some technical lemmas that lead to a similar characterization for the proper subvarieties of $\V^{L,U}(\M)$ for $k\geq 3$.

\begin{lemma} \label{Lemma h_i}
	Let $A\in \V^{L,U}(\M)$ for $k\geq 3$. If $x_1^{h_i} \cdots x_s^{h_i}\in \I^U(A)$ for some $1\leq i \leq k-1$ and $s\geq 1$, then there exists $t \geq s$ such that $x_1^{a_1}\cdots x_t^{a_t}\in \I^U(A)$ for all $a_1, \dots, a_t\in\mathcal{S}$.
\end{lemma}
\begin{proof}
	Since $\I^U(\M)\subseteq \I^U(A)$ by Proposition \ref{equivalentvarietiesPROPOSITION}(3), by the proof of Proposition \ref{Theorem Id^U K>2} it is enough to prove that modulo $\I^U(\M)$ there exists $t\geq s$ such that
\[x_1^{h_j} \cdots x_t^{h_j}, \ x_1^{h_j} \cdots x_{t-1}^{h_j}x_t^{e_{jl}}, \ x_1^{h_{k-1}}\cdots x_{t-1}^{h_{k-1}}x_t^{e_{km}}, \ x_1^{h_1}\cdots x_{t-1}^{h_1}x_t^{h_2}\in \I^U(A)\]
	for all $1\leq j,m \leq k-1$, $1\leq l\leq k$, $l\neq j$.
	
	Suppose that $x_1^{h_i} \cdots x_s^{h_i}\in \I^U(A)$ with $1\leq i\leq k-2$. If we multiply  $x_1^{h_i} \cdots x_s^{h_i}$ on the right by $x_{s+1}^{e_{i+1,i}}x_{s+2}^{e_{i,i+1}}$, then since $x^{h_i}y^{e_{i+1,i}}+ x^{h_{i+1}}y^{e_{i+1,i}}, [x^{h_i}, y^{h_{i+1}}]\in \I^U(\M)$  by Lemma \ref{multiplicationtablegivesidentitiesLEMMA}, we get that
	\begin{equation}\label{Identita 1}
		x_1^{h_{i+1}} \cdots x_s^{h_{i+1}}x_{s+1}^{e_{i+1,i}}x_{s+2}^{e_{i,i+1}}\in \I^U(A).
	\end{equation}
	Moreover, since $x^{e_{i+2,i}}y^{h_i}- x^{h_{i+1}}y^{e_{i+2,i}}\in \I^U(\M)$  (Lemma \ref{multiplicationtablegivesidentitiesLEMMA}) , by multiplying $x_1^{h_i} \cdots x_s^{h_i}$ on the left by $x_{s+1}^{e_{i+2,i}}$, we get $	x_1^{h_{i+1}} \cdots x_s^{h_{i+1}}x_{s+1}^{e_{i+2,i}}\in \I^U(A)$. Thus it follows that
	\begin{equation}\label{Identita 2}
			x_1^{h_{i+1}} \cdots x_s^{h_{i+1}}x_{s+1}^{e_{i+2,i}}x_{s+2}^{e_{i,i+2}}\in \I^U(A).
	\end{equation}
	Hence since $x^{e_{i+1,i}}y^{e_{i,i+1}}+ x^{e_{i+2,i}}y^{i,i+2}-x^{h_{i+1}}y^{h_{i+1}}\in \I^U(\M)$ (Lemma \ref{multiplicationtablegivesidentitiesLEMMA}), by \eqref{Identita 1} and \eqref{Identita 2} it follows that $x_1^{h_{i+1}}\cdots x_{s+2}^{h_{i+1}}\in \I^U(A)$. Therefore, by iteration, we obtain that $x_1^{h_j}\cdots x_{s+2}^{h_{j}}\in \I^U(A)$ for all $i \leq j \leq k-1$.
	
	Now if $2\leq i \leq k-1$, then we multiply $x_1^{h_i} \cdots x_s^{h_i}$ on the left by $x_{s+1}^{e_{i-1,i}}$ and on the right by $x_{s+2}^{e_{i,i-1}}$. Thus since $x^{e_{i-1,i}}y^{h_i}-x^{h_{i-1}}y^{e_{i-i,i}}\in \I^U(\M)$  (Lemma \ref{multiplicationtablegivesidentitiesLEMMA}), we get
	\begin{equation} \label{Identita 3}
		x_1^{h_{i-1}} \cdots x_s^{h_{i-1}}x_{s+1}^{e_{i-1,i}}x_{s+2}^{e_{i,i-1}}\in \I^U(A).
	\end{equation}
	Now since $x^{h_i}y^{e_{i,i+1}}+x^{h_{i-1}}y^{e_{i,i+1}}\in \I^U(\M)$  (Lemma \ref{multiplicationtablegivesidentitiesLEMMA}), by multiplying $x_1^{h_i} \cdots x_s^{h_i}$ on the right by $x_{s+1}^{e_{i,i+1}}x_{s+2}^{e_{i+1,i}}$ we have that
	\begin{equation} \label{Identita 4}
			x_1^{h_{i-1}} \cdots x_s^{h_{i-1}}x_{s+1}^{e_{i,i+1}}x_{s+2}^{e_{i+1,i}}\in \I^U(A).
	\end{equation}
	Thus since $x^{e_{i-1,i}}y^{e_{i,i-1}}+ x^{e_{i,i+1}}y^{e_{i+1,i}}-x^{h_{i-1}}y^{h_{i-1}}\in \I^U(\M)$ (Lemma \ref{multiplicationtablegivesidentitiesLEMMA}), from \eqref{Identita 3} and \eqref{Identita 4} we obtain that  $x_1^{h_{i-1}}\cdots x_{s+2}^{h_{i-1}}\in \I^U(A)$. Hence, by iteration, we get that $x_1^{h_j}\cdots x_{s+2}^{h_{j}}\in \I^U(A)$ for all $1 \leq j \leq k-1$. Therefore it follows that $x_1^{h_j} \cdots x_{s+2}^{h_j}x_{s+3}^{e_{jl}}, \ x_1^{h_{k-1}}\cdots x_{s+2}^{h_{k-1}}x_{s+3}^{e_{km}}, \ x_1^{h_1}\cdots x_{s+2}^{h_1}x_{s+3}^{h_2}\in \I^U(A)$ for all $1\leq j,m \leq k-1$, $1\leq l\leq k$, $l\neq j$ and the proof is complete.
\end{proof}

\begin{lemma} \label{Lemma e_ij}
	Let $A\in \V^{L,U}(\M)$,  for $k\geq 3$. If $x_1^{h_i} \cdots x_{s-1}^{h_i}x_s^{e_{ij}}\in \I^U(A)$ for some $1\leq i \leq k-1$, $1\leq j \leq k$ and $s\geq 1$, then there exists $t \geq s$ such that $x_1^{a_1}\cdots x_t^{a_t}\in \I^U(A)$ for all $a_1, \dots, a_t\in\mathcal{S}$.
\end{lemma}
\begin{proof}
Recall that $\I^U(\M)\subseteq \I^U(A)$ (by Proposition \ref{equivalentvarietiesPROPOSITION}(3)). Then  since $x^a y^b- y^a x^b\in \I^U(\M)$ for all $a,b\in \mathcal{S}$ by Lemma \ref{multiplicationtablegivesidentitiesLEMMA}, modulo $\I^U(\M)$, we can always reorder the indices $i_l$, $1\leq l \leq m$, in any $U$-monomial of type $x_{i_1}^{a_1} \cdots x_{i_m}^{a_m}$ with $a_1, \dots, a_m\in \mathcal{S}$.
	
	Let now $1\leq i \leq k-1$ and $1\leq j \leq k$ such that $x_1^{h_i} \cdots x_{s-1}^{h_i}x_s^{e_{ij}}\in \I^U(A)$, $s\geq 1$. Then we may assume that $j=i+1$: indeed if $j\neq i+1$, then $x_1^{h_i} \cdots x_{s-1}^{h_i}x_s^{e_{ij}}x_{s+1}^{e_{j,i+1}}\in \I^U(A)$. Thus since  $x^{h_i}y^{e_{i,i+1}}-x^{e_{ij}}y^{e_{j,i+1}}\in \I^U(\M)$  (Lemma \ref{multiplicationtablegivesidentitiesLEMMA}), it follows that $x_1^{h_i} \cdots x_{s}^{h_i}x_{s+1}^{e_{i,i+1}}\in \I^U(A)$.
	
	Since  $x_1^{e_{i+1,i}}x_2^{h_i}+x_1^{h_{i}}x_2^{e_{i+1,i}}\in \I^U(\M)$  (Lemma \ref{multiplicationtablegivesidentitiesLEMMA}), if we multiply $x_1^{h_i} \cdots x_{s-1}^{h_i}x_s^{e_{i,i+1}}\in \I^U(A)$ on the left by $x_{s+1}^{e_{i+1,i}}$, then by reordering the index of the variables, as we may modulo $\I^U(\M)$, we get $x_1^{h_i} \cdots x_{s-1}^{h_i}x_s^{e_{i+1,i}}x_{s+1}^{e_{i,i+1}}\in \I^U(A)$. Moreover, we have also that  $x_1^{h_i} \cdots x_{s-1}^{h_i}x_s^{e_{i,i+1}}x_{s+1}^{e_{i+1,i}}\in \I^U(A)$. Thus since $x^{e_{i,i+1}}y^{e_{i+1,i}}+ x^{e_{i+1,i}}y^{e_{i,i+1}}-x^{h_{i}}y^{h_{i}}\in \I^U(\M)$  (Lemma \ref{multiplicationtablegivesidentitiesLEMMA}), we obtain that $x_1^{h_i} \cdots x_{s+1}^{h_i}\in \I^U(A)$ and by Lemma \ref{Lemma h_i} we are done.
\end{proof}

\begin{lemma} \label{Lemma h_1h_2}
	Let $A\in \V^{L,U}(\M)$,  for $k\geq 3$. If $x_1^{h_1} \cdots x_{s-1}^{h_1}x_s^{h_2}\in \I^U(A)$ for some $1\leq i \leq k-1$ and $s\geq 2$, then there exists $t \geq s$ such that $x_1^{a_1}\cdots x_t^{a_t}\in \I^U(A)$ for all $a_1, \dots, a_t\in\mathcal{S}$.
\end{lemma}
\begin{proof}
	If $x_1^{h_1} \cdots x_{s-1}^{h_1}x_s^{h_2}\in \I^U(A)$ for some $s\geq 2$, then also $x_1^{h_1} \cdots x_{s-1}^{h_1}x_s^{h_2}x_{s+1}^{e_{21}}\in \I^U(A)$. Since $[x^{h_1}, y^{h_2}],\allowbreak x^{h_1}y^{e_{21}}+ x^{h_2}y^{e_{21}}\in \I^U(\M)$  (by Lemma \ref{multiplicationtablegivesidentitiesLEMMA}) and $\I^U(\M)\subseteq \I^U(A)$ (by Proposition \ref{equivalentvarietiesPROPOSITION}(3)), it follows that $x_1^{h_2}\cdots x_s^{h_2}x_{s+1}^{e_{21}}\in \I^U(A)$ and by Lemma \ref{Lemma e_ij} we are done.
\end{proof}

We are now in position to prove the following characterization of the proper subvarieties of $\V^{L,U}(\M)$ for $k\geq 3$.
\begin{proposition}\label{subvariety of M_k}
	Let $\textsc V=\V^{L,U}(A)$ be a subvariety of $\V^{L,U}(\M)$, $k\geq 3$. Then $\textsc V$ is a proper subvariety if and only if there exists $t\geq 1$ such that $x_1^{a_1}\cdots x_t^{a_t}\in \I^U(A)$ for all $a_1, \dots, a_t\in\mathcal{S}$.
\end{proposition}
\begin{proof}
	By the proof of Proposition \ref{Theorem Id^U K>2}, $x_1^{a_1}\cdots x_t^{a_t}\notin \I^U(\M)$ for any $t\geq 1$ and any $a_1, \dots, a_t\in\mathcal{S}$, hence by Proposition \ref{equivalentvarietiesPROPOSITION}(3) one implication is clear. Let $\textsc V=\V^{L,U}(A)$ be a proper subvariety of $\V^{L,U}(\M)$, $k\geq 3$. Then by Proposition \ref{equivalentvarietiesPROPOSITION}(3)  there exists a $U$-polynomial $f\in\I^U(A)$ such that $f\notin \I^U(\M)$. We may assume that $f$ is a multilinear $U$-polynomial of degree $n$. Moreover, by Lemma \ref{reductiontocomponentsLEMMA} and Proposition \ref{P_{r,n-r,c}isomorphismPROPOSITION} we may suppose that $f\in P^U_{r,n-r}\cap \I^U(A)$ where $0\leq r\leq n$.
 In order to simplify the notation let us identify  $x_{r+i}^a$ with $y_i^a$,  $a\in \mathcal S$, for all $ 1\leq i \leq n-r$, so that variables $x_1,\ldots,x_r$ correspond to exponents $g$ and variables $y_1,\ldots,y_{n-r}$ correspond to exponents $a\in\mathcal S$. Since by Proposition \ref{equivalentvarietiesPROPOSITION}(3) we have $\I^U(M_k(F))\subseteq \I^U(A)$, we may reduce $f$ modulo $\I^U(M_k(F))$. Thus, by the proof of Proposition \ref{Theorem Id^U K>2} we may suppose that either
	$f=x_1^g\cdots x_n^g$
	 if $r=n$,
	 \[ f=\sum_{1\leq i \leq k-1} \alpha_i x_1^g \cdots x_{n-1}^g y_1^{h_i}+\sum_{\substack{1\leq j,l \leq k\\ j\neq l}} \beta_{jl} x_1^g \cdots x_{n-1}^g y_1^{e_{jl}} \]
	  with $\alpha_i,$ $\beta_{jl}$ not all zero, if $r=n-1$, or
	\begin{align*}
	f= & \sum_{1\leq i \leq k-1} \alpha_i x_1^g \cdots x_r^g y_1^{h_i} \cdots y_{n-r}^{h_i} + \beta x_1^g \cdots x_r^g y_1^{h_1} \cdots y_{n-r-1}^{h_1} y_{n-r}^{h_2}+ \sum_{\substack{1\leq i \leq k-1\\ 1\leq j \leq k\\ i\neq j}} \gamma_{ij} x_1^g \cdots x_r^g y_1^{h_i} \cdots y_{n-r-1}^{h_i} y_{n-r}^{e_{ij}}\\
	& +\sum_{1\leq l \leq k-1} \gamma_{kl}
	x_1^g \cdots x_r^g y_1^{h_{k-1}} \cdots y_{n-r-1}^{h_{k-1}} y_{n-r}^{e_{kl}}
	\end{align*}
	with $\alpha_i, \beta, \gamma_{lj}$ not all zero, if $0\leq r\leq n-1$.
	
	Suppose first that $r=0$, i.e.
\[f=  \sum_{1\leq i \leq k-1} \alpha_i y_1^{h_i} \cdots y_{n}^{h_i} + \beta  y_1^{h_1} \cdots y_{n-1}^{h_1} y_{n}^{h_2}+ \sum_{\substack{1\leq i \leq k-1\\ 1\leq j \leq k\\ i\neq j}} \gamma_{ij}  y_1^{h_i} \cdots y_{n-1}^{h_i} y_{n}^{e_{ij}}
	 +\sum_{1\leq l \leq k-1} \gamma_{kl}
	 y_1^{h_{k-1}} \cdots y_{n-1}^{h_{k-1}} y_{n}^{e_{kl}} ,\]
	with $\alpha_i, \beta, \gamma_{lj}$ not all zero. Let us assume that $\alpha_{k-1}\neq 0$. Since $y_1^{e_{1k}}y_2^{h_i}, \ y_1^{e_{jl}}y_2^{e_{k,k-1}}\in \I^U(\M)$ for all $1\leq i\leq k-2,$ $1\leq j\leq k$, $ 1\leq l \leq k-1$, $l \neq j$ by Lemma \ref{multiplicationtablegivesidentitiesLEMMA}, by multiplying $f$ on the left by $y_{n+1}^{e_{1k}}$ and on the right by $y_{n+2}^{e_{k,k-1}}$ we get
	$\alpha_{k-1}y_{n+1}^{e_{1k}} y_1^{h_{k-1}} \cdots y_{n}^{h_{k-1}}y_{n+2}^{e_{k,k-1}}+\gamma_{k-1,k} y_{n+1}^{e_{1k}}  y_1^{h_k-1} \cdots y_{n-1}^{h_k-1} y_{n}^{e_{k-1k}} y_{n+2}^{e_{k,k-1}}
	\in \I^U(A)$. Since $y_1^{e_{1k}}y_2^{h_{k-1}} + y_1^{h_1}y_2^{e_{1k}}, \ y_1^{e_{1k}} y_2^{e_{k-1,k}}\in \I^U(\M)$ (Lemma \ref{multiplicationtablegivesidentitiesLEMMA}) and $\alpha_{k-1}\neq 0$ and after renaming the variables, it follows that $y_1^{h_1}  \cdots y_n^{h_1}y_{n+1}^{e_{1k}}y_{n+2}^{e_{k,k-1}}\in \I^U(A)$.
 Moreover, since $y_1^{e_{1k}}y_2^{e_{k,k-1}}- y_1^{h_1}y_2^{e_{1,k-1}}\in \I^U(\M)$ for $k\geq 3$  (Lemma \ref{multiplicationtablegivesidentitiesLEMMA}), then $y_1^{h_1}\cdots y_{n+1}^{h_1}y_{n+2}^{e_{1,k-1}}\in \I^U(A)$ and by Lemma \ref{Lemma e_ij} we are done. Thus we may assume that $\alpha_{k-1}=0$.
	 Now if $k\geq 4$ and $\alpha_{k-2}\neq 0$, then we multiply $f$ on the left by $y_{n+1}^{e_{1,k-1}}$ and on the right by $y_{n+2}^{e_{k-1,k-2}}$ and we obtain that $\alpha_{k-2}y_{n+1}^{e_{1,k-1}} y_1^{h_{k-2}} \cdots y_{n}^{h_{k-2}}y_{n+2}^{e_{k-1,k-2}}+\gamma_{k-2,k-1} y_{n+1}^{e_{1,k-1}}  y_1^{h_k-2} \cdots y_{n-1}^{h_k-2} y_{n}^{e_{k-2,k-1}} y_{n+2}^{e_{k-1,k-2}}\in \I^U(A)$ since $y_1^{e_{1,k-1}}y_2^{h_i}, \ y_1^{e_{jl}}y_2^{e_{k-1,k-2}}\in \I^U(\M)$ for all $1\leq i\leq k-3,$ $1\leq j\leq k$, $ 1\leq l \leq k$, $l \neq k-1$, $j\neq l$ (Lemma \ref{multiplicationtablegivesidentitiesLEMMA}). As above since $y_1^{e_{1k-1}}y_2^{h_{k-2}} + y_1^{h_1}y_2^{e_{1k-1}}, \ y_1^{e_{1k-1}} y_2^{e_{k-2,k-1}}\in \I^U(\M)$  (Lemma \ref{multiplicationtablegivesidentitiesLEMMA}), and $\alpha_{k-2}\neq 0$,  after renaming variables we get that $ y_1^{h_1} \cdots y_{n+1}^{e_{1,k-1}}y_{n+2}^{e_{k-1,k-2}}\in \I^U(A)$. Moreover since $y_1^{e_{1,k-1}}y_2^{e_{k-1,k-2}}- y_1^{h_1}y_2^{e_{1,k-2}}\in \I^U(\M)$ for $k\geq 4$  (Lemma \ref{multiplicationtablegivesidentitiesLEMMA}), then $y_1^{h_1}\cdots y_{n+1}^{h_1}y_{n+2}^{e_{1,k-2}}\in \I^U(A)$ and again by Lemma \ref{Lemma e_ij} we are done. Then we may assume that $\alpha_{k-2}=0$. Since if $k\geq 5$ we can iterate the above procedure, we may assume that $\alpha_2=\cdots=\alpha_{k-1}=0$.
	
	Now assume that $\alpha_1\neq 0$. By multiplying $f$ on the left by $y_{n+1}^{e_{21}}$ and on the right by $y_{n+2}^{e_{12}}$ we get that $y_{n+1}^{e_{21}}y_1^{h_1} \cdots y_{n}^{h_1}y_{n+2}^{e_{12}}\in \I^U(A)$ because $y_1^{e_{21}}y_2^{h_i}, \ y_1^{h_2}y_2^{e_{12}}, \ y_1^{e_{lj}}y_2^{e_{12}}\in\I^U(\M)$ for all $2\leq i \leq k-1,$ $2\leq j \leq k$, $1\leq l \leq k-1$, $j\neq l$  (Lemma \ref{multiplicationtablegivesidentitiesLEMMA}). Hence since $y_1^{e_{21}}y_2^{h_1}-y_1^{h_1}y_2^{e_{21}}, y_1^{e_{21}}y_2^{e_{12}}+ y_1^{h_1}y_2^{h_2} \in \I^U(\M)$  (Lemma \ref{multiplicationtablegivesidentitiesLEMMA}),  after renaming variables it follows that $y_1^{h_1} \cdots y_{n+1}^{h_1}y_{n+2}^{h_2}\in \I^U(A)$ and by Lemma \ref{Lemma h_1h_2} we are done. So we may assume also that $\alpha_1=0$.
	
	Let us suppose then that $\beta\neq 0$. By multiplying $f$ on the left by $y_{n+1}^{e_{12}}$ and on the right by $y_{n+2}^{e_{23}}$, as above we obtain that $y_1^{h_1} \cdots y_{n+1}^{h_1}y_{n+2}^{e_{13}}\in \I^U(A)$ since $y_1^{e_{12}}y_2^{h_i}, \ y_1^{e_{lj}}y_2^{e_{23}}, \ y_1^{e_{12}}y_2^{h_1}+ y_1^{h_1}y_2^{e_{12}}, \ y_1^{e_{12}}y_2^{h_2}-y_1^{h_1}y_2^{e_{12}}, \ y_1^{e_{12}}y_2^{e_{lj}}, \ y_1^{e_{12}}y_2^{e_{23}}-y_1^{h_1}y_2^{e_{13}} \in \I^U(\M)$ for all $3\leq i \leq k-1$, $1\leq j,j \leq k$, $j\neq 2$, $l \neq j$  (Lemma \ref{multiplicationtablegivesidentitiesLEMMA}). Thus by Lemma \ref{Lemma e_ij} we are done. So we may assume that
\[f= \sum_{\substack{1\leq i \leq k-1\\ 1\leq j \leq k\\ i\neq j}} \gamma_{ij}  y_1^{h_i} \cdots y_{n-1}^{h_i} y_{n}^{e_{ij}}
	+\sum_{1\leq l \leq k-1} \gamma_{kl}
y_1^{h_{k-1}} \cdots y_{n-1}^{h_{k-1}} y_{n}^{e_{kl}}\]
	where at least one $\gamma_{ij}$, $1\leq i,j\leq k$, $i\neq j$, is not zero. Let $1\leq p,q\leq k$, $p \neq q$, such that $\gamma_{p,q}\neq 0$. Since $k\geq 3$, there exists $1\leq m\leq k$ such that $m\neq p,q$. Since $y_1^{e_{mp}}y_2^{h_i}, \ y_1^{e_{lj}}y_2^{e_{qp}}\in \I^U(\M)$, $1\leq i \leq k-1, \ 1\leq j,l \leq k, \ i\neq p, p-1, \ j \neq l,q$ (Lemma \ref{multiplicationtablegivesidentitiesLEMMA}), if we multiply $f$ on the left by $y_{n+1}^{e_{mp}}$ and on the right by $y_{n+2}^{e_{qp}}$, then we get that either
	\[ \gamma_{pq}y_{n+1}^{e_{mp}} y_1^{h_p} \cdots y_{n-1}^{h_p} y_{n}^{e_{pq}}y_{n+2}^{e_{qp}}\in \I^U(A) \]
	 in case $p\neq k, k-1$ or in case $(p,q)=(k-1,k)$, or
	\[ \gamma_{k-1,q}y_{n+1}^{e_{m p}} y_1^{h_{k-1}} \cdots y_{n-1}^{h_{k-1}} y_{n}^{e_{k-1,q}}y_{n+2}^{e_{q,p}}+  \gamma_{kq} y_{n+1}^{e_{m, p}} y_1^{h_{k-1}} \cdots y_{n-1}^{h_{k-1}} y_{n}^{e_{kq}}y_{n+2}^{e_{qp}}\in \I^U(A) \]
	in case $p=k-1$ and $q\neq k$ or in case $p= k$. Let us suppose first that $p\neq k$ and $m=k$. Notice that from $m=k$ it follows that $q\neq k$. Now since  $y_1^{e_{kp}}y_2^{h_p}+ y_1^{h_{k-1}}y_2^{e_{kp}}, \ y_1^{e_{k,k-1}}y_2^{e_{kq}} \in \I^U(\M)$  (Lemma \ref{multiplicationtablegivesidentitiesLEMMA}) and $\gamma_{pq}\neq 0$,  after renaming we have that $y_1^{h_{k-1}} \cdots y_{n-1}^{h_{k-1}} y_{n}^{e_{kp}} y_{n+1}^{e_{pq}}y_{n+2}^{e_{qp}}\in \I^U(A)$. Thus since $y_1^{e_{ki}}y_2^{e_{ij}}+y_1^{h_{k-1}}y_2^{e_{kj}}\in \I^U(\M)$ with $(i,j)=(p,q)$ or $(i,j)=(q,p)$  (Lemma \ref{multiplicationtablegivesidentitiesLEMMA}), it follows that $ y_1^{h_{k-1}} \cdots y_{n+1}^{h_{k-1}} y_{n+2}^{e_{k-1,p}}\in \I^U(A)$ and by Lemma \ref{Lemma e_ij} we are done. The cases $p,m\neq k$ and $p=k$ follow with similar computations. Hence for $r=0$ the proof is complete.
	
	Suppose now that $r=n$, i.e., $f=x_1^{g}\cdots x_n^{g}\in \I^U(A)$. If for all $1\leq i \leq n$ we substitute in $f$ the variable $x_i$ with $y_{i}^{e_{12}}y_{i+n}^{e_{21}}$, then by Lemma \ref{Remark substitution} we obtain a $U$-polynomial with exponents in $\{h_i\}_{1\leq i \leq k-1}$, and by reducing it modulo $\I^U(\M)$ (using the same strategy of Proposition \ref{Theorem Id^U K>2}'s proof) we get that
	\[ \sum_{1\leq i \leq k-1} \alpha_i  y_1^{h_i} \cdots y_{n+r}^{h_i} + \beta  y_1^{h_1} \cdots y_{n+r-1}^{h_1} y_{n+r}^{h_2}\in \I^U(A),\]
	with $ \alpha_i, \beta$ not all zero. Therefore, by repeating the same argument of the case $r=0$, we are done.
	
	Finally, let us assume that $1\leq r \leq n-1$. By Lemma \ref{Remark substitution}, if we substitute  $x_i$ with $y_{i+n-r}^{e_{12}}y_{i+n}^{e_{21}}$ for all $1\leq i \leq r$ we obtain a $U$-polynomial with exponents in $\mathcal{S}$. Thus by reducing it modulo $\I^U(\M)$ we get that
	\begin{align*}
	 &\sum_{1\leq i \leq k-1} \alpha'_i  y_1^{h_i} \cdots y_{n+r}^{h_i} + \beta'  y_1^{h_1} \cdots y_{n+r-1}^{h_1} y_{n+r}^{h_2} + \sum_{\substack{1\leq i \leq k-1\\ 1\leq j \leq k\\ i\neq j}} \gamma'_{ij} y_1^{h_i} \cdots y_{n+r-1}^{h_i} y_{n+r}^{e_{ij}}\\
	& +\sum_{1\leq l \leq k-1} \gamma'_{kl}
	y_1^{h_{k-1}} \cdots y_{n+r-1}^{h_{k-1}} y_{n+r}^{e_{kl}} \in \I^U(A)
	\end{align*}
	with $\alpha'_i, \beta', \gamma'_{lj}$ not all zero. Thus by repeating the same argument of the case $r=0$ the proof is complete.
\end{proof}

We are now in position to prove the main result of this section.

\begin{theorem}\label{AlmostPolynomialGrowthTHEOREM}
Let $L$ be the Lie algebra of all derivations of $\M$, $k\geq 2$. Then the variety $\V^L(\M)$ has almost polynomial growth for all $k\geq 2$.
\end{theorem}
\begin{proof}
	Let $\textsc V$ be a proper subvariety of $\V^L(\M)$, $k\geq 2$. By Propositions \ref{equivalentvarietiesPROPOSITION}(2) and \ref{Ldata2UdataLEMMA}(2) it suffices to show that $c_n^U(\textsc V)$ is polynomially bounded. Notice that as a consequence of Formula \eqref{formula(C)} it is enough to prove that there exists $N\geq 1$ such that $c^U_{r,n-r}(\textsc V)=0$ whenever $n-r\geq N$: indeed, if such $N$ exists, then we have that
\[c_n^U(\textsc V)=\sum_{r=0}^n \binom{n}{n-r} (k^2 -1)^{n-r}c^U_{r,n-r}(\textsc V)\leq k^2 \sum_{n-r< N} \binom{n}{n-r} (k^2 -1)^{n-r} \leq k^2(k^2 -1)^{N} \sum_{n-r< N} n^N \leq \alpha n^N,\]
	for $\alpha:=k^2(k^2 -1)^{N} N$. Thus it follows that $\V^L(\M)$ has almost polynomial growth for all $k\geq 2$.
	
	Suppose first that $k=2$. Then by Proposition \ref{subvariety of M_2} there exists $t\geq 2$ such that $x_1^{e_{12}}x_2^{e_{21}}\cdots x_{t-1}^{e_{12}}x_t^{e_{21}}\in \I^U(\textsc V)$. It follows that $x_1^{e_{12}}x_2^{e_{21}}\cdots x_{t-1}^{e_{12}}x_t^{e_{21}}x_{t+1}^{e_{12}}, \ x_1^{e_{21}}x_2^{e_{12}}\cdots x_{t-1}^{e_{21}} x_t^{e_{12}} x_{t+1}^{e_{21}}, \ x_1^{e_{21}}x_2^{e_{12}}\cdots x_{t+1}^{e_{21}}x_{t+2}^{e_{12}}\in \I^U(\textsc V)$. Therefore since $[x^g,y^g], [x^g,y^a], x^c y^{h_1}+ x^{h_1}y^c\in \I^U(\textsc V)$ for all $a\in \mathcal{S}$ and $c\in \{e_{12}, e_{21}\}$  (Lemma \ref{multiplicationtablegivesidentitiesLEMMA}), we have that $c^U_{r,n-r}(\textsc V)=0$ if $n-r\geq t+2$.
	
	Now let us assume that $k\geq 3$. By Proposition \ref{subvariety of M_k} there exists $t\geq 1$ such that $x_1^{a_1}\cdots x_t^{a_t}\in \I^U(A)$ for all $a_1, \dots, a_t\in\mathcal{S}$. Since $[x^g,y^g], [x^g,y^a]\in \I^U(\textsc V)$ for all $a\in \mathcal{S}$, it follows that $c^U_{r,n-r}(\textsc V)=0$ if $n-r\geq t$ and we are done.
\end{proof}

\section{$U$-cocharacter of $\M$}\label{Section:Ucocharacter}
\setcounter{subsection}{1}

In this section, we shall compute the $(n,r)$th  $U$-cocharacter of $\M$ for $k\geq 2$.  Recall that notation $x^b$ with $x\in\FU$, $b\in\mathcal S$ is shorthand for the element $x^{\vp_{ab}}\in\FU$, for a fixed and elided first exponent index $a\in\mathcal S$ which we will not explicitly mention in most of the next results (see \ref{firstsuperindex}). Similarly we write $P^U_{r,n-r}$ instead of $P^U_{r,n-r,a}$ (see \ref{codimensionsformula}).


We start by proving some technical lemmas which give us a lower bound for the multiplicities $m_{\lambda, \mu}$ of the $S_r \times S_{n-r}$-character 
\[\chi_{r,n-r}^U(\M)=\sum_{(\lambda,\mu)\vdash(r,n-r)} m_{\lambda,\mu} \ \chi_\lambda \otimes \chi_\mu,\]
of $P^U_{r,n-r}(\M)$, for $k\geq 2$ (see \ref{U-character and its decomposition}).
 To this end recall that any irreducible left $S_r\times S_{n-r}$-module $W_{\lambda,\mu}\subseteq P_{r,n-r}^U$ with character $\chi_\lambda \otimes \chi_\mu$ can be generated as an $S_r\times S_{n-r}$-module by an element of the form $e_{T_\lambda}e_{T_\mu}f$, for some $f\in W_{\lambda,\mu}$ and
some pair of Young tableaux $(T_\lambda,T_\mu)$ of shape $\lambda\vdash r$ and $\mu\vdash n-r$, respectively. Here $e_{T_{\nu}}=\sum_{\substack{\sigma\in R_{T_{\nu}} \\ \tau\in C_{T_{\nu}}}} (\sgn\tau)\sigma\tau$  stands for the symmetrizer corresponding to
some Young tableau $T_\nu$ of shape $\nu \vdash p$, where $R_{T_{\nu}}$ and $C_{T_{\nu}}$ are the subgroups of $S_p$ stabilizing the rows and columns of $T_{\nu}$, respectively.

\begin{lemma}
	\label{Lemma multiplicities k^2-1}
	If $\lambda=(n-1)$ and $\mu=(1)$, then $m_{\lambda,\mu}\geq k^2 -1$.
\end{lemma}
\begin{proof}
	Let us consider the following tableaux
\[T_{\lambda}=\begin{array}{|c|c|c|c|}\hline
	1 & 2 & \cdots &n-1 \\ \hline
	\end{array}\ ,\quad T_{\mu}=\begin{array}{|c|}\hline
	n \\ \hline
	\end{array} \ .\]
 Then, the $U$-polynomials
\[f_b(x,y)= \underbrace{x^g \cdots x^g}_{n-1} y^b, \ \ \ b\in \mathcal{S},\]
 are obtained from the symmetrizers corresponding to the pair of tableaux $(T_\lambda, T_\mu)$ by identifying all the elements in the row of $\lambda$. It is easily checked that $f_b(x,y)$, $b\in \mathcal{S}$, are not $U$-identities for $\M$.
  Moreover, such $U$-polynomials are linearly independent modulo $\I^U(\M)$. In fact, suppose that
\[\sum_{1\leq i\leq k-1}\alpha_i f_{h_i}(x,y) + \sum_{\substack{1\leq l,j \leq k\\  l\neq j}}\beta_{lj} f_{e_{lj}}(x,y) \equiv 0 \ (\md \ \I^U(\M)).\]
	The evaluation $x=g$ and  $y=\displaystyle \sum_{a\in\mathcal S} a$, gives
\[\sum_{1\leq i\leq  k-1}\alpha_i h_i+ \sum_{\substack{1\leq l,j \leq k\\  l\neq j}} \beta_{lj} e_{lj}=0.\]
	Thus $\alpha_i=\beta_{lj}=0$, $1 \leq i \leq k-1$, $1\leq l, j \leq k$, $l\neq j$. As a consequence the $U$-polynomials $f_b$, $b\in \mathcal{S}$, are linearly independent modulo $\I^U(\M)$. For each $b\in \mathcal{S}$ let $e_{\lambda,\mu, b}(x_1,\dots,x_n):=e_{T_{\lambda}}e_{T_\mu}(x_{1}^{g}\cdots x_{n-1}^{g}x_n^{b})$ be the complete linearization of $f_b(x,y)$; it follows that the $k^2-1$ $U$-polynomials $e_{\lambda,\mu, b}(x_1,\dots,x_n)$ are linearly independent modulo $\I^U(\M)$. This implies that $ m_{(n-1),(1)}\geq k^2 -1$. 
\end{proof}

\begin{lemma}
	\label{Lemma multiplicites k^2}
	If $\lambda=(r)$ and $\mu=(n-r)$, $0\leq r \leq n-2$, then $m_{\lambda,\mu}\geq k^2$.
\end{lemma}
\begin{proof}
	For $0\leq r \leq n-2$, let us consider the tableaux
\[T_{\lambda}=\begin{array}{|c|c|c|c|}\hline
	1 & 2 & \cdots &r \\ \hline
	\end{array}\ ,\quad T_{\mu}=\begin{array}{|c|c|c|c|}\hline
	r+1 & r+2 & \cdots &n \\ \hline
	\end{array}\ .\]
Suppose first $k=2$. We associate to the pair of tableaux $(T_\lambda,T_\mu)$ the following $U$-polynomials:
	\begin{align*}
f_1(x,y) &= \underbrace{x^g \cdots x^g}_{r} \ \underbrace{y^{h_1} y^{e_{12}} y^{e_{21}} y^{e_{12}} \cdots y^{e_{21}} y^{e_{12}}}_{n-r},\quad
f_2 (x,y)= \underbrace{x^g \cdots x^g}_{r} \ \underbrace{y^{h_1} y^{e_{21}} y^{e_{12}} y^{e_{21}} \cdots y^{e_{12}} y^{e_{21}}}_{n-r},\\
f_3(x,y)&= \underbrace{x^g \cdots x^g}_{r} \ \underbrace{y^{e_{12}} y^{e_{21}}\cdots y^{e_{12}} y^{e_{21}}}_{n-r},\quad
f_4(x,y)= \underbrace{x^g \cdots x^g}_{r} \ \underbrace{y^{e_{21}} y^{e_{12}} \cdots y^{e_{21}} y^{e_{12}}}_{n-r},
\end{align*}
if $n-r$ is even, or
\begin{align*}
f_1(x,y) &= \underbrace{x^g \cdots x^g}_{r} \ \underbrace{y^{h_1} y^{e_{12}} y^{e_{21}} \cdots y^{e_{12}} y^{e_{21}}}_{n-r},\quad
f_2 (x,y)= \underbrace{x^g \cdots x^g}_{r} \ \underbrace{y^{h_1} y^{e_{21}} y^{e_{12}} \cdots y^{e_{21}}y^{e_{12}}}_{n-r},\\
f_3(x,y)&= \underbrace{x^g \cdots x^g}_{r}  \ \underbrace{y^{e_{12}} y^{e_{21}} y^{e_{12}} \cdots y^{e_{21}} y^{e_{12}}}_{n-r},\quad
f_4(x,y)= \underbrace{x^g \cdots x^g}_{r} \ \underbrace{y^{e_{21}} y^{e_{12}} y^{e_{21}} \cdots y^{e_{12}} y^{e_{21}}}_{n-r},
\end{align*}
	if $n-r$ is odd. These $U$-polynomials are obtained from the symmetrizers corresponding to the pair of tableaux $(T_\lambda,T_\mu)$ by identifying all the elements in the row of $T_\lambda$ and $T_\mu$, respectively. Clearly they do not vanish on $M_2(F)$.
 Also, they are linearly independent modulo  $\I^U(M_2(F))$. In fact, suppose that
 \[ \alpha_1 f_1(x,y) + \alpha_2 f_2(x,y) +\alpha_3 f_3(x,y) +\alpha_4 f_4(x,y) \equiv 0 \ (\md \I^U(M_2(F))). \]
 Then if we evaluate $x=g$ and $y= h_1 +  e_{12}+ e_{21}$, we get either $\alpha_1 e_{12}-\alpha_2 e_{21}+ \alpha_3 e_{11}+ \alpha_4 e_{22}=0$ if $n-r$ is even, or   $\alpha_1 e_{11}-\alpha_2 e_{22}+ \alpha_3 e_{12}+ \alpha_4 e_{21}=0$ if $n-r$ is odd. Thus $\alpha_1=\alpha_2=\alpha_3=\alpha_4=0$ in both cases. This implies that $f_i(x,y)$, $1\leq i\leq 4$, are linearly independent modulo $\I^U(M_2(F))$. 
 As in \ref{Lemma multiplicities k^2-1}, this implies that $m_{(r), (n-r)}\geq 4$, for $0\leq r \leq n-2$.
	
	Now let $k>2$. Then we consider the following $U$-polynomial associated to the pair of tableaux $(T_\lambda,T_\mu)$:
	\begin{align*}
		f_{i,j}(x,y) &= \underbrace{x^g \cdots x^g}_{r} \ \underbrace{y^{h_i} \cdots y^{h_i} y^{e_{ij}}}_{n-r},\quad
		f_{k,l} (x,y)= \underbrace{x^g \cdots x^g}_{r} \ \underbrace{y^{h_{k-1}}  \cdots y^{h_{k-1}} y^{e_{kl}}}_{n-r},\\
		g_m(x,y) &=\underbrace{x^g \cdots x^g}_{r} \ \underbrace{y^{h_m} \cdots y^{h_m} }_{n-r},\quad
		p (x,y)= \underbrace{x^g \cdots x^g}_{r} \ \underbrace{y^{h_1}  \cdots y^{h_1} y^{h_2}}_{n-r},
	\end{align*}
	$1\leq i,j,l\leq k$, $i\neq j$, $l\neq k$, $1\leq m\leq k-1$. These $U$-polynomials are obtained from the symmetrizers corresponding to the pair of tableaux $(T_\lambda,T_\mu)$ by identifying all the elements in the row of $T_\lambda$ and $T_\mu$, respectively. Also, it is clear that $f_{i,j}(x,y),$ $g_m(x,y)$, $p(x,y)$, $1\leq i,j\leq k$, $i\neq j$, $1\leq m\leq k-1$, are not $U$-identities of $\M$. Next, we shall prove that they are linearly independent modulo $\I^U(\M)$. Suppose that
\[\sum_{\substack{1\leq i,j \leq k\\  i\neq j}}\alpha_{i,j} f_{i,j}(x,y) + \sum_{1\leq m \leq k-1} \beta_m g_m (x,y)+ \gamma p (x,y) \equiv 0 \ (\md \ \I^U(\M)).\]
	If we evaluate $x=g$ and $y=\displaystyle \sum_{a\in\mathcal S} a$, then we get
\[\sum_{\substack{1\leq i \leq k-1\\ 1\leq j \leq k\\  i\neq j}}\alpha_{i,j} e_{ij} + (-1)^{n-r-1}\sum_{1\leq j \leq k-1} \alpha_{kj} e_{kj}+\sum_{1\leq m \leq k-1} \beta_m h_m + \gamma e_{22}=0.\]
	Thus it follows that $\alpha_{i,j}=\beta_m=\gamma=0$, $1\leq i,j\leq k$, $i\neq j$, $1\leq m\leq k-1$. Hence the $U$-polynomials $f_{i,j}(x,y),$ $g_m(x,y)$, $p(x,y)$, $1\leq i,j\leq k$, $i\neq j$, $1\leq m\leq k-1$, are linearly independent modulo $\I^U(\M)$. Again, as in \ref{Lemma multiplicities k^2-1}, this implies that $m_{(r),(n-r)}\geq k^2$ for $0\leq r \leq n-2$.
\end{proof}

At this point, we can prove the main result of this section.

\begin{theorem}\label{UcocharacterTHEOREM}
Let $n\geq 1$ and $0\leq r\leq n$. The $(n,r)$th $U$-cocharacter of $\M$, $k\geq 2$, is 
\[\chi^U_{(n;r)}(\M)=\displaystyle\begin{cases}
	\displaystyle\binom{n}{k}\,\chi_{(n)}\otimes\chi_{\emptyset}  & \mbox{ if } r=n, \\
	\displaystyle\binom{n}{k}(k^2-1)^2\,\chi_{(n-1)}\otimes\chi_{(1)}  & \mbox{ if } r=n-1, \\
	\displaystyle\binom{n}{k}(k^2-1)^{n-r}k^2\,\chi_{(r)}\otimes\chi_{(n-r)} & \mbox{ if } n-2\geq r\geq 0.
	\end{cases}\]
\end{theorem}
\begin{proof} By formula \eqref{formula(chi)} we have $\chi^U_{(n;r)}(\M)=\sum_{(\lambda,\mu)\vdash(r,n-r)} \binom{n}{k}(k^2-1)^{n-r} m_{\lambda,\mu} \ \chi_\lambda \otimes \chi_\mu$, so it is enough to compute the multiplicities $m_{\lambda,\mu}$ of the $S_r \times S_{n-r}$ character $\displaystyle \chi_{r,n-r}^U(\M)$ of $P^U_{r,n-r}(\M)$. We show that
\[m_{\lambda,\mu}=\begin{cases}
	1  & \mbox{ if } r=n\text{ and }(\lambda,\mu)=((n), \emptyset), \\
	k^2 -1  & \mbox{ if } r=n-1\text{ and }(\lambda,\mu)=((n-1), (1)), \\
	k^2 & \mbox{ if } \ n-2\geq r\geq 0\text{ and }(\lambda,\mu)=((r), (n-r)),\\
	0 & \mbox{ otherwise}.
	\end{cases}\]
To this end we shall use that $ c^U_{r,n-r}(\M)=\deg \chi_{r,n-r}^U(\M)=  \sum_{(\lambda,\mu)\vdash(r,n-r)} m_{\lambda,\mu} \ d_\lambda d_\mu $, where $d_\lambda=\deg \chi_\lambda$ and $d_\mu=\deg \chi_\mu$.
	
First notice that if $r=n$, $\lambda=(n)$ and $\mu=\emptyset$, then $f(x_1,\dots,x_n)=x_{1}^{g} \cdots x_{n}^{g}$ is a $U$-polynomial associated to the pair of tableaux
\[T_{\lambda}=\begin{array}{|c|c|c|c|}\hline
	1 & 2 & \cdots &n \\ \hline
	\end{array}\, ,\, \quad T_{\mu}=\emptyset\]
	that is not an $U$-identity of $\M$. Thus $m_{(n),\emptyset}\geq 1$. By Propositions \ref{Theorem Id^U K=2} and \ref{Theorem Id^U K>2} we have that
	\[
	1=c_{n,0}(\M)\geq d_{(n)} d_{\emptyset}=1.
	 \]
	Hence  $m_{(n),\emptyset}=1$.
 Now assume that $r=n-1$. By Propositions \ref{Theorem Id^U K=2}, \ref{Theorem Id^U K>2} and Lemma \ref{Lemma multiplicities k^2-1} we get
	\[ k^2 - 1=c_{n-1,1}(\M)\geq (k^2 -1)d_{(n-1)} d_{(1)}=k^2-1. \]
	Thus $m_{(n-1),(1)}=k^2 -1$.
 Finally, suppose $\lambda=(r)$ and $\mu=(n-r)$, $0\leq r \leq n-2$. By Propositions \ref{Theorem Id^U K=2}, \ref{Theorem Id^U K>2} and Lemma \ref{Lemma multiplicites k^2}, it follows that
	 \[ k^2=c^U_{r,n-r}(\M)\geq k^2 d_{(r)}d_{(n-r)}=k^2.\] Thus $m_{(n-r),(r)}=k^2$
  and we are done.
\end{proof}

\textbf{Acknowledgements:} We are grateful to Professor Alberto Elduque for pointing us to absolute irreducibility for the proof of Theorem \ref{fullenvelopingalgebraTHEOREM} and for detecting a mistake in an earlier version of Lemma \ref{maximalideal}.

\end{document}